\documentclass[11pt,reqno]{amsart}
\usepackage{amsfonts}
\usepackage{amsmath}
\usepackage{amssymb}
\usepackage{graphicx}
\usepackage{enumerate}
\usepackage{multicol}
\usepackage{mathrsfs}
\usepackage[usenames,dvipsnames]{pstricks}
\usepackage{epsfig}
\usepackage{pst-grad} 
\usepackage{pst-plot} 
\usepackage[hmargin=3cm,vmargin=2.5cm]{geometry}
\usepackage{rotating}
\usepackage[utf8]{inputenc}
\usepackage{mathtools}
\usepackage{amscd}

\newtheorem{defi}{Definition}
\newtheorem{rem}{Remark} 
\newtheorem{prop}{Proposition}
\newtheorem{theorem}{Theorem}

\newtheorem{cor}{Corollary}
\newtheorem{lemma}{Lemma}

\DeclareMathOperator{\n}{\mathfrak{n}}

\DeclareMathOperator{\so}{\mathfrak{so}}

\DeclareMathOperator{\bi}{\bf{i}}
\DeclareMathOperator{\bj}{\bf{j}}
\DeclareMathOperator{\bk}{\bf{k}}

\DeclareMathOperator{\wP}{\widetilde{P}}

\DeclareMathOperator{\wJ}{\widetilde{J}}
\DeclareMathOperator{\wR}{\widetilde{R}}
\DeclareMathOperator{\wE}{\widetilde{E}}
\DeclareMathOperator{\wF}{\widetilde{F}}
\DeclareMathOperator{\wG}{\widetilde{G}}
\DeclareMathOperator{\wbi}{\tilde{\bf{i}}}
\DeclareMathOperator{\wbj}{\tilde{\bf{j}}}
\DeclareMathOperator{\wbk}{\tilde{\bf{k}}}

\DeclareMathOperator{\End}{End}

\DeclareMathOperator{\spn}{span}
\DeclareMathOperator{\la}{\langle}
\DeclareMathOperator{\ra}{\rangle}
\DeclareMathOperator{\ad}{ad}
\DeclareMathOperator{\Id}{Id}
\DeclareMathOperator{\Cl}{\rm{Cl}}
\DeclareMathOperator{\GL}{\rm{GL}}

\begin{document}
\title{Complete classification of $H$-type algebras: I}
\author[
Kenro Furutani, Irina Markina]
{
Kenro Furutani, Irina Markina}

\thanks{
The first author was partially
supported by the NCTS at National Taiwan University, Taipei,
and both of authors were partially supported by ISP project 239033/F20 and SPIRE 710022, University of Bergen, Norway}



\address{K.~Furutani:  Department of Mathematics, Faculty of Science 
and Technology, Science University of Tokyo, 2641 Yamazaki, Noda, Chiba (278-8510), Japan}
\email{furutani\_kenro@ma.noda.tus.ac.jp}

\address{I.~Markina: Department of Mathematics, University of Bergen, P.O.~Box 7803,
Bergen N-5020, Norway}
\email{irina.markina@uib.no}

\begin{abstract}
Let $\mathscr N$ be a 2-step nilpotent Lie algebra endowed with
non-degenerate scalar product 
$\langle.\,,.\rangle$ 
and let $\mathscr N=V\oplus_{\perp}Z$, 
where $Z$ is the centre of the Lie algebra 
and $V$ its orthogonal complement with respect to the scalar product. 
We study the classification of the Lie algebras for which the space $V$ arises as a representation space of a Clifford algebra $\Cl(\mathbb R^{r,s})$ and the representation map $J\colon \Cl(\mathbb R^{r,s})\to\End(V)$ is related to the Lie algebra structure by $\langle J_zv,w\rangle=\langle z,[v,w]\rangle$ for all $z\in \mathbb R^{r,s}$ and $v,w\in V$. The classification is based on the range of parameters $r$ and $s$ and is completed for the Clifford modules $V$, having minimal possible dimension, that are not necessary irreducible. We find the necessary condition for the existence of a Lie algebra isomorphism according to the range of integer parameters $0\leq r,s<\infty$. We present the constructive proof for the isomorphism map for isomorphic Lie algebras and defined the class of non-isomorphic Lie algebras. 
\end{abstract}
\maketitle


\section{Introduction}


We are studying one special type of 2-step nilpotent Lie algebras. In
the work~\cite{Met} M\'etivier introduced 2-step real nilpotent Lie
algebras $\n=V\oplus Z$ with the center $Z$ such that the adjoint map
$\ad_x\colon\n\to Z$ is surjective for any $x\in V$, where $V$ is the
prescribed complement to the center. Equivalently, the bracket defines a vector valued anti-symmetric form $[.\,,.]\colon V\times V\to Z$,
such that anti-symmetric real valued bilinear form $B(x,y)=\omega([x,y])$ is non-degenerate on $V$ for all $\omega\in Z^*$, $\omega\neq 0$. Particularly, it immediately implies that the space $V$ is even dimensional and $[V,V]=Z$ since $Z=\ad_x(\n)\subseteq[\n,\n]$ for $x\in V$. These Lie algebras were introduced in order to study the analytic hypoellipticity and were called Lie algebras {\it satisfying hypothesis H}. The Lie algebras were also studied in~\cite[Definition 1.3]{E} under the name {\it non-singular}, in~\cite{LT,MS} as {\it Lie algebras of M\'etivier group} or in~\cite{KT} as {\it fat} algebras since they are source of fat distributions. 

Let us observe that if a 2-step nilpotent Lie algebra $\n$ carries a positive definite product  $\langle .\,,.\rangle$ on it, and the map $J_z\colon V\to V$ is defined by 
\begin{equation}\label{eq:0}
\la J_zx,y\ra=\la z,[x,y]\ra=\la z,\ad_x(y)\ra,
\end{equation} 
then $J_z$ is a non-singular linear map for any non-zero $z\in Z$ if and only if the Lie algebra $\n$ is non-singular. The presence of a positive definite product is not restrictive at all, because Eberlein showed in~\cite{Eber03} that any 2-step nilpotent Lie algebra is isomorphic to a standard metric form $(\mathcal N, \la.\,,.\ra_{\mathcal N})$, where $\mathcal N=\mathbb R^n\oplus W$, with $W\subset \so(n)$ and the positive definite product defined by $\la.\,,.\ra_{\mathcal N}=\la.\,,.\ra_{\mathbb R^n}+\la.\,,.\ra_{\so(n)}$. 
Thus any 2-step nilpotent Lie algebra can be considered as a metric Lie algebra with the scalar product defined as above. 

We are interested in those non-singular 2-step nilpotent Lie algebras, for which the map $J\colon Z\to\End(V)$ is a representation of a Clifford algebra $\Cl(\mathbb R^{r,s})$. The map $J$ and the Lie brackets are related by~\eqref{eq:0} by making use a sign indefinite non-degenerate scalar product. The corresponding Lie algebras, which we denote by $\mathcal N_{r,s}(V)$, received the name {\it pseudo $H$-type Lie algebras}. For the Clifford algebras $\Cl(\mathbb R^{r,0})$, generated by the Euclidean space $\mathbb R^r$, the $H$-type algebras $\mathcal N_{r,0}(V)$ were introduced by Kaplan~\cite{Ka1} and attracted a lot of attention~\cite{BTV,CD,CDKR,Ka2,Ka3,R1,R2}. The Lie algebras $\mathcal N_{r,0}(V)$ is a typical example of a standard metric form. For the Clifford algebras $\Cl(\mathbb R^{r,s})$, generated by a sign indefinite non-degenerate scalar product space $\mathbb R^{r,s}$, the pseudo $H$-type Lie algebras $\mathcal N_{r,s}(V)$ were introduced in~\cite{Ci}, and studied in~\cite{CP,FM,FMV,GKM,Gu}. 

We study the isomorphism properties between the Lie
algebras~$\mathcal N_{r,s}(V)$. We show that the Lie
algebras $\mathcal N_{r,s}(V)$ 
can not be isomorphic to $\mathcal N_{u,t}(V)$ unless $r=t$ and $s=u$ or $r=u$
and $s=t$. The present paper is the first part of the complete classification, where we concentrate on the classification of the Lie algebras based on the Clifford modules of minimal possible dimensions (which are not necessarily irreducible), admitting a scalar product making the representation map $J_z$ skew symmetric. Then, by making use the Atiyah-Bott periodicity for underlying Clifford algebras we extend the study to an arbitrary dimension pseudo $H$-type algebras. We also show that the Lie algebras based on the non-equivalent irreducible Clifford modules are isomorphic.
We stress, that the isomorphic relation between the Clifford algebras and the associated  pseudo $H$-type Lie algebras is not functorial. In some cases
the isomorphic Clifford algebras lead to isomorphic Lie algebras, in
other cases not. 

Apart from being motivated by itself interesting mathematical question
of the classification of Lie algebras, we want to mention here
possible applications in other areas of mathematics. It was
shown~\cite{CD,FM} that the pseudo $H$-type Lie algebras
admit the integer structure constants that in its turn, according to
the Mal\'cev theorem~\cite{Malcev}, guarantees the existence of
lattices on the corresponding Lie groups. The factorization of pseudo
$H$-type Lie groups by  lattices gives a vast of new examples of
nilmanifolds, which type strongly depend on the classification of
pseudo $H$-type algebras~\cite{Bat,Cheeger,CP,Gu}. 
Nilmanifolds are related to the Grushin type differential operators
descending from elliptic and sub-elliptic type operators on the
corresponding pseudo $H$-type Lie groups. This kind of nilmanifolds
allows precise construction of the spectral zeta function for the
Grusin operator,~\cite{BFI,BFI1} 
and gives new examples of iso-spectral but
 non-diffeomorphic nilmanifolds~(\cite{BFI2}).

Recently it was noticed that Tanaka prolongations of some pseudo
$H$-type Lie algebras coincide with Tanaka prolongations of simple Lie
algebras, factorized by parabolic subalgebras. It shows a close
relation between the classification of pseudo $H$-type algebras and
the theory of simple Lie algebras. 
The fact that Clifford algebras are pretty much useful in the
orthogonal design, 
signal processing, space-time block coding, 
or computer vision, is well known~\cite{Citko,GW,Ja,TJC}. 
The structure of pseudo $H$-type algebras allows a new construction of
orthogonal 
designs and possible wireless communications, as was shown in~\cite{FMV}. 

The article is organized in the following way. After the introduction
we give necessary definitions and notations in Section~\ref{sec:aux},
including the notion of admissible module and relation between
Clifford algebras and pseudo $H$-type Lie algebras. We also describe
the scheme of classification, 
that includes 4 steps. 
In the rest of  sections we realise 3 steps of the classification.


\section{Clifford algebras, modules, and pseudo $H$-type Lie algebras}\label{sec:aux}



\subsection{Clifford algebras and representations}


We use the notation $\mathbb{R}^{r,s}$ for the space
$\mathbb{R}^{r+s}$ equipped 
with the non-degenerate symmetric
bilinear form
\[
\la x,y\ra_{r,s}=\sum\limits_{i=1}^rx_iy_i-\sum\limits_{j=1}^{s}x_{r+j}y_{r+j},\quad x,y\in\mathbb{R}^{r,s}.
\]
An orthonormal basis we denote by $\{z_1,\ldots,z_{r+s}\}$. Thus 
$$
\la z_i,z_j\ra_{r,s}=\epsilon_{i}(r,s)\delta_{i,j},\quad
\epsilon_{i}(r,s)=
\begin{cases}
1,\quad & \text{if}\quad  i=1,\ldots,r,
\\
-1,\quad & \text{if}\quad  i=r+1,\ldots,r+s,
\end{cases}
$$ 
where $\delta_{i,j}$ is the Kronecker symbol.
By $\Cl_{r,s}$ we denote the Clifford algebra generated by
$\mathbb{R}^{r,s}$, that is, the quotient algebra
of the tensor algebra 
\[
\mathcal{T}(\mathbb{R}^{r+s})=\mathbb{R}\oplus\left(\mathbb{R}^{r+s}\right)\oplus 
\left(\stackrel{2}\otimes\mathbb{R}^{r+s}\right)\oplus \left(\stackrel{3}\otimes\mathbb{R}^{r+s}\right)\oplus\cdots, 
\]
divided by the two-sided
ideal $I_{r,s}$ generated by the elements of the form
$x\otimes x+\la x,x\ra_{r,s}$, $x\in\mathbb{R}^{r+s}$.
The explicit determination of the Clifford algebras 
is given in~\cite{ABS} and they are isomorphic 
to matrix algebras presented~\cite{LawMich}. 
We mention in~\eqref{perCl} useful isomorphisms of Clifford algebras, related to
8-periodicity, 
established in~\cite{ABS} 
and $(4-4)$-periodicity, see~\cite{LawMich}. 
To denote isomorphic objects we use the symbol ``$\cong$".
\begin{equation}\label{perCl}
\begin{array}{lclcll}
&\Cl_{r,s}\otimes\Cl_{0,8}&\cong &\Cl_{r,s+8}&\cong &\Cl_{r,s}\otimes\mathbb R(16),
\\
&\Cl_{r,s}\otimes\Cl_{8,0}&\cong &\Cl_{r+8,s}&\cong &\Cl_{r,s}\otimes\mathbb R(16),
\\
&\Cl_{r,s}\otimes\Cl_{4,4}&\cong &\Cl_{r+4,s+4}&\cong &\Cl_{r,s}\otimes\mathbb R(16).
\end{array}
\end{equation}
An algebra homomorphism $\widehat J\colon \Cl_{r,s}\to\End(U)$ is
called representation map and the vector space $U$ is said to be the
representation space. The representation space $U$ becomes Clifford
$\Cl_{r,s}$-module, where the multiplication is defined 
by $\phi u=\widehat J_{\phi}u$, $u\in U$, $\phi\in \Cl_{r,s}$. 
It is enough to define a linear map $J\colon \mathbb
R^{r,s}\to\End(U)$, satisfying $J^2_z=-\la z,z\ra_{r,s}\Id_{U}$ for an
arbitrary $z\in\mathbb R^{r,s}$. 
Then $J$ can be uniquely extended to the representation $\widehat J$ 
by the universal property, see, for instance~\cite{Hu,Lam,LawMich}. 


\subsection{Admissible modules}\label{sec:adm_mod}


Let $U$ be a Clifford $\Cl_{r,s}$-module.
We call the module $U$ {\it admissible}, 
if there is a non-degenerate symmetric bilinear 
form $\la.\,,.\ra_{U}$ on $U$ such that the representation map $J$ satisfies the following condition:
\begin{equation}\label{AdMo2}
\la J_{z}x,y\ra_U+\la x,J_{z}y\ra_{U}=0\quad\text{for all}\quad z\in\mathbb{R}^{r,s},\ \ x,y\in U.
\end{equation}
We say that the map $J_z$ is skew symmetric with respect 
to the bilinear symmetric form $\la.\,,.\ra_{U}$ and 
write $U=(U,\la.\,,.\ra_{U})$ for an admissible module. 
If $(U,\la.\,,.\ra_{U})$ is an admissible module with a non-degenerate
scalar product, then it decomposes into the 
orthogonal sum of minimal dimensional admissible modules~\cite{FM}, since 
the orthogonal complement to an admissible submodule is an admissible module.

If $U$ is a $\Cl_{r,0}$-module, then there 
always exists a positive definite scalar product 
$\la.\,,.\ra_{U}$ such that $U$ becomes an admissible module. 
Particularly, any irreducible module is an admissible 
with respect to some positive definite scalar product. 
It allowed to Kaplan to introduce $H$-type Lie algebras in~\cite{Ka1}.  

If $s>0$, and $(U,\la.\,,.\ra_{U})$ is an admissible
$\Cl_{r,s}$-module, 
then the scalar product space $(U,\la.\,,.\ra_{U})$ has to be a
neutral space~\cite{Ci}, 
that is an even dimensional space, where the bilinear symmetric form
has equal number of positive and negative eigenvalues.
In this case an irreducible module need not be admissible. 

Recall that the Clifford algebras $\Cl_{r,s}$ with $r-s\neq 3(mod\,\, 4)$
admit only one irreducible module up to equivalence. 
Some of irreducible modules $V$ can be supplied with a scalar
product with the property (\ref{AdMo2})
and becomes an admissible module. In other cases 
the direct sum $V\oplus V$ must be taken in order to 
define the scalar product, see~\cite{Ci}. 
In both cases we call the obtained admissible module 
{\it minimal admissible module}. 
Thus, for the Clifford algebras $\Cl_{r,s}$ with $r-s\neq 3(mod\,\,4)$ 
the minimal admissible module is either $(V,\langle.\,,.\rangle_{V})$
or $(V\oplus V,\langle.\,,.\rangle_{V\oplus V})$, where $V$ is an
irreducible module. 
We will denote a minimal admissible module of the Clifford algebra $\Cl_{r,s}$ by $V^{r,s}$.

We clarify now the structure of minimal admissible modules for
$\Cl_{r,s}$ with $r-s=3(mod\,\, 4)$. 
In this case, there are two non-equivalent irreducible modules. 
Let $\{z_1,\ldots,z_{r+s}\}$ be an orthonormal basis of 
$\mathbb{R}^{r,s}$ 
and $\{J_{z_1},\ldots, J_{z_{r+s}}\}$ the corresponding representation
maps. 
The product $\Omega^{r,s}=\prod_{j=1}^{r+s}J_{z_j}$ is called the
volume form. In the case of $r-s=3(mod\,\,4)$, 
it belongs to the center of the Clifford algebra $\Cl_{r,s}$ 
and $(\Omega^{r,s})^2=\Id$. 
Two non-equivalent irreducible
modules are distinguished 
by the action of $\Omega^{r,s}$. We denote by $V_+$ the irreducible
module, where the volume form acts as the identity operator and 
by $V_-$ the non-equivalent irreducible $\Cl_{r,s}$-module, 
where the volume form
$\Omega^{r,s}$ 
acts as the minus identity operator. 
If non of irreducible modules is admissible, 
then the minimal admissible module is one of the following forms 
$V_+\oplus V_+$, $V_-\oplus V_-$ or $V_+\oplus V_-$. 
A choice of a
possible form 
depends on the value of index $s$ and it is explained in
Proposition~\ref{prop:mod}. 
The summary of possible structures of minimal admissible modules for
all the cases 
is given in Table~\ref{ir}. 
\begin{table}[h]
\center\caption{Structure of possible minimal admissible modules $V^{r,s}$}
\begin{tabular}{|c||c|}
\hline
$r+s\neq 3(mod\,\,4)$&$r+s= 3(mod\,\,4)$
\\
\hline\hline
$
V\quad\text{or}\quad V\oplus V
$
&$
\begin{array}{c|c||c}
s\ \ \text{is even} & s\ \ \text{is even} & s\ \ \text{is odd} 
\\
\hline
\quad V_+\ \text{or}\ V_-\ \ &\quad V_+\oplus V_+\ \text{or}\ V_-\oplus V_-\ \ &\quad V_+\oplus V_-
\end{array}
$
\\
\hline
\end{tabular}\label{ir}
\end{table}

\begin{prop}\label{prop:mod}
Let $\Cl_{r,s}$ be a Clifford algebra with $r-s=3(mod\,\, 4)$. The following cases are possible.
\begin{itemize}
\item[1.]{If $s$ is odd, then an irreducible module can not be admissible. 
The minimal admissible module is unique, up to an isomorphism, and has the form $V^{r,s}=V_+\oplus V_-$;}
\item[2.]{If $s$ is even and if the irreducible module $V_+$ is
    admissible, then $V_-$ is also admissible 
and vice versa;}
\item[3.]{If $s$ is even and if one of irreducible modules is not admissible, then the other one is neither admissible. The minimal admissible module takes one of the forms: $V^{r,s}=V_+\oplus V_+$ or $V^{r,s}=V_-\oplus V_-$.}
\end{itemize}
\end{prop}

\begin{proof} 
To show the first claim we assume that $(V_+,\langle.\,, .\rangle_{V_+})$ is admissible.  Then 
\begin{equation}\label{eq:isom}
\langle x,x\rangle_{V_+}=\langle \Omega^{r,s} (x), \Omega^{r,s} (x)\rangle_{V_+}=\prod_{i=1}^{r+s}\langle z_{i}, z_{i}\rangle_{r,s}\langle x,x\rangle_{V_+}=(-1)^s\langle x,x\rangle_{V_+}
\end{equation}
for any $x\in V_+$.
This shows that all the vectors $x\in V_+$ are null vectors and the scalar product $\langle.\,, .\rangle_{V_+}$ is degenerate. Thus the irreducible module $V_+$ can not be supplied with non-degenerate bilinear symmetric form, such that the maps $J_{z}$ satisfies~\eqref{AdMo2}. Similar arguments are valid for $V_-$. Thus if 
$(V^{r,s},\langle.\,, .\rangle_{V^{r,s}})$ is a minimal admissible
module, then $V^{r,s}$ has to contain 
both of $V_{\pm}$. 

The second statement is obvious. Before starting to prove the last
statement, we note that if $r-s=3(mod\,\, 4)$, then $r+s=2s+3(mod\,\, 4)$ 
is always odd and $\frac{r+s-1}{2}=s+1(mod\,\, 2)$ is also odd in the
case when $s$ is even. We assume now that non of two non-equivalent
irreducible modules is admissible 
and we consider a minimal admissible module $(V^{r,s},\langle.\,,
.\rangle_{V^{r,s}})$. Then the volume form is an isometry
by~\eqref{eq:isom} and a symmetric operator 
because of 
$$
\langle \Omega^{r,s} (x),y\rangle_{V^{r,s}}=(-1)^{r+s}\langle  x, J_{z_{r+s}}\ldots J_{z_1} y\rangle_{V^{r,s}}=(-1)^{r+s}(-1)^{\frac{r+s-1}{2}}\langle x,\Omega^{r,s}(y)\rangle_{V^{r,s}}.
$$
Thus, if $V^{r,s}$ contains two eigenspaces 
$V_+$ and $V_-$ of $\Omega^{r,s}$, then $V_+$ and $V_-$ have 
to be orthogonal non-degenerate subspaces of 
$(V^{r,s},\langle.\,, .\rangle_{V^{r,s}})$ and therefore admissible
modules. 
This contradicts to the assumption that non of irreducible modules is admissible.
\end{proof}

In Table~\ref{t:dim} we give the dimensions of minimal admissible modules $V^{r,s}$, $r,s\leq 8$. By the black colour we denote the dimensions of minimal admissible modules, that are also irreducible Clifford modules. The red colour is used for the minimal admissible modules which are direct product of two irreducible Clifford modules. The notation $N_{\times 2}$ means that there are two minimal admissible modules.

\begin{table}[h]
\center\caption{Dimensions of minimal admissible modules}
\begin{tabular}{|c||c|c|c|c|c|c|c|c|c|}
\hline
${\text{\small 8}} $&$ {\text{\small{16}}}$&${\text{\small 32}}$&${\text{\small{64}}}
$&${\text{\small{64}$_{\times 2}$}}$&${\text{\small{128}}}$&${\text{\small{128}}}$&${\text{{\small{128}}}}$&$
{\text{{\small{128}$_{\times 2}$}}} $&${\text{\small{256}}}$
\\
\hline
${\text{\small 7}}$ &$ {\text{\small{16}}}$&${\text{\small{32}}}$&$
{\text{\small{\color{red}{64}}}} $&${\text{\small{64}}}
$&${\text{{\color{red}\small{128}}}}$&${\text{\small{{\color{red}{128}}}}}  $&${\text{\color{red}{\small{128}}}}$&$ {\text{\small{128}}} $&${\text{\small{256}}}$
\\
\hline
${\text{\small 6}}$ &${\text{\small{16}}}$&${\text{\small{16}$_{\times 2}$}}$&${\text{\small{32}}}$&${\text{\small{32}}}$&${\text{\small{\color{red}{64}}}}
$&${\text{\color{red}{\small{64}$_{\times 2}$}}} $&${\text{\color{red}\small{128}}} $&${\text{\small{128}}}$&$ {\text{\small{256}}} $
\\
\hline
${\text{\small 5}} $&${\color{red}\text{\small 16}}$&${\text{\small 16}}$&${\text{\small 16}}$&${\text{\small 16}}$&${\color{red}{\text{\small 32}}}$&${\color{red}{\text{\small 64}}} $&${\text{\small{\color{red}128}}}$&${\text{\small{128}}} $&$\text{\small{\color{red}256}}$
\\
\hline
${\text{\small 4}} $&$  {\text{\small 8}}$&$ {\text{\small 8}}$&$
{\text{\small 8}}$&$ 8_{\times 2}$&$16$&${\text{\small 32}}$&${\text{\small 64}}
$&${\text{\small 64}_{\times 2}} $&${\text{\small{128}}}$
\\
\hline
${\text{\small 3}}$&${\color{red}{\text{\small 8}}}$&${\color{red}{\text{\small 8}}}$&${\text{\small\color{red}8}}$&$8$&$16$&$32$
&${\text{\small\color{red}64}}$&$64$&${\color{red}{\text{\small 128}}}$
\\
\hline
${\text{\small 2}}$&${\color{red}{\text{\small 4}}}$&$
{\color{red}4_{\times 2}}$&${\color{red}8}$&$ 8$&$16$&$16_{\times 2}$&$32$&$32 $&${\color{red}{\text{\small 64}}}$
\\
\hline
${\text{\small 1}}$ &${\color{red}2}$&${\color{red}4}$&${\color{red}8}$& $8$&${\color{red}16}$&$16$&$16$&$16$&${\color{red}{\text{\small 32}}}$
\\
\hline
${\text{\small 0}} $&$  1$&$ 2$&$ 4$&$ 4_{\times 2}$&$ 8$&$ 8$&$ 8$&$ 8_{\times 2}$&$16$
\\
\hline\hline
{s/r}&  {\text{\small 0}}& {\text{\small 1}}& 
{\text{\small 2}}&{\text{\small 3}} & {\text{\small 4}}& {\text{\small 5}}& {\text{\small 6}}& {\text{\small 7}}& {\text{\small 8}}
\\
\hline
\end{tabular}\label{t:dim}
\end{table}
We need a couple of more properties of the admissible modules, see~\cite{FM}.

\begin{lemma}\cite{FM}\label{orthogonal}
Let $(V,\langle .\,,.\rangle_V)$ be an admissible module and
$\mathbb J_1,\ldots,\mathbb J_l$ 
symmetric or anti-symmetric linear operators on $V$ such that 
\begin{itemize}
\item[1)] $\mathbb J^2_k=-\Id$, $k=1,\ldots,l$;
\item[2)] $\mathbb J_k\mathbb J_j=-\mathbb J_j\mathbb J_k$ for all $k,j=1,\ldots,l$.
\end{itemize}
Then for any $v\in V$ with $\langle v,v\rangle_V=1$ there is a vector $\tilde v$ satisfying:
$$
\langle\tilde v,\mathbb J_k\tilde v\rangle_V=0,\quad\text{and}\quad\langle\tilde v,\tilde v\rangle_V=1,\ \ k=1,\ldots,l.
$$
If $P$ is a linear operator on $V$ such that 
$
P^2=\Id$, $P\mathbb J_k=\mathbb J_kP$, $k=1,\ldots,l
$, and $v\in V$ with $\langle v,v\rangle_{V}=1$, satisfies $Pv=v$, then the vector $\tilde v$ is also eigenvector of $P$: $P\tilde v=\tilde v$.
\end{lemma}

\begin{rem}\label{sign change}
Let $(V, \la .\,,.\ra_{V})$ be an 
admissible module of a Clifford algebra $\Cl_{r,s}$.
Then it can be easily seen form the definition of an admissible
module, that the same module with the scalar product of the opposite sign
$(V, -\la .\,,.\ra_{V})$ is also an 
admissible module. 
\end{rem}


\subsection{Pseudo $H$-type algebras}


We give the definition of pseudo $H$-type algebras that is convenient
for us to work. 
The equivalent definitions and their relations to Clifford algebras 
can be found in~\cite{AFM,CD,Ci,CDKR,FM,GKM,Ka1,Ka2,Ka3}.

\begin{defi}\label{definition Pseudo}
Let $(U,\la.\,,.\ra_U)$ be an admissible module of a Clifford algebra
$\Cl_{r,s}$ and a map $J\colon \Cl_{r,s}\to\End(U)$ a representation. A 2-step nilpotent
Lie algebra $U\oplus\mathbb{R}^{r,s}$ 
with the center $\mathbb{R}^{r,s}$ and the Lie bracket defined via the relation 
\begin{equation}\label{bracket definition}
\la J_zx,y\ra_{U}=\la z,[x,y]\ra_{r,s},\quad z\in\mathbb{R}^{r,s}, \ \ x,y \in U,
\end{equation}
is called a pseudo $H$-type Lie algebra and is 
denoted by $\mathcal{N}_{r,s}(U)$. If $U=V^{r,s}$ is minimal
admissible, 
then we write $\mathcal{N}_{r,s}$
\end{defi} 
In Section~\ref{sec:step3} we prove the uniqueness of 
the algebra $\mathcal{N}_{r,s}$. One of the particular consequences 
of Definition~\ref{definition Pseudo} is 
$\la J_{z}x,J_{z'}x\ra_{U}=\la z,z'\ra_{r,s}\la x,x\ra_{r,s}$.
Thus for an orthonormal basis $\{z_1,\ldots,z_{r+s}\}$ 
the maps $J_{z_j}\colon U\to U$ are isometries for 
$j=1,\ldots,r$ 
and anti-isometries 
for $j=r+1,\ldots,r+s$.

\begin{theorem}[\cite{CD,Eber03,FM}]\label{integral basis I}
We fix an orthonormal basis $\{z_k\}_{k=1}^{r+s}$ for
$\mathbb{R}^{r,s}$ and  assume that $(V^{r,s},\la.\,,.\ra_{V^{r,s}},)$ is 
a minimal admissible module of $\Cl_{r,s}$ of 
dimension $2N$. Then there exists an orthonormal basis $\{x_i\}_{i=1}^{2N}$ for $V^{r,s}$ such that
\begin{itemize}
\item[1.]{$\la x_i,x_j\ra_{V^{r,s}}=\epsilon_i(N,N)\delta_{i,j}$;}
\item[2.]{For each $k$, the operator $J_{z_k}$ maps $x_i$ to some $x_j$ or $-x_j$ with $j\not=i$;}
\item[3.]{There is a vector $v\in V^{r,s}$, $\la v,v\ra_{V^{r,s}}\neq 0$, such that all the basis $\{x_i\}$ is obtained from $v$ by action of $J_{z_j}$, $j=1,\ldots,r+s$ or their product. 
}
\end{itemize}
\end{theorem}
We call the basis $\{x_i,z_j\}$ for $\mathcal N_{r,s}$ satisfying the
properties of Theorem~\ref{integral basis I} an {\it integral
  basis}. Let $(W,\la.\,,.\ra_W)$ be a vector space with a 
non-degenerate scalar product. 
We say that a vector $w\in W$ is {\it positive} if $\la w,w\ra_{W}>0$, {\it negative} if $\la w,w\ra_{W}<0$, and  {\it a null-vector} if $\la w,w\ra_{W}=0$. 
We formulate some consequences of Theorem~\ref{integral basis I}.

\begin{cor}
If there exists an index $i\in \{ 1, \ldots, 2N\}$ such that 
$J_{z_k}x_i=\pm J_{z_{l}}x_i$, then $k=l$.  Hence any basis
vector $x_i$ is mapped 
to $x_j$ or $-x_j$ by at most one~$J_{z_k}$. 
\end{cor}
\begin{proof}
If $k \leq r$ then $J_{z_k}$ preserves positive and negative
elements. 
If $k>r$, then $J_{z_k}$ interchange the
positive and negative elements. 
Therefore, under the assumption of the corollary only the cases $k,l\leq r$ or $k,l>r$ are possible.
Assume $k\not=l$. Then, from one hand $\pm x_i=J_{z_k}J_{z_{l}}x_i$, but from the other hand 
$$(J_{z_k}J_{z_{l}})^2=- J_{z_k}^2J_{z_{l}}^2=-\la z_k,z_k\ra_{r,s} \la z_{l},z_{l}\ra_{r,s}\Id=-\Id,
$$
which contradicts to the existence of the eigenvalue $1$ or $-1$ of the operator $J_{z_k}J_{z_{l}}$.
\end{proof}

\begin{cor}\label{integral basis II}
Let $\mathcal N_{r,s}$ be a pseudo $H$-type algebra and $\{x_i,z_j\}$
an integral basis. Set $[x_i,x_j]=\sum c_{ij}^kz_k$, then for fixed
$i$ and $j$ the coefficients $c_{ij}^{k}$ 
vanish for all but one $k$ and in the later case $c_{ij}^{k}=\pm 1$. 
\end{cor}
\begin{proof}
The proof follows from 
$
\la J_{z_k}x_i,x_j\ra_{V^{r,s}}=\la z_l,[x_i,x_j]\ra_{r,s}=
\begin{cases}
c_{ij}^{k} & \text{if } k \leq r\\
- c_{ij}^{k} & \text{if } k > r
\end{cases}
$.
\end{proof}

\begin{cor}\label{variable1}
Let $\mathcal N_{r,s}$ be a pseudo $H$-type algebra, $\{x_i,z_j\}$ an
integral basis, 
and $[x_i,x_j]=\pm z_k$. Then
\begin{itemize}
\item[1.]{if either $1\leq i,j\leq N$ or $N<i,j\leq 2N$ then $z_k$ is positive, that
is $k\leq r$,}
\item[2.]{ if $1\leq i\leq N<j\leq 2N$ then $z_k$ is negative, i.e., $k>r$.}
\end{itemize}
\end{cor}
\begin{proof} We prove only the first statement, 
since the second one can be shown similarly. If we assume, 
by contrary, that $k>r$, then $J_{z_k}$ should be an anti-isometry and 
\begin{equation*}
0=\la J_{z_k}x_i,x_j\ra_{V^{r,s}}=\la z_{k},[x_i,x_j]\ra_{r,s}=\pm 1, 
\end{equation*}
which is a contradiction. 
\end{proof}


\subsection{Scheme for 4 step classification}


{\sc Step 1.} We study the isomorphic and non isomorphic cases of Lie algebras 
$\mathcal N_{r,s}$ and $\mathcal N_{s,r}$, $r,s\leq 8$, $r\neq s$ of
equal dimensions, 
see Section~\ref{step1}.  
We also construct an automorphism of $N_{r,r}$, $r=1,2,4$ having 
a special property and show that there is no such an automorphism 
of $\mathcal N_{3,3}$, see Theorem~\ref{th:automorphisms} and Corollary~\ref{non-existence Cl33}. 
Then the periodicity property~\eqref{perCl} will be applied to 
extend these results to higher dimensional Lie algebras, see Theorems~\ref{periodicity} and~\ref{periodicity1}.
\\

{\sc Step 2.} If $\dim(V^{r,s})=2\dim(V^{s,r})$, then the Lie algebras
$\mathcal{N}_{r,s}$ 
and $\mathcal{N}_{s,r}$ are not isomorphic simply because 
they have different dimension.
We call these algebras {\it trivially non-isomorphic}. In this case we  
prove the isomorphism or non-isomorphism of the Lie algebras
$\mathcal{N}_{r,s}(V^{r,s})$ and
$\mathcal{N}_{s,r}(V^{s,r}\oplus V^{s,r})$, see Section~\ref{sec:4}.
\\

{\sc Step 3.} Let $V^{r,s}=V_+$ or $V^{r,s}=V_-$, where $V_+, V_-$ are
non-equivalent irreducible 
modules. We show that the Lie algebras
$\mathcal{N}_{r,s}(V_+)$ 
and $\mathcal{N}_{r,s}(V_{-})$ are isomorphic. 
An analogous question is considered 
when $V^{r,s}=V_+\oplus V_+$ or $V^{r,s}=V_-\oplus V_-$. 
The isomorphism of Lie algebras particularly shows 
the uniqueness of the pseudo $H$-type algebra 
corresponding to two minimal admissible modules, see Section~\ref{sec:step3}. 
\\

{\sc Step 4.}  The last step is devoted to the classification of Lie
algebras, 
constructed from the multiple sum of minimal admissible modules. 
The admissible modules can differ either by the choice of the scalar
product on it 
or they can be defined by non-equivalent representations. 
\\

In this paper we present 3 
steps, finishing the classification of the Lie algebras whose
complement 
to the centre is a minimal admissible module.  
We summarise the classification of the Step 1 among the basic pairs 
in Table~\ref{t:step1}. 
\begin{table}[h]
\center\caption{Classification result after the first step}
\begin{tabular}{|c||c|c|c|c|c|c|c|c|c|}
\hline
$8$&$\cong$&&&&&&&&
\\
\hline
7&d&d&d&$\not\cong$&&&&&
\\
\hline
6&d&$\cong$&$\cong$&h&&&&&
\\
\hline
5&d&$\cong$&$\cong$&h&&&&&
\\
\hline
4&$\cong$&h&h&h&$\circlearrowright$&&&&\\\hline
3&d&$\not\cong$&$\not\cong$&$\not\circlearrowright$&d&d&d&$\not\cong$&d\\\hline
2&$\cong$&h&$\circlearrowright$&$\not\cong$&d&$\cong$&$\cong$&h&$\cong$\\\hline
1&$\cong$&$\circlearrowright$&d&$\not\cong$&d&$\cong$&$\cong$&h&$\cong$\\\hline
0&&$\cong$&$\cong$&h&$\cong$&h&h&h&$\cong$\\\hline\hline
$s/r$&0&1&2&3&4&5&6&7&8\\
\hline
\end{tabular}\label{t:step1}
\end{table}
Here ``d'' stands for ``double'', meaning that $\dim V^{r,s}=2\dim V^{s,r}$, and
``h'' (half) means that $\dim V^{r,s}=\frac{1}{2}\dim V^{s,r}$. 
The corresponding pairs are trivially non-isomorphic due to the 
different dimension of minimal admissible modules. 
The symbol $\cong$ denotes the Lie algebra having 
isomorphic pair, $\not\cong$ shows that the pair 
is non-isomorphic, the symbol $\circlearrowright$ denotes 
the Lie algebra admitting a special type of automorphisms, 
and $\not\circlearrowright$ denotes the Lie algebra not having this type of automorphism.


\subsection{Remarks on Step 4 and further development}


In the forthcoming paper~\cite{FM1} we will deal with Step 4, where we plan to consider the multiple sum $U=\oplus_i V_i$ of 
several minimal admissible modules $V_i=(V,\la.\,,.\ra_V)$. 
Here different minimal admissible modules $V_i$ can have a common
vector space $V$ but 
allows the scalar products of opposite sign. 
The minimal admissible modules can differ also by the choice of the
irreducible modules 
for their construction: $V=V_+$ or $V=V_-$. Finally the minimal
admissible modules can be both 
based on non-equivalent Clifford modules and admit the scalar products
of opposite signs. 
Different combinations can give non-isomorphic Lie algebras 
$\mathcal N_{r,s}(\oplus_i V_i)$.

We also aim to study the automorphism groups of the algebras
$\mathcal{N}_{r,s}$. 
They are determined by solving the equations 
arising during the construction of the map
$A\colon V^{r,s}\to V^{r,s}$. The present paper 
indicates that it is reduced to  
the exact sequence
\[
\{0\}\to K \to Aut(\mathcal N_{r,s})\to O(r,s)\to \{0\},
\]
that defines the map $\Phi \to C$, see~\eqref{iso_form_1} for the form of
$\Phi$. 
The last map is distinguished 
by the properties of $C$. In some cases $C^{\tau} C= \Id$, as, for
instance, 
in the case $\mathcal N_{3,3}$, meanwhile for $\mathcal N_{r,r}$,
$r=1,2,4$ 
one has $C^{\tau}C=\pm \Id$. The map $C$ determines 
the map $A$ and the freedom in the construction of the map $A$ gives
the kernel $K$. 
It can be seen from the present paper that 
it could be $K=\pm \Id$ or $K=SO(2)$. In the forthcoming 
papers we aim to describe all the cases not only for the Lie algebras 
based on the minimal admissible modules, 
but also for the admissible modules of the type $U=\oplus_i V_i$.


\section{Step 1: Lie algebras of minimal dimensions} \label{step1}


\subsection{Necessary condition of existence of an isomorphism}

Let $A\colon U\to \widetilde U$ be a linear map. We denote by $A^{\tau}$ the adjoint map with respect to the 
scalar products on $(U,\la.\,,.\ra_U)$ and $(\widetilde U,\la.\,,.\ra_{\widetilde U})$:
\begin{equation*}\label{}
\la A(x),y\ra_{\widetilde U}=\la x,A^{\tau}(y)\ra_{U},\quad x\in U,\ \ y \in \widetilde U. 
\end{equation*}

\begin{theorem}\label{property of isomorphism}
Let $\{U,\la.\,,.\ra_U;\,J\}$  and $\{\widetilde
U,\la.\,,.\ra_{\widetilde U});\,\wJ\}$ 
be admissible modules and representation maps of the Clifford algebras $\Cl_{r,s}$ and $\Cl_{\tilde r,\tilde s}$, respectively. Assume that $\dim U=\dim \widetilde U$,  $r+s=\tilde r+\tilde s$, and that there is a Lie algebra isomorphism 
$$\Phi\colon\mathcal{N}_{r,s}(U)\to\mathcal{N}_{\tilde r,\tilde s}(\widetilde U)$$
between the corresponding  pseudo $H$-type algebras. 
Then, necessarily, one of the cases $(r,s)=(\tilde r,\tilde s)$ or $(r,s)=(\tilde s,\tilde r)$ holds. 
Moreover, $\Phi$ has to be of the form
\begin{equation}\label{iso_form_1}
\Phi=
\begin{pmatrix}
A&0\\
B&C
\end{pmatrix}:
\begin{array}{l}
U \\
\oplus_{\perp}\\
\mathbb{R}^{r,s}
\end{array}
\longrightarrow 
\begin{array}{l}
\widetilde U \\
\oplus_{\perp}\\
\mathbb{R}^{\tilde r,\tilde s}
\end{array},
\end{equation}
where $A\colon U\to\widetilde  U$ and $C\colon\mathbb{R}^{r,s}\to\mathbb{R}^{\tilde r,\tilde s}$ are linear bijective maps satisfying the relation
\begin{equation}\label{isomorphism relation}
A^{\tau}\wJ_{z}A=J_{C^{\tau}(z)}\quad\text{for any}\quad z\in\mathbb{R}^{\tilde r,\tilde s}. 
\end{equation} 
There is no condition on $B\colon U\to\mathbb{R}^{\tilde r,\tilde s}$
and we may set $B=0$. Multiplying $A$ by a suitable constant,
we may assume that $|\det \left(AA^{\tau}\right)|=1$ and $CC^{\tau}=\pm \Id$.
\end{theorem}
\begin{proof}
If a Lie algebra isomorphism $\Phi\colon\mathcal{N}_{r,s}(U)\to\mathcal{N}_{\tilde r,\tilde s}(\widetilde U)$ exists, then it must be of the form~\eqref{iso_form_1}, 
since it maps the center to the center. The relation~\eqref{isomorphism relation} follows from the definition of Lie brackets
\begin{eqnarray}\label{eq:isomAC}
\la A^{\tau}\wJ_{z} A(x),y\ra_{U} 
& = &
\la\wJ_{z} A(x),A(y)\ra_{\widetilde U}
=
\la z,[A(x),A(y)]\ra_{\tilde r,\tilde s}=\la z,C([x,y])\ra_{\tilde r,\tilde s}\nonumber
\\ 
&= &\la C^{\tau}(z),[x,y]\ra_{r,s}=\la J_{C^{\tau}(z)}x,y\ra_{U},
\end{eqnarray}
for all $x,y\in U$ and $z\in\mathbb{R}^{\tilde r,\tilde s}$ which shows~\eqref{isomorphism relation}. Conversely, if~\eqref{isomorphism relation} holds, then from~\eqref{eq:isomAC} we obtain $[A(x),A(y)]=C([x,y])$ and therefore the map $\Phi=A\oplus C$ is a Lie algebra isomorphism. 
Note that~\eqref{isomorphism relation} implies that $J_{C^{\tau}(z)}$ is singular, if and only if $\wJ_{z}$ is singular. 

Let $z_+$ and $z_-$  be a positive and a negative vector in $\mathbb{R}^{r,s}$, respectively. We set $a_t=(1-t)z_++tz_-$, $0\leq t\leq 1$. Then
\[
\la a_0,a_0\ra_{r,s}=\la z_+,z_+\ra_{r,s}>0\quad\text{and}\quad \la a_1,a_1\ra_{r,s}=\la z_-,z_-\ra_{r,s}<0.
\]
There is ${t_0}\in (0,1)$ with $\la a_{t_0},a_{t_0}\ra_{r,s}=0$ and  therefore $J_{a_{t_0}}$ is singular. On the other hand, if  
$z_1$ and $z_2$ are orthonormal and both positive (negative) vectors in $\mathbb{R}^{r,s}$ and $b_t=(1-t)z_1+tz_2$, $0\leq t\leq 1$, then
$\la b_t,b_t\ra _{r,s}=(1-t)^2+t^2>0$, for all $t \in [0,1]$. This implies that $J_{b_t}$ is non-singular for all $t \in [0,1]$. Hence, the operator $C^{\tau}$ either preserves or reverses the sign of elements in 
$\mathbb{R}^{r,s}$. These observations imply that only the cases $(r,s)=(\tilde r,\tilde s)$ or $(r,s)=(\tilde s,\tilde r)$ are possible,
if $r\not= s$.

For the remaining part of the proof we assume that $r\not=s$ and $\Phi\colon\mathcal{N}_{r,s}(U)
\to \mathcal{N}_{s,r}(\widetilde U)$ is a Lie algebra isomorphism. Then
$
(A^{\tau}\wJ_{z}A)^2=J_{C^{\tau}(z)}^2=-\la C^{\tau}(z),C^{\tau}(z)\ra_{r,s} \Id_{U}$
by~\eqref{isomorphism relation} and therefore
\[
\det\big((A^{\tau}\wJ_{z}A)^2\big)=\left(\det AA^{\tau}\right)^2\la z,z\ra_{s,r}^{2N}=\la C^{\tau}(z),C^{\tau}(z)\ra_{r,s}^{2N},
\]
where $2N=\dim U=\dim \widetilde U$. Since the operator $C^{\tau}\colon\mathbb{R}^{s,r} \to\mathbb{R}^{r,s}$ 
reverses the sign of vectors we obtain
\[
|\left(\det AA^{\tau}\right)|^{1/N}\cdot \la z,z\ra_{s,r}=-\la C^{\tau}(z),C^{\tau}(z)\ra_{r,s}=-\la z,CC^{\tau}(z)\ra_{s,r}.
\]
Multiplying $A$ by a suitable constant  we assume that $|\det AA^{\tau}|=1$ and $CC^{\tau}= -\Id$.
\end{proof}
\begin{cor}\label{Identity relation}
Let $r\not= s$ and $\Phi\colon\mathcal{N}_{r,s}(U)\to\mathcal{N}_{r,s}(\widetilde U)$ is a Lie algebra isomorphism, written in form~\eqref{iso_form_1}. Then $CC^{\tau}=\Id$.  If $r=s$ both cases $CC^{\tau}=\pm \Id$ are possible, see Theorem~\ref{th:automorphisms} and Corollary~\ref{non-existence Cl33}.
\end{cor}

\begin{lemma}\label{iso-form 2}
Let $\{U,\la.\,,.\ra_U;\,J\}$  and $\{\widetilde
U,\la.\,,.\ra_{\widetilde U};\,\wJ\}$ 
be admissible modules and representation maps of the Clifford algebras 
$\Cl_{r,s}$ and $\Cl_{s,r}$, respectively. Assume that $r\not=s$.~If  
$$\Psi=\begin{pmatrix}A&0\\0&C\end{pmatrix}\colon
\mathcal{N}_{r,s}(U)\to\mathcal{N}_{r,s}(\widetilde U)\quad  \text{and}\quad
\Phi
=\begin{pmatrix}A&0\\0&C\end{pmatrix}\colon
\mathcal{N}_{r,s}(U)\to\mathcal{N}_{s,r}(\widetilde U),
$$
are Lie algebra isomorphisms, then the maps defined by
$$
\Psi^{\tau}=\begin{pmatrix}A^{\tau}&0\\0&C^{\tau}\end{pmatrix}
\colon\mathcal{N}_{r,s}(\widetilde U)\to\mathcal{N}_{r,s}(U)\quad 
\text{and}\quad\Phi^{\tau}=\begin{pmatrix}A^{\tau}&0\\0&C^{\tau}\end{pmatrix}
\colon\mathcal{N}_{s,r}(\widetilde U)\to\mathcal{N}_{r,s}(U), 
$$
respectively, are Lie algebra isomorphisms as well.
\end{lemma}
\begin{proof}
First we show that $\Psi^{\tau}$ defines a Lie algebra automorphism.  
According to~\eqref{isomorphism relation} and Corollary~\ref{Identity relation} we have
\begin{equation}\label{GL_relation_r_ne_s}
(A^{\tau}\wJ_{z}A)^2=-\la C^{\tau}(z),C^{\tau}(z)\ra_{r,s}\Id=-\la z,z\ra_{r,s}\Id,
\end{equation}
which implies 
$
AA^{\tau}\wJ_{z}AA^{\tau}\wJ_{z}AA^{\tau}=
-\la z,z\ra_{r,s}AA^{\tau}$.
Multiplying by $(AA^{\tau})^{-1}$ from the right hand side we obtain
\begin{equation*}
AA^{\tau}\wJ_zAA^{\tau}\wJ_{z} =-\la z,z\ra_{r,s}\Id=\wJ_{z}^2\quad\Longrightarrow\quad
A A^{\tau}\wJ_zAA^{\tau}=\wJ_{z}.
\end{equation*}
Replacing $A^{\tau}\wJ_zA$ by $J_{C^{\tau}(z)}$, we get
$
A A^{\tau}\wJ_zAA^{\tau} =AJ_{C^{\tau}(z)}A^{\tau}=\wJ_{z}=\wJ_{CC^{\tau}(z)}$.
Hence the map $\Psi^{\tau}$ is a Lie algebra automorphism.

Next we show that $\Phi^{\tau}$ defines a Lie algebra isomorphism.  From $C^{\tau}C=C C^{\tau}=-\Id$ 
we obtain 
\begin{equation*}\label{}
(A^{\tau}\wJ_{z}A)^2=J_{C^{\tau}(z)}^2=-\la C^{\tau}(z),C^{\tau}(z)\ra_{r,s}\Id= \la z,z\ra_{s,r}\Id=-\wJ_z^2
\end{equation*}
instead of~\eqref{GL_relation_r_ne_s}. It leads to
$
- \wJ_z=AA^{\tau}\wJ_z AA^{\tau}=AJ_{C^{\tau}(z)}A^{\tau}
$
by the same argument as above. Replacing $z$ by $C(z)$, we obtain  
$AJ_zA^{\tau}=\wJ_{C(z)}$, 
which proves the assertion. 
\end{proof}

The structure of a Lie algebra isomorphism inherits somehow $\mathbb Z_2$-grading of the underlying Clifford algebras as shows the following lemma.

\begin{lemma}\label{relation between volume form}
Let $\{z_i\}_{i=1}^{r+s}$ be an orthonormal basis of $\mathbb{R}^{r,s}$ and the maps $A$ and $C$ as in Lemma~\ref{iso-form 2}. Then the following relations hold
\begin{itemize}
\item[1.] If $p=2m$, $m \in \mathbb{N}$, then
\begin{equation}\label{volume form1}
A\prod_{j=1}^{p}J_{z_j}=(-1)^m\prod_{j=1}^{p}\wJ_{C(z_j)}A,\qquad 
A^{\tau}\prod_{j=1}^{p} \wJ_{z_j}=(-1)^m\prod_{j=1}^{p}{J}_{C^{\tau}(z_j)}A^{\tau}.
\end{equation}
\begin{equation}\label{volume_form5}
A^{\tau}A \prod_{j=1}^{p}J_{z_j}=\prod_{j=1}^{p}J_{z_j}A^{\tau} A,\qquad
AA^{\tau}\prod_{j=1}^{p}\wJ_{C(z_j)}=\prod_{j=1}^{p}\wJ_{C(z_j)} A A^{\tau}.
\end{equation}
\item[2.] If $p=2m+1$, $m \in \mathbb{N}\cup\{0\}$, then
\begin{equation}\label{volume form4}
A\prod_{j=1}^{p}J_{z_j}A^{\tau}=(-1)^m\prod_{j=1}^{p}\wJ_{C(z_j)}\qquad
A^{\tau}\prod_{j=1}^{p}\wJ_{z_j} A=(-1)^m \prod_{j=1}^{p}J_{C^{\tau}(z_j)}.
\end{equation}
\begin{equation}\label{volume_form7}
A^{\tau}A \prod_{j=1}^{p}J_{z_j}A^{\tau}A=-\prod_{j=1}^{p}J_{z_j},\qquad
AA^{\tau} \prod_{j=1}^{p}\wJ_{z_j}A A^{\tau}=-\prod_{j=1}^{p}\wJ_{z_j}
\end{equation}
\end{itemize}
\end{lemma}

\begin{proof}
We only show the second parts of equalities, since the first parts can be obtained from them by transpositions. We assume that $C^{\tau}C=-\Id$ and apply the induction arguments. If $m=0$ ($p=1$) then~\eqref{volume form4} is reduced to~\eqref{isomorphism relation}. Assume now that~\eqref{volume form4} holds for  $p=2m+1$. Choose $z^*$ from the orthonormal basis $\{z_i\}_{i=1}^{r+s}$ and calculate
\begin{eqnarray*}
\Big(A^{\tau}\prod_{j=1}^{p}\wJ_{z_j}\Big)\wJ_{z^{*}}
&=&
A^{\tau}\prod_{j=1}^{p}\wJ_{z_j}AA^{-1}\wJ_{z^{*}}
=
(-1)^{m}
\prod_{j=1}^{p}J_{C^{\tau}(z_j)}A^{-1}\wJ_{z^{*}}
(A^{\tau})^{-1}A^{\tau}
 \\
 &= &
 (-1)^{m+1}
 \Big(\prod_{j=1}^{p}J_{C^{\tau}(z_j)} \Big)J_{C^{\tau}(z^{*})}A^{\tau} .
\end{eqnarray*}
Thus, we proved~\eqref{volume form1} for $p=2(m+1)$. In the last equality we argued as follows. Since 
$$
A^{\tau}\wJ_{z^*}^{-1} A =A^{\tau} \left
(-\frac{1}{\la z^*,z^*\ra_{s,r}}\wJ_{z^*}\right) A =
-\frac{1}{\la z^*,z^*\ra_{s,r}}{J}_{C^{\tau}(z^{*})},
$$
and $\la C^{\tau}(z^*),C^{\tau}(z^*)\ra_{r,s}=-\la z^*,z^*\ra_{s,r}$ we obtain
$$
A^{-1}\wJ_{z^*}(A^{\tau})^{-1}=\big(A^{\tau}\wJ_{z^*}^{-1}A \big)^{-1}=\left(-\frac{1}{\la z^*,z^*\ra_{s,r}}{J}_{C^{\tau}(z^*)}\right)^{-1}=-{J}_{C^{\tau}(z^{*})}.$$

Thus in the previous step we, particularly showed that~\eqref{volume form1} is true for $m=1$ ($p=2$). Assume now that~\eqref{volume form1} holds for $p=2m$, $m=0,1,\ldots$, then
\begin{equation*}
A^{\tau}\Big(\prod_{j=1}^{p}\wJ_{z_j}\Big)\wJ_{z^*} A
=(-1)^m
\prod_{j=1}^{p}J_{C^{\tau}(z_j)}A^{\tau}\wJ_{z^*}A
=(-1)^{m}
\Big(\prod_{j=1}^{p}J_{C^{\tau}(z_j)}\Big){J}_{C^{\tau}(Z^{*})}.
\end{equation*}
Thus the assertion~\eqref{volume form4} holds for $p=2m+1$. 

It is sufficient to show~\eqref{volume_form5} for $p=2$. We have by~\eqref{volume form1} 
$$
A^{\tau} A J_{z_1}J_{z_2}
=
-A^{\tau}\wJ_{C(z_1)} \wJ_{C(z_2)} A
=J_{C^{\tau}C(z_1)}J_{C^{\tau}C(z_2)} A^{\tau}A
=J_{z_1}J_{z_2} A^{\tau}A.
$$
Identity~\eqref{volume_form7} can be deduced from~\eqref{volume form4}.  We obtain  
$$
A^{\tau}A \prod_{j=1}^{p}J_{z_j} A^{\tau}  A=(-1)^mA^{\tau}\prod_{j=1}^{p} \wJ_{C(z_j)}A =(-1)^{2m}  \prod_{j=1}^{p}J_{C^{\tau}C(z_j)}=(-1)^p \prod_{j=1}^{p}J_{z_j}.
$$
and since $p$ is odd the equality~\eqref{volume_form7} follows. 
\end{proof}
\begin{rem}
It is clear that the result of Lemma~\ref{relation between volume form} does not depend on the permutation of the basis elements. We emphasise that the existence of a Lie algebra isomorphism $\Phi=A\oplus C\colon \mathcal{N}_{r,s}(U)\to\mathcal{N}_{s,r}(\widetilde U)$ is equivalent to the requirement that relation~\eqref{isomorphism relation} holds.
Moreover, ~\eqref{isomorphism relation} implies all the equalities listed in Lemma~\ref{relation between volume form}. 
\end{rem}


\subsection{Observations on general structure of a possible isomorphism.}


We set up the notations. We denote by $\{z_1,\ldots,z_{r+s}\}$ an orthonormal basis for $\mathbb R^{r+s}$, where $\la z_i,z_j\ra_{r,s}=\epsilon_i(r,s)\delta_{i,j}$.
A linear map $P\colon V^{r,s}\to V^{r,s}$ such that
$P^2=\Id$ is called an involution. The eigenspaces of an
involution $P$ we denote by $E^{k}_{P}$, where $k\in\{1,-1\}$
according to the eigenvalue. In order to denote the intersection of
eigenspaces of several mutually commuting 
involutions $P_j$, $j=1,\ldots,N$, we use multi-index 
$I=(k_1,\ldots,k_N)$ and write $E^I=\cap_{j=1}^{N}E^{k_j}_{P_j}$. 

The basis for $\mathbb R^{s,r}$ 
we denote by $\{w_{s+r},\ldots,w_{r+1},w_r,\ldots,w_1\}$ 
with first ``$s$" elements being positive 
and the last ``$r$'' vectors being negative. 
Therefore, the representation maps $J\colon \Cl_{s,r}\to\End(V^{s,r})$ satisfies
$$
\wJ_{w_j}^2=-\Id_{V^{s,r}},\  j=s+r,\ldots,r+1,\qquad \wJ_{w_j}^2=\Id_{V^{s,r}},\  j=r,\ldots,1.
$$
In general, operators and other objects related to 
the Clifford algebra $\Cl_{r,s}$ will be denoted by 
letters $P,E,R,\ldots$, meanwhile the operators, 
associated to the Clifford algebra $\Cl_{s,r}$ will carry the tilde on the top: $\wP,\wE,\widetilde R,\ldots$. 
At the end we formulate an immediate corollary of 
Lemma~\ref{relation between volume form} that will be used frequently in the paper.

\begin{cor}\label{rem:important}
Let $r\neq s$ and
assume that there is a Lie algebra isomorphism $\Phi=A\oplus C\colon\mathcal N^{r,s}\to\mathcal N^{s,r}$ with
$A\colon V^{r,s}\to V^{s,r}$, $C\colon \mathbb R^{r,s}\to\mathbb R^{s,r}$, where we set 
$
C(z_j)=w_j$ and $C^{\tau}(w_j)=-z_j$.
Let $P_j$, $j=1,\ldots,N$ be mutually commuting isometric involutions on $V^{r,s}$ obtained by product of some $J_{z_k}$. Let $\wP_j$ be mutually commuting isometric involutions on $V^{s,r}$ obtained from $P_j$ by changing $J_{z_k}$ to $\wJ_{w_k}$ and such that $AP_j=\wP_jA$, $j=1,\ldots,N$. We denote by $E^I$ and $\wE^I$ the common eigenspaces of $P_j$ and $\wP_j$, respectively. Then
\begin{itemize}
\item[1.] {the map $A$ can be written as $A=\oplus A_{I}$, where $A_{I}\colon E^{I}\to \wE^{I}$ for any choice of $I=(k_1,\ldots,k_N)$;} 
\item[2.]{if $\prod\limits_{j=1}^p J_{z_j}\colon E^{I}\to E^{I}$ for some $I$, then $\prod\limits_{j=1}^p\wJ_{w_j}\colon \wE^{I}\to \wE^{I}$, and 
$$
A_{I}\prod\limits_{j=1}^p J_{z_j}=
\begin{cases}
(-1)^m\prod\limits_{j=1}^p\wJ_{w_j}(A^{\tau}_{I})^{-1}(x_{I}), \quad &\text{if}\quad p=2m+1,
\\
(-1)^m\prod\limits_{j=1}^p \wJ_{w_j}A_{I}(x_I), \quad &\text{if}\quad p=2m,
\end{cases}\qquad x_I\in E_I,
$$
$$ 
A^{\tau}_{I}\prod\limits_{j=1}^p\wJ_{w_j}=
\begin{cases}
(-1)^{m+1}\prod\limits_{j=1}^p J_{z_j} (A_{I})^{-1}(y_I), \quad &\text{if}\quad p=2m,+1
\\
(-1)^m\prod\limits_{j=1}^p J_{z_j} A^{\tau}_{I}(y_I),\quad &\text{if}\quad p=2m,
\end{cases}\qquad y_I\in \wE_I.
$$
}
\end{itemize}
\end{cor}

Corollary~\ref{rem:important} gives 
an idea of a possible construction of an 
isomorphism $\Phi=A\oplus C\colon\mathcal{N}^{r,s}\to\mathcal{N}^{s,r}$. 
Choosing 
the bases $\{z_j\}_{j=1}^{r+s}$ for $\mathbb R^{r,s}$ and $\{w_j\}_{j=1}^{s+r}$ for $\mathbb R^{s,r}$ we define the map $C\colon \mathbb R^{r,s}\to\mathbb R^{s,r}$, by setting $C(z_j)=w_j$ and $C^{\tau}(w_j)=-z_j$.
Further, if we find mutually commuting isometric involutions $P_j$ and $\wP_j$, $j=1,\ldots,N$, acting on $V^{r,s}$ and $V^{s,r}$, respectively, we can reduce the construction of the map $A\colon V^{r,s}\to V^{s,r}$ to the construction of the maps $A^I\colon E^I\to \wE^I$. Finally, we set $A=\oplus A_I$. Theorem~\ref{th:general} states that, under some conditions, the construction of all maps $A_I$ can be obtained from the only one map $A_1\colon E^1\to \wE^1$, where we denote $E^1=\bigcap_{j=1}^{N}E^{1}_{P_j}$.

\begin{theorem}\label{th:general} We set $C(z_j)=w_j$ and $C^{\tau}(w_j)=-z_j$ for orthonormal bases $\{z_j\}_{j=1}^{r+s}$ for $\mathbb R^{r,s}$ and $\{w_j\}_{j=1}^{s+r}$ for $\mathbb R^{s,r}$. 
Let $P_j$, $j=1,\ldots,N$ be mutually commuting isometric involutions on $V^{r,s}$ obtained by product of some $J_{z_k}$ and $\wP_j$ be mutually commuting isometric involutions on $V^{s,r}$ obtained from $P_j$ by changing $J_{z_k}$ to $\wJ_{w_k}=\wJ_{C(z_k)}$. We denote by $E^I$ and $\wE^I$ the common eigenspaces of $P_j$ and $\wP_j$, respectively, and set $E^1=\bigcap_{j=1}^{N}E^{1}_{P_j}$ and $\wE^1=\bigcap_{j=1}^{N}\wE^{1}_{\wP_j}$. We assume also that 
\begin{itemize}
\item[(a)]{there are maps $G_I\colon E^1\to E^I$ for all multi-indices $I$, written in the form of product $G_I=\prod J_{z_i}$, and}
\item[(b)]{there exists a map $A_1\colon E^1\to \wE^1$ such that 
\begin{equation}\label{eq:A11}
A_{1}\prod_{j=1}^pJ_{z_j}=
\begin{cases}
(-1)^m\prod\limits_{j=1}^p\wJ_{C(z_j)}(A^{\tau}_{1})^{-1},\ \ &\text{if}\ \ p=2m+1,
\\
(-1)^m\prod\limits_{j=1}^p\wJ_{C(z_j)}A_1,\ \ &\text{if}\ \ p=2m,
\end{cases}
\end{equation}
for any choice of the product $\prod_{j=1}^pJ_{z_j}$ that leaves invariant the space $E^1$.
}
\end{itemize}
Then there is a map $A\colon V^{r,s}\to V^{s,r}$ such that $\Phi=A\oplus C\colon\mathcal N^{r,s}\to\mathcal N^{s,r}$ is the Lie algebra isomorphism.
\end{theorem}

\begin{proof}
We define the maps $A_I\colon E^I\to\wE^I$ by the following 
\begin{equation}~\label{eq:AII}
A_I=
\begin{cases}
(-1)^m\wG_I(A_1^{-1})^{\tau}G_I^{-1},\ \ &\text{if}\ \  G_I=\prod\limits_{j=1}^{p=2m+1} J_{z_j},\ \ \wG_I=\prod\limits_{j=1}^{p=2m+1} \wJ_{w_j},
\\
(-1)^m\wG_IA_1G_I^{-1},\ \ &\text{if}\ \  G_I=\prod\limits_{j=1}^{p=2m} J_{z_j},\ \quad \wG_I=\prod\limits_{j=1}^{p=2m} \wJ_{w_j}.
\end{cases}
\end{equation}
Here and further $\wJ_{w_k}=\wJ_{C(z_k)}$. For the convenience we also write the adjoint maps.
\begin{equation}~\label{eq:AIt}
A_I^{\tau}=
\begin{cases}
(-1)^{m+1}G_IA_1^{-1}\wG_I^{-1},\ \ &\text{if}\ \  G_I=\prod\limits_{j=1}^{p=2m+1} J_{z_j},\ \ \wG_I=\prod\limits_{j=1}^{p=2m+1} \wJ_{w_j},
\\
(-1)^m G_IA_1^{\tau}\wG_I^{-1},\ \ &\text{if}\ \  G_I=\prod\limits_{j=1}^{p=2m} J_{z_j},\ \quad \wG_I=\prod\limits_{j=1}^{p=2m} \wJ_{w_j}.
\end{cases}
\end{equation}
Then we set $A=\oplus A_I$. We only need to check the condition $AJ_{z_j}A^{\tau}=\wJ_{C(z_j)}$ for any $z_j$ from the orthonormal basis for $\mathbb R^{r,s}$. 

Observe the following facts. The spaces $E^I$ are mutually orthogonal because if $P_j(x)=x$, and $P_j(y)=-y$ for some isometry $P_j$, then
\begin{equation}\label{eq:isomP}
\la x,-y\ra_{V^{r,s}}=\la P_j(x),P_j(y)\ra_{V^{r,s}}=\la x,y\ra_{V^{r,s}}\quad\Longrightarrow\quad \la x,y\ra_{V^{r,s}}=0.
\end{equation}
Thus $V^{r,s}=\oplus E^I$, and $V^{s,r}=\oplus \wE^I$, where the direct sums are orthogonal. The maps $G_I$ and $\wG_I$ are invertible and 
\begin{equation}\label{eq:prod}
G_I^{-1}=(\prod_{j=1}^{p}J_{z_j})^{-1}=(-1)^p\prod_{j=1}^{p}\la z_j,z_j\ra_{r,s}^{-1}\prod_{k=0}^{p-1}J_{z_{p-k}}.
\end{equation} 
Lemma~\ref{relation between volume form} implies that
\begin{equation}\label{eq:AI}
(A_1^{-1})^{\tau}\prod_{j=1}^{p}J_{z_j}A_1^{-1}=(-1)^{m+1}\prod_{j=1}^{p}\wJ_{C(z_j)},\qquad
A_1\prod_{j=1}^{p}J_{z_j}A_1^{\tau}=(-1)^m\prod_{j=1}^{p}\wJ_{C(z_j)},
\end{equation}
if $p=2m+1$, $m=0,1,\ldots$, and 
\begin{equation}\label{eq:AI1}
(A_1^{-1})^{\tau}\prod_{j=1}^{p}J_{z_j}A_1^{\tau}=(-1)^m\prod_{j=1}^{p}\wJ_{C(z_j)},
\qquad
A_1\prod_{j=1}^{p}J_{z_j}A_1^{-1}=(-1)^m\prod_{j=1}^{p}\wJ_{C(z_j)},
\end{equation}
if $p=2m$, $m=1,\ldots$.

We choose $J_{z_{j_0}}$ and $y\in V^{s,r}=\oplus\wE^I$. Then we write $y=\oplus y_I$ with $y_I\in\wE^I$. Thus we distinguish the cases when the map $G_I$ is the product of odd or even number of representation maps $J_{z_i}$. Moreover, for the multi-index $I$ we find a multi-index $K$ such that $G_K^{-1}J_{z_{j_0}}G_I$ leaves invariant the space $E^1$. Since $G_K$ can also be product of even or odd number of $J_{z_k}$, we differ the following cases:
\begin{eqnarray*}
&&AJ_{z_{j_0}}A^{\tau}y_I=A_KJ_{z_{j_0}}A^{\tau}_Iy_I=
\\
&=&
\begin{cases}
(-1)^{k+m+1}\wG_K(A_1^{-1})^{\tau}G_K^{-1}J_{z_{j_0}}G_IA_1^{-1}\wG_I^{-1}y_I\ \ &\text{if}\ G_I=\prod\limits_{i=1}^{2m+1}J_{z_i},\ G_K=\prod\limits_{l=1}^{2k+1} J_{z_l},
\\
(-1)^{k+m+1}\wG_KA_1G_K^{-1}J_{z_{j_0}}G_IA_1^{-1}\wG_I^{-1}y_I\ \ &\text{if}\ G_I=\prod\limits_{i=1}^{2m+1}J_{z_i},\ G_K=\prod\limits_{l=1}^{2k} J_{z_l},
\\
(-1)^{k+m}\wG_K(A_1^{-1})^{\tau}G_K^{-1}J_{z_{j_0}}G_IA_1^{\tau}\wG_I^{-1}y_I\ \ &\text{if}\ G_I=\prod\limits_{i=1}^{2m}J_{z_i},\ \ G_K=\prod\limits_{l=1}^{2k+1} J_{z_l},
\\
(-1)^{k+m}\wG_KA_1G_K^{-1}J_{z_{j_0}}G_IA_1^{\tau}\wG_I^{-1}y_I\ \ &\text{if}\ G_I=\prod\limits_{i=1}^{2m}J_{z_i},\ \ G_K=\prod\limits_{l=1}^{2k} J_{z_l},
\end{cases}
\end{eqnarray*}
by definitions~\eqref{eq:AII} and~\eqref{eq:AIt}  of $A_I$ and $A_I^{\tau}$. 
Now we observe that
\begin{eqnarray}\label{eq:prodt}
\prod_{l=1}^{q}\frac{1}{\la z_l,z_l\ra_{r,s}}\prod_{n=0}^{q-1}\wJ_{C(z_{q-n})}
&=&
\prod_{l=1}^{q}\frac{-1}{\la C(z_l),C(z_l)\ra_{s,r}}\prod_{n=0}^{q-1}\wJ_{C(z_{q-n})}\nonumber
\\
&=&\prod_{n=0}^{q-1}\wJ_{C(z_{q-n})}^{-1}=\wG_K^{-1}.
\end{eqnarray}
Counting the number of elements in the product $G_K^{-1}J_{z_{j_0}}G_I$ and using~\eqref{eq:prod} for $G_K$, we apply corresponding formulas from~\eqref{eq:AI} or~\eqref{eq:AI1}, and then use~\eqref{eq:prodt}. We obtain
\begin{eqnarray*}
&&AJ_{z_{j_0}}A^{\tau}y_I=
\\
&=&
\begin{cases}
(-1)^{4k+2m+4}\wG_K\wG_K^{-1}\wJ_{C(z_{j_0})}\wG_I\wG_I^{-1}y_I\ \ &\text{if}\ G_I=\prod\limits_{i=1}^{2m+1}J_{z_i},\ G_K=\prod\limits_{l=1}^{2k+1} J_{z_l},
\\
(-1)^{4k+2m+2}\wG_K\wG_K^{-1}\wJ_{C(z_{j_0})}\wG_I\wG_I^{-1}y_I\ \ &\text{if}\ G_I=\prod\limits_{i=1}^{2m+1}J_{z_i},\ G_K=\prod\limits_{l=1}^{2k} J_{z_l},
\\
(-1)^{4k+2m+2}\wG_K\wG_K^{-1}\wJ_{C(z_{j_0})}\wG_I\wG_I^{-1}y_I\ \ &\text{if}\ G_I=\prod\limits_{i=1}^{2m}J_{z_i},\ \ G_K=\prod\limits_{l=1}^{2k+1} J_{z_l},
\\
(-1)^{4k+2m}\wG_K\wG_K^{-1}\wJ_{C(z_{j_0})}\wG_I\wG_I^{-1}y_I\ \ &\text{if}\ G_I=\prod\limits_{i=1}^{2m}J_{z_i},\ \ G_K=\prod\limits_{l=1}^{2k} J_{z_l}.
\end{cases}
\end{eqnarray*}
Thus $AJ_{z_{j_0}}A^{\tau}y_I=\wJ_{C(z_{j_0})}y_I$ and we finish the proof.
\end{proof}


\subsection{Isomorphic Lie algebras}\label{sec:basic}


We start from the construction of the isomorphism of Lie algebras $\mathcal N_{r,s}$ and $\mathcal N_{s,r}$ in 8 basic cases. We also show the existence of Lie algebra automorphisms $\Psi\colon\mathcal N_{r,r}\to\mathcal N_{r,r}$, $r=1,2,4$ such that $\Psi=A\oplus C$, $CC^{\tau}=-\Id$. This allows to apply the periodicity arguments in Theorems~\ref{periodicity} and~\ref{periodicity1}  for the classification of higher dimensional Lie algebras. 

In the forthcoming theorems in order to build a Lie algebra
isomorphism we start from the construction of a convenient basis for
the space $E^{1}=\cap_{j=1}^{N}E^1_{P_j}$, where $P_j$, $j=1,\ldots,N$
are some mutually commuting isometric involutions. To construct the
basis we need to find 
a vector $v\in E^{1}$ with $\la v,v\ra_{V^{r,s}}=1$.
Therefore, one has to be sure that the restriction of $\la.\,,.\ra_{V^{r,s}}$ to $E^1$ is positive definite or neutral. 
The following lemmas describe sufficient conditions for that. In the case when $E^1$ is one dimensional we change
the sign of the scalar product on the module space if it needs, see Remark~\ref{sign change}.

\begin{lemma}\label{lem:PT1}
Let $(V,\langle .\,,.\rangle_V)$ be a neutral scalar product space and $P\colon V \to V$ an isometric involution. Then we have the following cases.
\begin{itemize} 
\item[1)] If a linear map $R\colon V\to V$ is an isometry such that $PR=-RP$, then each of eigenspaces 
$E^{1}$ and $E^{-1}$ of $P$ is a neutral scalar product space with respect to the 
restriction of the scalar product $\langle .\,,.\rangle_V$ on $E^{1}$ and $E^{-1}$.
\item[2)] If a linear map $R\colon V\to V$ is an anti-isometry such that  $PR=-RP$, 
then the restriction of $\langle .\,,.\rangle_V$ on each of
$E^{1}$, $E^{-1}$ is non-degenerate neutral or sign definite,
\item[3)] If a linear map $R\colon V\to V$ is an anti-isometry 
such that $PR=RP$, then the restriction 
of $\langle .\,,.\rangle_V$ on each of $E^{1}$, $E^{-1}$ is non-degenerate neutral.
\end{itemize}
\end{lemma}
\begin{proof} 
To show the first statement we observe that the isometry $R$ acts as an isometry from $E^{1}$ to $E^{-1}$.
Since the eigenspaces $E^{1}$ and $E^{-1}$ are orthogonal, see~\eqref{eq:isomP}, the scalar product $\langle .\,,. \rangle_V$
restricted to each $E^{1}$, $E^{-1}$ is non-degenerate. If the scalar product
restricted to $E^1$ would be positive definite, then the scalar product
restricted to $E^{-1}$ would be also positive definite, since the map $R$
is an isometry which contradicts the assumption that space
$(V,\langle.\,,.\rangle_V)$ is neutral. The same arguments
show that the restriction to $E^1$ could not be negative definite. 
So the scalar product restricted to $E^1$ and therefore to $E^{-1}$
should be neutral.

In order to prove the second statement, we note that since 
$R\colon E^1\to E^{-1}$ is an anti-isometry, the restriction of 
$\langle.\,,.\rangle_V$ to $E^1$ can be sign definite 
and the restriction of $\langle.\,,.\rangle_V$ to $E^{-1}$ 
will have opposite sign due to neutral nature of $(V,\langle.\,,.\rangle_V)$.

In the third case since the eigenspaces $E^1$ and $E^{-1}$ are invariant under
$R$ 
but contains positive and negative vectors, 
then each of them is decomposed into subspaces of equal 
dimension and the restriction of $\langle.\,,.\rangle_V$ on these subspaces is sign definite but of opposite signs. Thus $E^1$ and $E^{-1}$ are neutral spaces. 
\end{proof}

\begin{lemma}\label{lem:PT2}
Let $(V,\langle.\,,.\rangle_V)$ be a neutral scalar product space.
Let $P_1,\ldots, P_N$ be isometric mutually commuting involutions defined on $(V,\langle.\,,.\rangle_V)$ and $R_1,\ldots, R_N, R_{N+1}$ linear anti-isometric operators on $V$ such that
$$
\begin{array}{lllllll}
&P_1R_1=-R_1P_1, & P_1R_2=R_2P_1, &\ldots & P_1R_N=R_NP_1, & P_1R_{N+1}=R_{N+1}P_1,
\\
&&P_2R_2=-R_2P_2,&\ldots& P_2R_N=R_NP_2, &P_2R_{N+1}=R_{N+1}P_2,
\\
&&&&\vdots
\\
&&&&P_NR_N=-R_NP_N,& P_NR_{N+1}=R_{N+1}P_N.
\end{array}
$$ 
Then each common eigenspace $E^I$ of $P_1,\ldots, P_N$ is a non-trivial and neutral scalar product space.
\end{lemma}
\begin{proof} Let us assume that $P_1$ and $R_1$, $R_2$ satisfies the conditions: $P_1R_1=-R_1P_1$ and $P_1R_2=R_2P_1$.
The non-degeneracy of the restriction can be shown as in Lemma~\ref{lem:PT1}.
The presence of the operator $R_1$ ensures that the restriction of $\langle.\,,.\rangle_V$ to $E^{1}_{P_1}$ or $E^{-1}_{P_1}$ is neutral or sign definite and the spaces $E^{1}_{P_1}$ and $E^{-1}_{P_1}$ have equal dimention. Since $R_2$ preserves $E^{1}_{P_1}$ and it is an anti-isometry, the space $E^{1}_{P_1}$ contains both positive and negative vectors forming subspaces of equal dimension. The same arguments, applied to $E^{-1}_{P_1}$. Thus, spaces $E^{1}_{P_1}$ and $E^{-1}_{P_1}$ are, actually, neutral spaces.

Now we repeat the arguments applying them to the neutral spaces $E^{1}_{P_1}$ and $E^{-1}_{P_1}$ and the operators $P_2$ and $R_2$, $R_3$. After $N$ steps we finish the proof.
\end{proof}

We call the anti-isometric operators $R_1,\ldots, R_N, R_{N+1}$, described in Lemma~\ref{lem:PT2} {\it complementary operators} to the family $P_1,\ldots, P_N$. In some of situations the operator $R_{N+1}$ can be omitted, but we still call the system of operators $R_1,\ldots, R_N$ complementary.

We say that three operators $\bi,\bj,\bk\colon V^{r,s}\to V^{r,s}$ form a quaternion structure if they satisfy the relation
\begin{equation}\label{eq:q_rel}
{\bf i}^2={\bf j}^2={\bf k}^2= -\Id_{V^{r,s}}, \quad
{\bf i}{\bf j}={\bf k}=-{\bf j}{\bf i}, \ \ 
{\bf j}{\bf k}={\bf i}=-\bk \bj,\ \ \bk\bi=\bj=-\bi\bk.
\end{equation}

\begin{theorem}\label{th:10-01}
The Lie algebras $\mathcal N_{r,0}$ and $\mathcal N_{0,r}$ are isomorphic for $r=1,2,4,8$.
\end{theorem}
\begin{proof} {\sc Case $r=1$}.
The Clifford algebra $\Cl_{1,0}$ has 2-dimensional minimal admissible module $V^{1,0}$ with positive definite scalar product. The 2-dimensional minimal admissible module $V^{0,1}$ is isometric to $\mathbb R^{1,1}$. We choose the vectors $v\in V^{1,0}$ and $u\in V^{0,1}$, such that $\la v,v\ra_{V^{1,0}}=\la u,u\ra_{V^{0,1}}=1$. Then we construct the orthonormal bases:
$$
\{v,\ J_{z}v,\  z\}\quad\text{for}\quad \mathcal N_{1,0},\qquad\text{and}\qquad
\{u,\ J_{w}u,\  w\}\quad\text{for}\quad \mathcal N_{0,1}.
$$
The isomorphism map $\Phi\colon\mathcal N_{1,0}\to\mathcal N_{0,1}$ is given by
$$
v\mapsto u\quad J_{z}v\mapsto J_{w}u,\quad z\mapsto w,
$$
and the non-vanising commutators are $[v,J_{z}v]=z$, and $[u,J_{w}u]=w$. Here $A^{\tau}u=v$, $A^{\tau}J_{w}u=-J_{z}v$, $C^{\tau}(w)=-z$.
We see from the commutation relations that the Lie algebras $\mathcal N_{1,0}$ and $\mathcal N_{0,1}$ are isomorphic to the Heisenberg algebra.
\\

{\sc Case $r=2$.}  The minimal admissible module $V^{2,0}$ is isometric to $\mathbb R^{4,0}$ and $V^{0,2}$ is isometric to $\mathbb R^{2,2}$. We choose $v\in V^{2,0}$ and $u\in V^{0,2}$, with $\la v,v\ra_{V^{2,0}}=\la u,u\ra_{V^{0,2}}=1$ and construct the orthonormal bases:
$$
\{x_1=v,\ x_2=J_{z_1}v,\ x_3=J_{z_2}v,\ x_4=J_{z_1}J_{z_2}v,\  z_1,\ z_2\}\ \ \text{for }\ \mathcal N_{2,0},
$$
$$
\{y_1=u,\ y_2=J_{w_1}u,\ y_3=J_{w_2}u,\ y_4=J_{w_2}J_{w_1}u,\  w_1,\ w_2\}\ \text{for}\ \mathcal N_{0,2}.
$$
The isomorphism $\Phi$ is given by $x_j\mapsto y_j$, $j=1,\ldots,4$ and $z_k\mapsto w_k$, $k=1,2$
and then it is extended by linearity. The non-vanishing commutation relations on $\mathcal N_{2,0}$ are 
$$
[x_1,x_2]=z_1,\quad [x_1,x_3]=z_2,\quad [x_2,x_4]=-z_2,\quad [x_3,x_4]=z_1,
$$
and, correspondingly, for the basis of $\mathcal N_{0,2}$.
\\

{\sc Case $r=4$}.
The minimal admissible module $V^{4,0}$ is isometric to $\mathbb R^{8,0}$. We choose an isometric involution $
P=J_{z_1}J_{z_2}J_{z_3}J_{z_4}
$ on $V^{4,0}$ and  write $V^{4,0}=E^1_P\oplus E^{-1}_P$. The operators
${\bf i}=J_{z_1}J_{z_2}$, ${\bf j}=J_{z_1}J_{z_3}$, ${\bf k}=J_{z_2}J_{z_3}$,
define a quaternion structure on $E^{1}_P$, commute with $P$ and therefore leave invariant the space $E^1_P$. Let
$v\in V^{4,0}$ be such that $\la v,v\ra_{V^{4,0}}=1$ and $P(v)=v$. Then
$$
\{x_1=v,\ x_2=\bi(v),\ x_3=\bj(v),\ x_4=\bk(v)\}\quad\text{is an orthonormal basis for}\quad E^{1}_P.
$$
The minimal admissible module $V^{0,4}$ is isometric to $\mathbb R^{4,4}$. Let $
\wP=\wJ_{w_1}\wJ_{w_2}\wJ_{w_3}\wJ_{w_4}
$ and we write $V^{0,4}=\wE^1_{\wP}\oplus \wE^{-1}_{\wP}$. The complementary operators are $\widetilde R_1=\wJ_{w_3}$ and $\widetilde R_2=\wJ_{w_1}\wJ_{w_3}$. Choose
$u\in \wE^1_{\wP}$ such that $\la u,u\ra_{V^{0,4}}=1$. Then the operators 
$\wbi=\wJ_{w_1}\wJ_{w_2}$, $\wbj=\wJ_{w_1}\wJ_{w_3}$, $\wbk=-\wJ_{w_2}\wJ_{w_3}$
are used to define the orthonormal basis 
$
\{y_1=u,\ y_2=\wbi(u),\ y_3=\wbj(u),\ y_4=\wbk(u)\}$ for $\wE^{1}_{\wP}$.

Now we construct the map $\Phi=A\oplus C$ by setting $C(z_k)=w_k$, $C^{\tau}(w_k)=-z_k$ and $A=A_1\oplus A_{-1}$ according to Corollary~\ref{rem:important}. To construct $A_1\colon E^1_P\to \wE^{1}_{\wP}$, we write $
A_{1}(v)=(a_1+a_2\tilde{\bf i}+a_3\tilde{\bf j}+a_4\tilde{\bf k})u$. Moreover, $A_1$ has to satisfy the relations $A_1\bi=-\wbi A_1$, $A_1\bj=-\wbj A_1$, $A_1\bk=\wbk A_1$. Thus we obtain
\begin{equation*}
A_{1}(x_2)=-\wbi A_{1}(v),\quad
A_{1}(x_3)=-\wbj A_{1}(v),\quad
A_{1}(x_4)=\wbk A_{1}(v).
\end{equation*}
We conclude that the matrix for the map $A_1$ is given by
$$
A_1=
\begin{pmatrix}
a_1 &a_2&a_3&-a_4
\\
a_2 &-a_1&-a_4&-a_3
\\
a_3 &a_4&-a_1&a_2
\\
a_4 &-a_3&a_2&a_1
\end{pmatrix}.
$$
Notice that $\det A_1=0$ if and only if $a_1=a_2=a_3=a_4=0$. Any choice of $a_j$, $j=1,2,3,4$, such that $\det A_1\neq 0$, defines the map $A_1$.

The map $J_{z_1}\colon E^1_P\to E^{-1}_P$ is used to define $A_{-1}\colon E^{-1}_P\to \wE^{-1}_{\wP}$ by 
$
A_{-1}=\wJ_{w_1}(A_{1}^{-1})^{\tau}J_{z_1}^{-1}$. The proof of this case is finished by applying Theorem~\ref{th:general}.
\\

{\sc Case $r=8$.} Recall that the minimal admissible module $V^{8,0}$ is isometric to $\mathbb R^{16,0}$.
We fix the mutually commuting isometric involutions acting on~$V^{8,0}$:
$$
P_1=J_{z_1}J_{z_2}J_{z_3}J_{z_4},\quad
P_2=J_{z_1}J_{z_2}J_{z_5}J_{z_6},\quad
P_3=J_{z_1}J_{z_2}J_{z_7}J_{z_8},\quad
P_4=J_{z_1}J_{z_3}J_{z_5}J_{z_7}.
$$
The common iegenspaces $E^{I}$ are one dimensional. We construct an orthonormal basis for $V^{8,0}$ starting from a vector $v\in E^{1}=\cap_{j=1}^{4}E^1_{P_j}$ such that $\la v,v\ra_{V^{8,0}}=1$:
\begin{equation}\label{eq:basis80}
\begin{array}{lllllll}
&x_{1}=v,\quad &x_{2}=J_{z_1}J_{z_2}v,\quad &x_{3}=J_{z_1}J_{z_3}v,\quad &x_{4}=J_{z_1}J_{z_4}v,
\\
&x_{5}=J_{z_1}J_{z_5}v, & x_{6}=J_{z_1}J_{z_6}v, & x_{7}=J_{z_1}J_{z_7}v, & x_{8}=J_{z_1}J_{z_8}v,
\\
&x_{9}=J_{z_1}v, & x_{10}=J_{z_2}v, & x_{11}=J_{z_3}v, & x_{12}=J_{z_4}v,
\\
&x_{13}=J_{z_5}v, &x_{14}=J_{z_6}v, & x_{15}=J_{z_7}v, &x_{16}=J_{z_8}v.
\end{array}
\end{equation}
Analogously, the isometric involutions $\wP_j$, obtained by changing $J_{z_j}$ to $\wJ_{w_j}$ in $P_j$, $j=1,2,3,4$, are used to construct an orthonormal basis for $V^{0,8}$ by changing $x_k=\prod_l J_{z_{j_l}}v$ to $y_k=\prod_l\wJ_{w_{j_l}}u$, where $u\in\wE^{1}$ with $\la u,u\ra_{V^{0,8}}=1$. The complementary anti-isometric operators $\widetilde R_1=\wJ_{w_1}\wJ_{w_5}$, $\widetilde R_2=\wJ_{w_6}$, and $\widetilde R_3=\wJ_{w_7}$ guarantee that the space $\wE=\cap_{j=1}^{3}\wE^1_{\wP_j}$ is two dimensional neutral. If the restriction of $\la .\,,.\ra_{V^{0,8}}$ to the space $\wE^{1}_{\wP_4}\cap \wE$ is not positive definite, then we change the sign of the scalar product by Remark~\ref{sign change}.

We claim that the map $\Phi=A\oplus C\colon \mathcal N_{8,0}\to\mathcal N_{0,8}$ such that
$$
A(v)=u, \quad A(x_j)=-y_j,\ j=2,\ldots, 8,\quad A(x_j)=y_j,\ j=9,\ldots,16,
$$
and $C(z_k)=w_k$, $C^{\tau}(w_k)=-z_k$, $k=1,\ldots,8$
is the Lie algebra isomorphism. We show it by checking the commutators.
First observe, that the structure of involutions implies that for any $1<i<j\leq 8$ there is $1<k\leq 8$ such that
\begin{equation}\label{eq:bas80}
J_{z_i}J_{z_j}=\pm J_{z_1}J_{z_k}\quad\text{and simultaneously }\quad \wJ_{w_i}\wJ_{w_j}=\pm \wJ_{w_1}\wJ_{w_k}.
\end{equation}
The second observation is that $[x_i,x_j]=[y_i,y_j]=0$ if either $1\leq i,j\leq 8$ or $9\leq i,j\leq 16$. Indeed, for instance, for any $1\leq i\leq 8$ and $1<j\leq 8$, we calculate
\begin{eqnarray*}
\la [x_1,x_i],z_j \ra_{8,0}
&=&
\la [v,J_{z_1}J_{z_i}v],z_j \ra_{8,0}=\la J_{z_j}v,J_{z_1}J_{z_i}v\ra_{V^{8,0}}
=\la J_{z_i}J_{z_j}v,J_{z_1}v\ra_{V^{8,0}}
\\
&=&
\pm\la J_{z_1}J_{z_k}v,J_{z_1}v\ra_{V^{8,0}}=\pm \la J_{z_k}v,v\ra_{V^{8,0}}=0,
\end{eqnarray*} 
for any $1<k\leq 8$. Analogously, the rest of the cases is proved by using~\eqref{eq:bas80}.

To show that $\Phi=A\oplus C$ is a Lie algebra isomorphism, we need to check
$C[x_i,x_j]=[A(x_i),A(x_j)]$ for any choice of $i=1,\ldots 8$ and $j=9,\ldots 16$, since all other commutators vanish.
We calculate for any $k=1,\ldots, 8$, $i=1$, and $j=9,\ldots 16$
$$
\la C[x_1,x_j],w_k\ra_{0,8}
=
-\la [x_1,x_j],z_k\ra_{8,0}
=-\la J_{z_k}v,J_{z_l}v\ra_{V^{8,0}}
=
-\la z_k,z_l\ra_{8,0}\delta_{k,l},
$$
$$
\la [A(x_1),A(x_j)],w_k\ra_{0,8}
=
\la [y_1,y_j],w_k\ra_{0,8}
=\la \wJ_{w_k}u,\wJ_{w_l}u\ra_{V^{0,8}}
=
\la w_k,w_l\ra_{0,8}\delta_{k,l}.
$$
Since  
$
-\la z_k,z_l\ra_{8,0}\delta_{k,l}=\la w_k,w_l\ra_{0,8}\delta_{k,l}$, we obtain that $C[x_1,x_j]=[A(x_1),A(x_j)]$.
We continue and calculate for any $k=1,\ldots, 8$, $i=2,\ldots 8$, and $j=9,\ldots 16$
\begin{eqnarray}\label{eqn:C}
\la C[x_i,x_j],w_k\ra_{0,8}
&=&
\la [x_i,x_j],C^{\tau}(w_k)\ra_{8,0}
=
-\la [x_i,x_j],z_k\ra_{8,0}
\\
& = &-\la J_{z_k}J_{z_1}J_{z_m}v,J_{z_l}v\ra_{V^{8,0}}
=
-\varepsilon\la J_{z_1}J_{z_m}v,J_{z_1}J_{z_n}v\ra_{V^{8,0}}\nonumber
\\
&= &
-\varepsilon\la z_1,z_1\ra_{8,0}\la z_m,z_n\ra_{8,0}\delta_{m,n},\nonumber
\end{eqnarray}
where $\varepsilon=\pm 1$ and depends on number of permutations and sign in~\eqref{eq:bas80}.
Analogously
\begin{eqnarray*}
\la [A(x_i),A(x_j)],w_k\ra_{0,8}
&=&
\la [-y_i,y_j],w_k\ra_{0,8}
\\
& = &-\la \wJ_{w_k}\wJ_{w_1}\wJ_{w_m}u,\wJ_{w_l}u\ra_{V^{0,8}}
=
-\varepsilon\la \wJ_{w_1}\wJ_{w_m}u,\wJ_{w_1}\wJ_{w_n}u\ra_{V^{0,8}}
\\
&= &
-\varepsilon\la w_1,w_1\ra_{0,8}\la w_m,w_n\ra_{0,8}\delta_{m,n},
\end{eqnarray*}
where the value of $\varepsilon$ is the same as in~\eqref{eqn:C}, due to the same number of permutations and the equalities in~\eqref{eq:bas80}. Since 
$$
-\varepsilon\la z_1,z_1\ra_{8,0}\la z_m,z_n\ra_{8,0}\delta_{m,n}=-\varepsilon\la w_1,w_1\ra_{0,8}\la w_m,w_n\ra_{0,8}\delta_{m,n}
$$
we obtain that $C[x_i,x_j]=[A(x_i),A(x_j)]$ and finish the proof. Note that the map $A^{\tau}$ is given by
$$
A^{\tau}(u)=v, \quad A^{\tau}(y_j)=-x_j,\ j=2,\ldots, 8,\quad A^{\tau}(y_j)=-x_j,\ j=9,\ldots,16.
$$

Remark that we can also construct the isomorphism $\Phi=A\oplus C$ with the same $C$ declaring $A(v)=cu$ with any $c\in\mathbb R\setminus\{0\}$.
\end{proof}

We need a technical lemma.
 
\begin{theorem}\label{th:51-62}
The Lie algebras $\mathcal N_{r,s}$ and $\mathcal N_{s,r}$ are
isomorphic for the values 
of indices $(r,s)\in \{(5,2),(5,1)\}$ and $\{(6,2),(6,1)\}$.
\end{theorem}
\begin{proof} {\sc Case $(r,s)=(5,2)$}.
The minimal admissible module $V^{5,2}$ is isometric to $\mathbb R^{8,8}$. We fix two mutually commuting isometric involutions and complementary anti-isometric operators
$$
P_1=J_{z_1}J_{z_2}J_{z_3}J_{z_4},\quad P_2=J_{z_1}J_{z_2}J_{z_6}J_{z_7},
\qquad R_1=J_{z_4}J_{z_6}, \quad R_2=J_{z_5}J_{z_6}.
$$ 
We also define a quaternion structure 
\begin{equation}\label{eq:quat52}
\bi=J_{z_1}J_{z_2},\quad \bj=J_{z_1}J_{z_3}J_{z_5}J_{z_7},\quad \bk=J_{z_2}J_{z_3}J_{z_5}J_{z_7}.
\end{equation}
\begin{table}[h]
\center\caption{Commutation relations of operators on $V^{5,2}$}
\begin{tabular}{|c||c|c|c|c|c|c|c|c|c|c|c|c|}
\hline
\ &$J_{z_1}$&$J_{z_2}$&$J_{z_3}$&$J_{z_4}$&$J_{z_5}$&$J_{z_6}$&$J_{z_7}$&$R_1$&$R_2$&$\bi$&$\bj$&$\bk$
\\
\hline\hline
$P_1$&-1&-1&-1&-1&1&1&1&-1&1&1&1&1
\\
\hline
$P_2$&-1&-1&1&1&1&-1&-1&\ &-1&1&1&1
\\
\hline\hline
\end{tabular}\label{t:52}
\end{table}
Table~\ref{t:52} shows that all the spaces $E^{I}=\bigcap\limits_{j=1}^{2}E^{k_j}_{P_j}$, with multi indices $I=(k_1,k_2)$, $k_j\in \{1,-1\}$, are neutral 4-dimensional spaces invariant under the action of quaternion structure~\eqref{eq:quat52}, which allows to find a convenient basis of $V^{5,2}$. Let $E^1=E^{1,1}$ and $v\in E^{1}$ with $\la v,v\ra_{V^{5,2}}=1$. Then the basis 
$\{x_1=v,\ x_2=\bi(v),\ x_3=\bj(v),\ x_4=\bk(v)\}$ for $E^1$
is orthonormal by Lemma~\ref{orthogonal}, where we set $\mathbb J=\bj$. 
Analogous calculations we make for the Lie algebra $\mathcal N_{2,5}$. The mutually commuting isometric involutions and the anti-isometric complementary operators are
$$
\wP_1=\wJ_{w_1}\wJ_{w_2}\wJ_{w_3}\wJ_{w_4},\quad \wP_2=\wJ_{w_1}\wJ_{w_2}\wJ_{w_6}\wJ_{w_7},\qquad
\widetilde R_1=\wJ_{w_1},\quad \wR_2=\wJ_{w_5}\wJ_{w_6}.
$$ 
We also fix the quaternion structure on $V^{2,5}$
$
\wbi=\wJ_{w_1}\wJ_{w_2}$, $\wbj=\wJ_{w_1}\wJ_{w_3}\wJ_{w_5}\wJ_{w_7}$, and $\wbk=-\wJ_{w_2}\wJ_{w_3}\wJ_{w_5}\wJ_{w_7}$.
The orthonormal basis of $V^{2,5}$ is produced from a vector $u\in\wE^{1,1}$ with $\la u,u\ra_{V^{2,5}}=1$ by action of this quaternion structure. 

Let us assume that there is an isomorphism $\Phi\colon\mathcal N_{5,2}\to\mathcal N_{2,5}$,
$\Phi=A\oplus C$, $A\colon V^{5,2}\to V^{2,5}$, where we define $C$ by
$
C(z_j)=w_j$, $C^{\tau}(w_j)=-z_j$, $j=1,\ldots,7$.
Then according to Lemma~\ref{relation between volume form} the map $A\colon V^{5,2}\to V^{2,5}$ has to satisfy the relations
\begin{equation*}\label{eq:Pijl0}
AP_j=\wP_jA,\quad A\bi=-\wbi A,\quad A\bj=\wbj A,\quad A\bk=-\wbk A.
\end{equation*}
Thus, we apply Corollary~\ref{rem:important} and set $A=\oplus A_I$. To construct  $A_{1}\colon E^{1}\to\wE^{1}$ we write $A_{1}(v)=(a_1+a_2\wbi+a_3\wbj+a_4\wbk)u$, $a_j\in\mathbb R$. Since 
$$
A_{1}(x_2)=  -\wbi A_{1}(v),
\quad
A_{1}(x_3) =  \wbj A_{1}(v),
\quad
A_{1}(x_4)= -\wbk A_{1}(v),
$$
we obtain the matrices for the map $A_{1}$ and $A_1^{\tau}$:
\begin{equation}\label{eq:A11_1}
A_{1}=
\begin{pmatrix}
a_{1}&a_2&-a_3&a_4\\
a_{2}&-a_1&a_4&a_3\\
a_{3}&a_4&a_1&-a_2\\
a_{4}&-a_3&-a_2&-a_1
\end{pmatrix},\qquad
A_{1}^{\tau}=
\begin{pmatrix}
a_{1}&a_2&-a_3&-a_4\\
a_{2}&-a_1&-a_4&a_3\\
a_{3}&-a_4&a_1&-a_2\\
-a_{4}&-a_3&-a_2&-a_1
\end{pmatrix}.
\end{equation}
To find relations between $a_j$, we observe that $J_{z_5}$ preserves $E^{1}$ and $\wJ_{w_5}$ preserves
$\wE^{1}$ and therefore they have to satisfy the relation $A^{\tau}_{1}\wJ_{w_5}A_{1}=-J_{z_5}$. In order to calculate the matrices for $J_{z_5}$ and $\wJ_{w_5}$ we observe that the isometric involution $T=J_{z_1}J_{z_2}J_{z_5}$ commuts with $P_1$ and $P_2$ and $E^1\cap E^1_{T}=E^1$. Therefore $Tv=v$ and we obtain
$J_{z_5}v=-\bi v$. To find the matrix $\wJ_{w_5}$ we note that the isometric involution $\widetilde T=\wJ_{w_1}\wJ_{w_3}\wJ_{w_7}$ commutes with $\wP_j$, $j=1,2$ and $\wE^1=E^1_{\widetilde T}\oplus E^{-1}_{\widetilde T}$, where the eigenspaces $E^1_{\widetilde T}$ and $E^{-1}_{\widetilde T}$ of $\widetilde T$ are neutral. Thus, we can assume that $\widetilde Tu=u$, which leads to $\wJ_{w_5}u=\wbj u$. Then 
\begin{equation}\label{eq:A11-2}
A^{\tau}_{1}\wJ_{w_5}A_{1}=A_1^{\tau}\begin{pmatrix}
0&0&1&0\\0&0&0&1\\1&0&0&0\\0&1&0&0
\end{pmatrix}
A_1
=
-\begin{pmatrix}0&1&0&0\\-1&0&0&0\\0&0&0&1\\0&0&-1&0\end{pmatrix}
=-J_{z_5}.
\end{equation}
 Thus, equation~\eqref{eq:A11-2} leads to three relations
$$
-2a_2a_3+2a_1a_4=-1,\quad a_1a_2+a_3a_4=0,\quad a_1^2-a_2^2+a_3^2-a_4^2=0,
$$
giving the solution
\begin{equation}\label{eq:solA}
a_2=a_3,\quad a_1=-a_4,\quad a_1^2+a_3^2=\frac{1}{2}.
\end{equation}
The operator $T=J_{z_1}J_{z_2}J_{z_5}$ leaves invariant subspace $E^1$ and therefore we have to check the equality $A_1J_{z_1}J_{z_2}J_{z_5}A_1^{\tau}=-\wJ_{w_1}\wJ_{w_2}\wJ_{w_5}$. We calculate 
$$
A_1J_{z_1}J_{z_2}J_{z_5}A_1^{\tau}=-\wJ_{w_1}\wJ_{w_2}A_1J_{z_5}A_1^{\tau}=-\wJ_{w_1}\wJ_{w_2}\wJ_{w_5},
$$
where we used the relations $A_1\bi=-\wbi A_1$ and $A_{1}J_{z_5}A^{\tau}_{1}=\wJ_{w_5}$.

To construct the remaining parts $A_I$ of the map $A$, we use Theorem~\ref{th:general} and the maps
$$
G_{1,-1}=J_{z_6}\colon E^1\to E^{1,-1},\quad G_{-1,1}=J_{z_3}\colon E^1\to E^{-1,1},\quad G_{-1,-1}=J_{z_1}\colon E^1\to E^{-1,-1}.
$$
Observe that the solution~\eqref{eq:solA} shows that $A_{1}v$ is a null vector. 
\\

{\sc Case $(r,s)=(5,1)$}. 
In this case we change the arguments and use the isomorphism
between $\mathcal{N}_{5,2}$ and $\mathcal{N}_{2,5}$.
Assume that the map $
\Phi=A\oplus C\colon\mathcal{N}_{5,2}\to\mathcal{N}_{2,5}$, and $C(z_k)=w_k$
is a Lie algebra isomorphism. Recall that minimal admissible modules $V^{5,2},V^{2,5}$ and $V^{5,1},V^{1,5}$ are irreducible and isometric to $\mathbb R^{8,8}$.
The natural inclusions $\mathbb{R}^{5,1}\subset \mathbb{R}^{5,2}$ and 
$\mathbb{R}^{1,5}\subset \mathbb{R}^{2,5}$ 
define the Clifford action of $\Cl_{5,1}$ and $\Cl_{1,5}$ on $V^{5,1}$ and $V^{1,5}$, respectively by restrictions of the Clifford action of $\Cl_{5,2}$ and $\Cl_{2,5}$.

Let $\pi_{-}:\mathbb{R}^{5,2}\to\mathbb{R}^{5,1}$ be the projection
map defined by
\[
z_1\mapsto z_1,~\,~\ldots,~\,~z_{6}\mapsto z_{6}, \quad z_{7}\mapsto 0,
\]
and let
$\pi_{+}:\mathbb{R}^{2,5}\to\mathbb{R}^{1,5}$ be the projection
defined by
\[
w_1\mapsto w_1,~\,\ldots,\,~w_{6}\mapsto w_6, \quad w_{7}\mapsto 0.
\]
Then the map 
\[
\Id\oplus\,\pi_{-}\colon\mathcal N_{5,2}=V^{5,2}\oplus\mathbb{R}^{5,2}\to \mathcal N_{5,1}=V^{5,1}\oplus\mathbb{R}^{5,1}
\]
is a Lie algebra homomorphism with kernel $K_{-}=\spn\{z_7\}$.
Also $\Id\oplus\,\pi_{+}$ is a Lie algebra homomorphism
from $\mathcal{N}_{2,5}$ to $\mathcal{N}_{1,5}$ with kernel $K_{+}=\spn\{w_7\}$.
Then the isomorphism $\Phi$
induces an isomorphism $\overline{\Phi}$  
between $\mathcal{N}_{5,1}$ and $\mathcal{N}_{1,5}$ by
\[
\begin{CD}
\{0\}@>>> K_{-} @>>> \mathcal{N}_{5,2} @>{\Id\oplus\,\pi_{-}}>> \mathcal{N}_{5,1}@>>>\{0\}\\
@.        @VV{C}V @V{\Phi}VV @VV{\overline{\Phi}}V\\                          
\{0\}@>>> K_{+}@>>> \mathcal{N}_{2,5}@>{\Id\oplus\,\pi_{+}}>> \mathcal{N}_{1,5}@>>>\{0\}.
\end{CD}
\]
Hence the Lie algebras $\mathcal{N}_{5,1}$ and $\mathcal{N}_{1,5}$ are isomorphic. 

The Lie algebra isomorphism $\bar\Phi\colon\mathcal{N}_{5,1}\to\mathcal{N}_{1,5}$ can be also induced by the isomorphism $\Phi\colon\mathcal{N}_{6,1}\to\mathcal{N}_{1,6}$, which we will construct later in this theorem.
\\

{\sc Case $(r,s)=(6,2)$}.
The minimal admissible modules, that are also irreducible, of $\Cl_{6,2}$ and $\Cl_{2,6}$ are isometric to $\mathbb{R}^{16,16}$.  
We fix mutually commuting isometric involutions 
\[
P_1=J_{z_1}J_{z_2}J_{z_3}J_{z_4},\quad P_2=J_{z_1}J_{z_2}J_{z_5}J_{z_6},\quad
P_{3}=J_{z_1}J_{z_2}J_{z_7}J_{z_8},
\]
on $V^{6,2}$ and the complementary anti-isometric operators
$
R_1=J_{z_3}J_{z_7}$, $R_2=J_{z_5}J_{z_7}$, $R_3=J_{z_7}$.
Define the quaternion structure
$
\bi=J_{z_1}J_{z_2}$, $\bj=J_{z_1}J_{z_3}J_{z_5}J_{z_7}$, $\bk=J_{z_2}J_{z_3}J_{z_5}J_{z_7}$.
Each of the space $E^{I}=\bigcap\limits_{j=1}^3 E_{P_j}^{k_j}$ is isometric to $\mathbb R^{2,2}$ according to Table~\ref{t:62}.
\begin{table}[h]
\center\caption{Commutation relations of operators on $V^{6,2}$}
\begin{tabular}{|c||c|c|c|c|c|c|c|c|c|c|c|c|c|c|}
\hline
\ &$J_{z_1}$&$J_{z_2}$&$J_{z_3}$&$J_{z_4}$&$J_{z_5}$&$J_{z_6}$&$J_{z_7}$&$J_{z_8}$&$R_1$&$R_2$&$R_3$&$\bi$&$\bj$&$\bk$
\\
\hline\hline
$P_1$&-1&-1&-1&-1&1&1&1&1&-1&1&1&1&1&1
\\
\hline
$P_2$&-1&-1&1&1&-1&-1&1&1&\ &-1&1&1&1&1
\\
\hline
$P_3$&-1&-1&1&1&1&1&-1&-1 &\ &\ &-1&1&1&1
\\
\hline\hline
\end{tabular}\label{t:62}
\end{table}
Denote $E^1=\cap_{j=1}^{3}E^1_{P_j}$ and fix a vector $v\in E^{1}$ such that
$\la v,v\ra_{V^{6,2}}=1$. Likewise we choose
\[
\wP_1=\wJ_{w_1}\wJ_{w_2}\wJ_{w_3}\wJ_{w_4},\quad \wP_2=\wJ_{w_1}\wJ_{w_2}\wJ_{w_5}\wJ_{w_6},\quad
\wP_{3}=\wJ_{w_1}\wJ_{w_2}\wJ_{w_7}\wJ_{w_8}\quad\text{involutions},
\]
$$
\widetilde R_1=\wJ_{w_1},\quad \widetilde R_2=\wJ_{w_5},\quad \widetilde R_3=\wJ_{w_1}\wJ_{w_3}\wJ_{w_5}\quad\text{complementary operators},
$$
$$
\wbi=\wJ_{w_1}\wJ_{w_2},\quad \wbj=\wJ_{w_1}\wJ_{w_3}\wJ_{w_5}\wJ_{w_7},\quad\wbk=-\wJ_{w_2}\wJ_{w_3}\wJ_{w_5}\wJ_{w_7}\quad\text{quaternion structure}.
$$
The table of commutations is preserved if we change $J_{z_k}$ to $\wJ_{w_k}$. 

Assume that there is a Lie algebra isomorphism
$$
\Phi=A\oplus C\colon\mathcal N_{6,2}\to\mathcal N_{2,6}, \qquad C(z_k)=w_k, \quad C^{\tau}(w_k)=-z_k.
$$ 
Since $AP_j=\wP_jA$ we have $A=\oplus A_I$, $A_I\colon E^I\to\wE^I$, and 
\begin{equation}\label{eq:q62}
A_I\bi=-\wbi A_I,\quad A_I\bj=\wbj A_I,\quad A_I\bk=-\wbk A_I\quad\text{for any}\quad I=(k_1,k_2,k_3)
\end{equation}
by Corollary~\ref{rem:important}. The map $A_{1}\colon E^1\to \wE^1$ is defined by the relation
$A_{1}v=(a_1+a_2\wbi+a_3\wbj+a_4\wbk)u$ and~\eqref{eq:q62} for $I=(1,1,1)$. All other operators leaving the space $E^1$ invariant are linear combination of quaternion structure.

The maps $A_I$ for other multi-indices $I$ are constructed by Theorem~\ref{th:general} by making use of the following maps 
\begin{equation*}
\begin{array}{lclcllll}
&G_{1,1,-1} &= &J_{z_7}\colon E^1\to E^{1,1,-1},\qquad
& G_{1,-1,1}&=&J_{z_5}\colon E^1\to E^{1,-1,1},
\\
&G_{1,-1,-1} &=&J_{z_1}J_{z_3}\colon E^1\to E^{1,-1,-1},
&G_{-1,1,1}&=&J_{z_3}\colon E^1\to E^{-1,1,1},
\\
&G_{-1,1,-1}& = &J_{z_1}J_{z_5}\colon E^1\to E^{-1,1,-1},
&G_{-1,-1,1}&=&J_{z_1}J_{z_7}\colon E^1\to E^{-1,-1,1},
\\
&G_{-1,-1,-1}&=&J_{z_1}\colon E^1\to E^{-1,-1,-1}.
\end{array}
\end{equation*}
\\

{\sc Case $(r,s)=(6,1)$}. 
The minimal admissible module $V^{6,1}$ 
is isometric to $\mathbb{R}^{8,8}$. The isometric involutions and the complementary anti-isometric operators are
$$
P_1=J_{z_1}J_{z_2}J_{z_3}J_{z_4},\quad P_2=J_{z_1}J_{z_2}J_{z_5}J_{z_6},\qquad
R_1=J_{z_1}J_{z_7},\quad R_2=J_{z_5}J_{z_7},\quad R_3=J_{z_7}.
$$
The quaternion structure is 
$\bi=J_{z_1}J_{z_2}$, $\bj=J_{z_1}J_{z_3}J_{z_5}J_{z_7}$, and $\bk=J_{z_2}J_{z_3}J_{z_5}J_{z_7}$. 
\begin{table}[h]
\center\caption{Commutation relations of operators on $V^{6,1}$}
\begin{tabular}{|c||c|c|c|c|c|c|c|c|c|c|c|c|c|c|}
\hline
\ &$J_{z_1}$&$J_{z_2}$&$J_{z_3}$&$J_{z_4}$&$J_{z_5}$&$J_{z_6}$&$J_{z_7}$&$R_1$&$R_2$&$R_3$&$\bi$&$\bj$&$\bk$
\\
\hline\hline
$P_1$&-1&-1&-1&-1&1&1&1&-1&1&1&1&1&1
\\
\hline
$P_2$&-1&-1&1&1&-1&-1&1&\ &-1&1&1&1&1
\\
\hline
$T$&1&-1&1&-1&1&-1&-1&\ &\ &-1&-1&-1&1
\\
\hline\hline
\end{tabular}\label{t:61}
\end{table}
Let $v\in E^{1}=\cap_{j=1}^{2}E^{1}_{P_j}$ be such that $\la v,v\ra_{V^{6,1}}=1$. Then
$
\{v,\ \bi(v),\ \bj(v),\ \bk(v)\}$ is an orthonormal basis for $E^{1}$.
To show that the basis is orthogonal we argue as following. The operator $T=J_{z_1}J_{z_3}J_{z_5}$ is an isometry, commutes with $P_j$, $j=1,2$, and therefore it decomposes the space $E^1$ on two orthogonal subspaces: $E^1=\spn\{v,\bk(v)\}\oplus\spn\{\bi(v),\bj(v)\}$, see~\eqref{eq:isomP}. If it is necessary we change $v$ to $\tilde v$ by Lemma~\ref{orthogonal}, where we set $\mathbb J=\bk$.

Analogously, we fix the involutions and the anti-isometric complementary operators
\[
\wP_1=\wJ_{w_1}\wJ_{w_2}\wJ_{w_3}\wJ_{w_4},\quad
\wP_2=\wJ_{w_1}\wJ_{w_2}\wJ_{w_5}\wJ_{w_6},\qquad
\wR_1=\wJ_{w_1}, \quad \wR_2=\wJ_{w_5}
\] 
acting on $V^{1,6}$. Set the quaternion structure $\wbi=\wJ_{w_1}\wJ_{w_2}$, $\wbj=\wJ_{w_1}\wJ_{w_3}\wJ_{w_5}\wJ_{w_7}$, and $\wbk=-\wJ_{w_2}\wJ_{w_3}\wJ_{w_5}\wJ_{w_7}$.
Choose a vector $u\in \wE^{1}$ such that $\la u,u\ra_{V^{1,6}}=1$. By making use the quaternion structure we form an orthonormal basis on space $\wE^1$.

Is $\Phi=A\oplus C\colon\mathcal N_{6,1}\to\mathcal N_{1,6}$ is an isomorphism, then it has to satisfy Lemma~\ref{relation between volume form}. We construct the map  $A\colon V^{6,1} \to V^{1,6}$, $A=\oplus A_I$, by blocks $A_I\colon E^{I}\to\wE^{I}$. Put $A_{1}(v)=(a_1+a_2\wbi+a_3\wbj+a_4\wbk)u\not=0$. Then by the action of quaternion structure
$$
A_{1}\bi(v)=-\wbi A_{1}(v),\quad A_{1} \bj(v)=\wbj A_{1}(v),\quad A_{1} \bk(v)=-\wbk A_1(v), 
$$
we find all the coefficients of $A_1$. The map $A_1$ must satisfies
the condition $A_{1}^{\tau}\wJ_{w_7}A_{1}=-J_{z_7}$. Arguing as in the case of the construction of the isomorphism $\mathcal N_{5,2}\cong\mathcal N_{2,5}$, we fined that $J_{z_7}v=-\bj(v)$ and $\wJ_{w_7}u=-\wbi(v)$ and the matrices $A$ and $A^{\tau}$ are given by~\eqref{eq:A11_1}. thus we obtain the solution
$
a_2=-a_3$, $a_1=a_4$, $a_1^2+a_3^2=\frac{1}{2}$.
We finish the proof by applying Theorem~\ref{th:general} and using the maps
$$
G_{1,-1}=J_{z_5}\colon E^1\to E^{1,-1},
\quad
G_{-1,1}=J_{z_3}\colon E^1\to E^{-1,1},\quad
G_{-1,-1}=J_{z_1}\colon E^1\to E^{-1,-1}.
$$
\end{proof}

\begin{theorem}\label{th:automorphisms}
The Lie algebras $\mathcal N_{r,r}$, admit a Lie algebra automorphism $\Psi=A\oplus C\colon \mathcal N_{r,r}=V^{r,r}\oplus\mathbb R^{r,r}\to\mathcal N_{r,r}=V^{r,r}\oplus\mathbb R^{r,r}$ with $CC^{\tau}=-\Id$, if $r=1,2,4$.
\end{theorem}
\begin{proof} 
{\sc Case $\mathcal N_{1,1}$}. The minimal admissible module $V^{1,1}$ is isometric to $\mathbb R^{2,2}$. We choose the basis $\{v,\  J_{z_1}v,\  J_{z_2}v,\  J_{z_1}J_{z_2}v,\ \  z_1,\  z_2\}$ for $\mathcal N_{1,1}$.
Set $
C(z_1)=z_2$, $C(z_2)=z_1$, and $C^{\tau}(z_1)=-z_2$, $C^{\tau}(z_2)=-z_1$.
In order to satisfy Lemma~\ref{relation between volume form} we define
$$
A=\begin{pmatrix} 
1&0&0&0
\\
0&0&1&0
\\
0&1&0&0
\\
0&0&0&1
\end{pmatrix}
\quad\text{and}\quad
A^{\tau}=\begin{pmatrix} 
1&0&0&0
\\
0&0&-1&0
\\
0&-1&0&0
\\
0&0&0&1
\end{pmatrix}.
$$

{\sc Case $\mathcal N_{2,2}$}.  The minimal admissible module $V^{2,2}$ is isometric to $\mathbb R^{4,4}$. We fix the isometric involution $P=J_{z_1}J_{z_2}J_{z_3}J_{z_4}$ and choose the basis
$$
\left\{
\begin{array}{lllll}
&x_1=v,\ \ &x_2=J_{z_1}v,\ \ &x_3=J_{z_2}v,\ \ &x_4=J_{z_1}J_{z_2}v,
\\
&x_5=J_{z_3}v,\ \ &x_6=J_{z_4}v,\ \ &x_7=J_{z_1}J_{z_3}v,\ \ &x_8=J_{z_1}J_{z_4}v
\end{array}
\right\}\quad\text{for} \quad V^{2,2},
$$
where the vector $v$ is such that $Pv=v$ and $\la v,v,\ra_{V^{2,2,}}=1$. We first listed the positive vectors of the basis and then negative vectors and therefore the matrix for the metric is the standard one: the diagonal matrix $I_{4,4}$ with first four diagonal entries 1 and the last four diagonal entries (-1). Satisfying the conditions of Lemma~\ref{relation between volume form} we set
$$
C=
\begin{pmatrix}
0&0&0&1
\\
0&0&1&0
\\
0&1&0&0
\\
1&0&0&0
\end{pmatrix}
\quad\text{and}\quad
A=
\begin{pmatrix}
1&0&0&0&0&0&0&0
\\
0&0&0&0&0&1&0&0
\\
0&0&0&0&1&0&0&0
\\
0&0&0&-1&0&0&0&0
\\
0&0&1&0&0&0&0&0
\\
0&1&0&0&0&0&0&0
\\
0&0&0&0&0&0&1&0
\\
0&0&0&0&0&0&0&1
\end{pmatrix}.
$$

{\sc Case $\mathcal N_{4,4}$}. The minimal admissible module $V^{4,4}$ is isometric to $\mathbb R^{8,8}$. We fix the isometric involutions 
$$
P_1=J_{z_1}J_{z_2}J_{z_3}J_{z_4},\ P_2=J_{z_1}J_{z_2}J_{z_5}J_{z_6},\ P_3=J_{z_1}J_{z_2}J_{z_7}J_{z_8},\ P_4=J_{z_1}J_{z_3}J_{z_5}J_{z_7}.
$$ We choose $v$ such that $P_j(v)=v$, $j=1,2,3,4$, and $\la v,v\ra_{V^{2,2,}}=1$ and construct the basis for $V^{4,4}$, where we place first the positive vectors and then the negative ones.
\begin{equation}\label{eq:basis44}
\left\{
\begin{array}{lllll}
&x_1=v,\ \ &x_2=J_{z_1}v,\ \ &x_3=J_{z_2}v,\ \ &x_4=J_{z_3}v,
\\
&x_5=J_{z_4}v,\ \ &x_6=J_{z_1}J_{z_2}v,\ \ &x_7=J_{z_1}J_{z_3}v,\ \ &x_8=J_{z_1}J_{z_4}v,
\\
&x_9=J_{z_5}v,\ \ &x_{10}=J_{z_6}v,\ \ &x_{11}=J_{z_7}v,\ \ &x_{12}=J_{z_8}v,
\\
&x_{13}=J_{z_1}J_{z_5}v,\ \ &x_{14}=J_{z_1}J_{z_6}v,\ \ &x_{15}=J_{z_1}J_{z_7}v,\ \ &x_{16}=J_{z_1}J_{z_8}v.
\end{array}
\right\}
\end{equation}
Now we define $C$ as before by $C(z_j)=z_{9-j}$ and $C^{\tau}(z_j)=-z_{9-j}$, $j=1,\ldots,8$. The map $A$ is also similar to the previous cases. Namely, we set
$$
A(v)=v,\quad A(J_{z_j}v)=J_{C(z_j)}v,\quad \text{and}\quad
A^{\tau}(v)=v,\quad A^{\tau}(J_{z_j}v)=-J_{C^{\tau}(z_j)}v.
$$
Then also 
$$
A(J_{z_j}J_{z_k}v)=
\begin{cases}
-J_{C(z_j)}J_{C(z_k)}v,\quad&\text{if}\quad J_{z_j}J_{z_k}v\ \text {is positive},
\\
J_{C(z_j)}J_{C(z_k)}v,\quad&\text{if}\quad J_{z_j}J_{z_k}v\ \text {is negative},
\end{cases}
$$
and analogously for $A^{\tau}$. All conditions of Lemma~\ref{relation between volume form} are satisfied.
\end{proof}

\begin{theorem}\label{periodicity}
If the Lie algebra $\mathcal{N}_{r,s}$ is isomorphic to the Lie algebra $\mathcal{N}_{s,r}$, then 
\begin{itemize}
\item[1.]{the Lie algebras $\mathcal{N}_{r,s+8k}$ and $\mathcal{N}_{s+8k,r}$ 
are isomorphic;}
\item[2.]{the Lie algebras $\mathcal{N}_{r+8k,s}$ and $\mathcal{N}_{s,r+8k}$ 
are isomorphic;}
\item[3.]{the Lie algebras $\mathcal{N}_{r+4k,s+4k}$ and $\mathcal{N}_{s+4k,r+4k}$ 
are isomorphic.}
\end{itemize}
for any $k=1,2,\ldots$.
\end{theorem}
\begin{proof}
Recall that if $(V^{r,s},\la .\,,.\ra_{V^{r,s}})$ is a minimal admissible module, then the products
$$
V^{r,s}\otimes V^{0,8},\quad V^{r,s}\otimes V^{8,0},\quad V^{r,s}\otimes V^{4,4}
$$
are minimal admissible if $V^{0,8}$, $V^{8,0}$, and $V^{4,4}$ are minimal admissible modules, see~\cite{FM}. The scalar product on $V^{r,s+8}$ is given by the product of bilinear symmetric forms on $V^{r,s}$ and $V^{0,8}$. Analogously, for the tensor products with $V^{8,0}$ and $V^{4,4}$. The representations are constructed as follows. Let $\{\zeta_1,\ldots, \zeta_8\}$ be an orthonormal bases for $\mathbb R^{0,8}$, $\mathbb R^{8,0}$, or $\mathbb R^{4,4}$, and $\bar J_{\zeta_{\alpha}}$, $\alpha=1,\ldots,8$ be the respective representations. Let $J_{z_j}$, $j=1,\ldots,r+s$ be representations of an orthonormal basis for $\mathbb R^{r,s}$. We denote by $\Omega^{0,8}=\prod_{\alpha=1}^{8}\bar J_{\zeta_{\alpha}}$ the volume form for $\Cl_{0,8}$ and analogously for others Clifford algebras. Set
\begin{eqnarray*}
\hat J_{z_j} & = & J_{z_j}\otimes\Omega^{0,8} \quad\text{for}\quad j=1,\ldots,r+s,
\\
\hat J_{\zeta_\alpha} & = & \Id_{V^{r,s}}\otimes\, \bar J_{\zeta_\alpha} \quad\text{for}\quad \alpha=1,\ldots,8.
\end{eqnarray*}
Then the maps $\hat J_{z_j}$ and $\hat J_{\zeta_\alpha}$ are representations of an orthonormal basis for $\mathbb R^{r,s+8}$ as it was shown in~\cite{FM}. If we substitute the volume form $\Omega^{0,8}$ by $\Omega^{8,0}$ or $\Omega^{4,4}$, then we obtain the representations for the basis vectors of $\mathbb R^{r+8,s}$ and $\mathbb R^{r+4,s+4}$, respectively.

We start from the proof of the first case, since the rest can be proven similarly. Let $\Phi=A\oplus C\colon\mathcal N_{r,s}\to\mathcal N_{s,r}$ and $\bar \Phi=\bar A\oplus \bar C\colon \mathcal N_{0,8}\to\mathcal N_{8,0}$ be the Lie algebra isomorphisms, with $A\colon V^{r,s}\to V^{s,r}$ and $\bar A\colon V^{0,8}\to V^{8,0}$. Let $P_j$, $j=1,\ldots, p$ and $Q_k$, $k=1,2,3,4$, be mutually commuting isometric involutions on $V^{r,s}$ and $V^{0,8}$, respectively. Then $\hat P_j=P_j\otimes \Id$ and $\hat Q_k=\Id\otimes Q_k$ are mutually commuting isometric involutions on $V^{r,s}\otimes V^{0,8}$.
Let $E^1=\cap_{j=1}^p E^1_{P_j}$, $F^1=\cap_{k=1}^{4}E^1_{Q_j}$ and 
$$
v\in E^1,\ \ \la v,v\ra_{V^{r,s}}=1,\quad\text{  }\quad u\in F^1,\ \ \la u,u\ra_{V^{0,8}}=1.
$$
We have $\hat P_j(v\otimes u)=P_j(v)\otimes \Id(u)=v\otimes u$,  and $\hat Q_k(v\otimes u)=\Id(v)\otimes Q_k(u)=v\otimes u$. Therefore $v\otimes u\in E^1\otimes F^1$ and $\la v\otimes u,v\otimes u \ra_{V^{r,s+8}}=1$. The linear map $\hat A\colon V^{r,s+8}\to V^{s+8,r}$ of a Lie algebra isomorphism $\hat \Phi=\hat A\oplus\hat C\colon\mathcal N_{r,s+8}\to \mathcal N_{s+8,r}$ should satisfy Lemma~\ref{relation between volume form}. Thus we could apply Corollary~\ref{rem:important} and start the construction of $\hat A$ from the map $\hat A_{1,1}\colon E^1\otimes F^1\to \wE^1\otimes \wF^1$ and then extend it to an arbitrary $\hat A_{IJ}\colon E^I\otimes F^J\to \wE^I\otimes \wF^J$. Observe that if $\prod J_{z_k}$ leaves invariant the space $E^1$, then the product $\prod \hat J_{z_j}$ leaves the space $E^1\otimes F^1$ invariant and, analogously, if $\prod J_{\zeta_\alpha}$ leaves invariant the space $F^1$, then $\prod \hat J_{\zeta_\alpha}$ leaves the space $E^1\otimes F^1$ invariant. As a consequence, we also obtain that the space $E^1\otimes F^1$ will be invariant under the action of $\prod \hat J_{z_j}\prod \hat J_{\zeta_\alpha}$.

We denote by $\{z_1,\ldots,r_{r+s},\zeta_1,\ldots,\zeta_8\}$ an orthonormal basis for $\mathbb R^{r,s+8}$ with 
$$\{z_1,\ldots,z_r\}\  \text{positive and}\  \{z_{r+1},\ldots,z_{r+s},\zeta_1,\ldots,\zeta_8\}\  \text{negative elements}.
$$ Then, let $\{w_{r+s},\ldots,w_{1},\omega_8,\ldots,\omega_1\}$ be an orthonormal basis for $\mathbb R^{s+8,r}$ with 
$$\{w_{r+s},\ldots,w_{r+1},\omega_8,\ldots\omega_{1}\}\ \text{positive vectors and}\ \{w_{r},\ldots,w_{1}\}\ \text{negative vectors}.
$$ We let the map $\hat C\colon \mathbb R^{r,s+8}\to\mathbb R^{s+8,r}$ act on the basis by the following
$$
\hat C(z_j)=w_j,\  \hat C^{\tau}(w_j)=-z_j,\  j=1,\ldots,r+s,\ \ \hat C(\zeta_\alpha)=\omega_{\alpha},\ \hat C^{\tau}(\omega_\alpha)=-\zeta_{\alpha},\  \alpha=1,\ldots,8.
$$
We define the map $\hat A_{1,1}\colon E^1\otimes F^1\to \wE^1\otimes \wF^1$ by its action on different type of products of $\hat J_{z_j}$ and $\hat J_{\zeta_{\alpha}}$. Recall that 
$\prod_{j=1}^{p} \hat J_{z_j}\prod_{\alpha=1}^q\hat J_{\zeta_{\alpha}}=\prod_{j=1}^{p}J_{z_j}\otimes (\Omega^{0,8})^{p}\prod_{\alpha=1}^qJ_{\zeta_{\alpha}}$.
Then we define $\hat A_{1,1}\prod_{j=1}^{p} \hat J_{z_j}\prod_{\alpha=1}^q\hat J_{\zeta_{\alpha}}=$
\begin{equation}\label{eq:tensor}
\begin{cases}
A_1\prod\limits_{j=1}^{p}J_{z_j}\otimes\bar A_1(\Omega^{0,8})^{p}\prod\limits_{\alpha=1}^qJ_{\zeta_{\alpha}}
=(-1)^{m+k}\prod\limits_{j=1}^{p}\wJ_{\hat C(z_j)}(A_1^{\tau})^{-1}\otimes\Omega^{8,0}\prod\limits_{\alpha=1}^q\wJ_{\hat C(\zeta_{\alpha})}(\bar A_1^{\tau})^{-1},
\\  
\qquad\qquad \text{if}\ p=2m+1, \ \  q=2k+1,
\\
\\
A_1\prod\limits_{j=1}^{p}J_{z_j}  \otimes\bar A_1(\Omega^{0,8})^{p}\prod\limits_{\alpha=1}^qJ_{\zeta_{\alpha}}
=(-1)^{m+k}\prod\limits_{j=1}^{p}\wJ_{\hat C(z_j)}(A_1^{\tau})^{-1}\otimes\Omega^{8,0}\prod\limits_{\alpha=1}^q\wJ_{\hat C(\zeta_{\alpha})}\bar A_1,
\\  
\qquad\qquad \text{if}\ p=2m+1, \ \  q=2k,
\\
\\
A_1\prod\limits_{j=1}^{p}J_{z_j} \otimes\bar A_1\prod\limits_{\alpha=1}^qJ_{\zeta_{\alpha}}
=(-1)^{m+k}\prod\limits_{j=1}^{p}\wJ_{\hat C(z_j)}A_1\otimes\prod\limits_{\alpha=1}^q\wJ_{\hat C(\zeta_{\alpha})}(\bar A_1^{\tau})^{-1},
\\  
\qquad\qquad  \text{if}\ p=2m, \ \  q=2k+1,
\\
\\
A_1\prod\limits_{j=1}^{p}J_{z_j} \otimes\bar A_1\prod\limits_{\alpha=1}^qJ_{\zeta_{\alpha}}
=(-1)^{m+k}\prod\limits_{j=1}^{p}\wJ_{\hat C(z_j)}A_1\otimes\prod\limits_{\alpha=1}^q\wJ_{\hat C(\zeta_{\alpha})}\bar A_1,
\\  
\qquad\qquad \text{if}\ p=2m, \ \  q=2k.
\end{cases}
\end{equation}
We also can write the transposed map $\hat A_{1,1}^{\tau}$ by
$\hat A_{1,1}^{\tau}\prod_{j=1}^{p} \hat J_{\hat C(z_j)}\prod_{\alpha=1}^q\hat J_{\hat C(\zeta_{\alpha})}=$
\begin{eqnarray*}
\begin{cases}
A_1^{\tau}\prod\limits_{j=1}^{p}\wJ_{\hat C(z_j)}\otimes\bar A_1^{\tau}(\Omega^{8,0})^{p}\prod\limits_{\alpha=1}^q\wJ_{\hat C(\zeta_{\alpha})}
=(-1)^{m+k}\prod\limits_{j=1}^{p}J_{z_j}A_1^{-1}\otimes\Omega^{0,8}\prod\limits_{\alpha=1}^qJ_{\zeta_{\alpha}}\bar A_1^{-1},
\\  
\qquad\qquad \text{if}\ p=2m+1, \ \  q=2k+1,
\\
\\
A_1^{\tau}\prod\limits_{j=1}^{p}\wJ_{\hat C(z_j)}  \otimes\bar A_1^{\tau}(\Omega^{8,0})^{p}\prod\limits_{\alpha=1}^q\wJ_{\hat C(\zeta_{\alpha})}
=(-1)^{m+k+1}\prod\limits_{j=1}^{p}J_{z_j}A_1^{-1}\otimes\Omega^{0,8}\prod\limits_{\alpha=1}^qJ_{\zeta_{\alpha}}\bar A_1^{\tau},
\\  
\qquad\qquad \text{if}\ p=2m+1, \ \  q=2k,
\\
\\
A_1^{\tau}\prod\limits_{j=1}^{p}\wJ_{\hat C(z_j)} \otimes\bar A_1^{\tau}\prod\limits_{\alpha=1}^q\wJ_{\hat C(\zeta_{\alpha})}
=(-1)^{m+k+1}\prod\limits_{j=1}^{p}J_{z_j}A_1^{\tau}\otimes\prod\limits_{\alpha=1}^qJ_{\zeta_{\alpha}}\bar A_1^{-1},
\\  
\qquad\qquad  \text{if}\ p=2m, \ \  q=2k+1,
\\
\\
A_1^{\tau}\prod\limits_{j=1}^{p}\wJ_{\hat C(z_j)} \otimes\bar A_1^{\tau}\prod\limits_{\alpha=1}^q\wJ_{\hat C(\zeta_{\alpha})}
=(-1)^{m+k}\prod\limits_{j=1}^{p}J_{z_j}A_1^{\tau}\otimes\prod\limits_{\alpha=1}^qJ_{\zeta_{\alpha}}\bar A_1^{\tau},
\\  
\qquad\qquad \text{if}\ p=2m, \ \  q=2k.
\end{cases}
\end{eqnarray*}
Then the maps $G_{IJ}=G_I\otimes \bar G_J\colon E^1\otimes F^1\to E^I\otimes F^J$ will be used to define 
$\hat A_{IJ}\colon E^I\otimes F^J\to\wE^I\otimes \wF^J$. Namely
\begin{equation}\label{eq:tensor1}
\hat A_{IJ}=
\begin{cases}
(-1)^m\widetilde G_{IJ}(\hat A_{1,1}^{-1})^{\tau}G_{IJ}^{-1}\quad&\text{if}\quad p=2m+1,
\\
(-1)^m\widetilde G_{IJ}\hat A_{1,1}G_{IJ}^{-1}\quad&\text{if}\quad p=2m,
\end{cases}
\end{equation}
and
 $$
\hat A_{IJ}^{\tau}=
\begin{cases}
(-1)^{m+1} G_{IJ}\hat A_{1,1}^{-1}\widetilde G_{IJ}^{-1}\quad&\text{if}\quad p=2m+1,
\\
(-1)^m G_{IJ}\hat A_{1,1}^{\tau}\widetilde G_{IJ}^{-1}\quad&\text{if}\quad p=2m.
\end{cases}
$$
Thus, we obtain that the map $\hat \Phi=\hat A\oplus \hat C$, with $\hat A=\oplus_{IJ}\hat A_{IJ}$, is a Lie algebra isomorphism from $\mathcal N_{r,s+8}$ to $\mathcal N_{s+8,r}$, according to Corollary~\ref{rem:important}. We recall that we can choose the following map $\bar A$: $\bar A_1v=u$ and since the spaces $F^J$ are one dimensional, the corresponding maps $G_J\colon F^1\to F^J$ are given by the basis~\eqref{eq:basis80}.

The third statement is proved analogously, where we change the map $\bar A\colon V^{0,8}\to V^{8,0}$ to the map $\bar A\colon V^{4,4}\to V^{4,4}$ constructed in Theorem~\ref{th:automorphisms}. Then we use the definitions~\eqref{eq:tensor} and~\eqref{eq:tensor1} to construct the isomorphism $\hat \Phi=\hat A\oplus\hat C\colon\mathcal N_{r+4,s+4}\to\mathcal N_{s+4,r+4}$ by tensor product, where we change the volume forms $\Omega^{0,8}$ and $\Omega^{8,0}$ to $\Omega^{4,4}$. The maps $G_J\colon F^1\to F^J$ are given by the basis~\eqref{eq:basis44}.
\end{proof}

\begin{rem}
The reader can recognise in the construction of $\hat A$ the $\mathbb Z^2$-graded tensor product. Indeed we write $A=A^0\oplus A^1$ and $\bar A=\bar A^{0}\oplus \bar A^1$, where $A^0$ and $\bar A^0$ act on the even product of generators $J_{z_j}$ and $A^1$ and $\bar A^1$ act on the odd product of generators. Then formula~\eqref{eq:tensor} can be written as follows
$$
\hat A_1=A_1\hat\otimes \bar A_1=(A_1\otimes \bar A_1)^0\oplus(A_1\otimes \bar A_1)^1=\big((A_1^0\otimes \bar A_1^0)\oplus(A_1^1\otimes \bar A_1^1)\big)\otimes\big((A_1^0\otimes \bar A_1^1)\oplus(A_1^1\otimes \bar A_1^0)\big).
$$
This is not surprising, according to the $\mathbb Z_2$-graded structure of Clifford algebra and the isomorphism
$
\Cl(\mathbb R^{r,s}\oplus \mathbb R^{p,q})\cong\Cl(\mathbb R^{r,s})\hat\otimes\Cl(\mathbb R^{p,q})$, based on the $\mathbb Z^2$-graded tensor product~$\hat\otimes$.
\end{rem}

\begin{theorem}\label{periodicity1}
The following is true:
\begin{itemize}
\item[1.]{the Lie algebras $\mathcal{N}_{r,r+8k}$ and $\mathcal{N}_{r+8k,r}$ 
are isomorphic for $r=1,2,4$;}
\item[2.]{the Lie algebras $\mathcal{N}_{r+4k,r+4k}$, $r=1,2,4$ admit an automorphism $\Psi=A\oplus C$ with $CC^{\tau}=-\Id$.}
\end{itemize}
\end{theorem}

\begin{proof}
The proof is literary the same as the proof of Theorem~\ref{periodicity}, where we need to change the existence of an isomorphism $\Phi=A\oplus C\colon\mathcal N_{r,s}\to \mathcal N_{s,r}$ to an automorphism $\Psi=A\oplus C\colon\mathcal N_{r,r}\to\mathcal N_{r,r}$.
\end{proof}


\subsection{Non-isomorphic Lie algebras}\label{sec:nonisom}


We start from a small technical observation.

\begin{lemma}\label{rem:contradiction}
Let $(V,\la.\,,.\ra_V)$ be a neutral space and $T$ a linear map on $V$ with the properties: $T^2=\Id$ and the scalar product $(x,y):=\la x,Ty\ra_V$ is positive definite. Then there is no linear map $S$ on $V$, such that 
$S=S^{\tau}$, where $S^{\tau}$ is transposed with respect to $\la.\,,.\ra_V$,
and $STS=-T$.
\end{lemma}
\begin{proof}
Let us assume that a linear map $S\colon V\to V$ such that $S=S^{\tau}$
and $STS=-T$ exists. Then $^tST=TS$, where $^tS$ is the transposition 
with respect to the positive definite scalar product $(.\,,.)$ and therefore
$$
-T=STS=T(^tS)TTS=T(^tS)S\quad\Longrightarrow\quad ^tSS=-\Id,
$$
which is a contradiction.
\end{proof}

\begin{theorem}\label{32-23}
The Lie algebras $\mathcal{N}_{r,s}$ and $\mathcal{N}_{s,r}$ for $(r,s)\in\{(3,1),\,(3,2),\,(3,7),\,(3,11)\}$ are not isomorphic.
\end{theorem}
\begin{proof}
{\sc Case $(r,s)=(3,1)$}.
The minimal admissible module $V^{3,1}$ is isometric to $\mathbb R^{4,4}$. We define the isometric involution $T=J_{z_1}J_{z_2}J_{z_3}$ and the orthonormal basis for $V^{3,1}$, starting from $v\in V^{3,1}$, $\la v,v\ra_{V^{3,1}}=1$, and $Tv=v$:
\begin{equation}\label{eq:b31}
\begin{array}{lllll}
x_1=v,\quad & x_2=J_{z_1}v,\quad & x_3=J_{z_2}v,\quad &x_4=J_{z_3}v,
\\
x_5=J_{z_4}v,\quad & x_6=J_{z_4}J_{z_1}v,\quad & x_7=J_{z_4}J_{z_2}v,\quad &x_8=J_{z_4}J_{z_3}v,
\end{array}
\end{equation}
with $\la x_k,x_k\ra_{V^{3,1}}=-\la x_{k+4},x_{k+4}\ra_{V^{3,1}}=1$, $k=1,\ldots,4$. Moreover $T(x_i)=x_i$, $i=1,2,3,4$ and $T(x_i)=-x_i$, $i=5,6,7,8$. 
Assume that there is an isomorphism $\Phi\colon\mathcal N_{3,1}\to \mathcal N_{1,3}$, $\Phi=A\oplus C$ such that $A\colon V^{3,1}\to V^{1,3}$ and $C(z_j)=w_j$. Then the map $\Phi^{\tau}\Phi=A^{\tau}A\oplus -\Id_{\mathbb R^{3,2}}\colon\mathcal N_{3,2}\to \mathcal N_{3,2}$ is an automorphism by Lemma~\ref{iso-form 2}. Denote $S=A^{\tau}A$ and obtain a contradiction 
as in Lemma~\ref{rem:contradiction} with $V=V^{3,1}$ and $T=J_{z_1}J_{z_2}J_{z_3}$.
\\

{\sc Case $(r,s)=(3,2)$}.
We consider mutually commuting isometric involutions and the complementary anti-isometric operators
$$
P=J_{z_1}J_{z_2}J_{z_4}J_{z_5},\quad T=J_{z_1}J_{z_2}J_{z_3}
\qquad
R_1=J_{z_5},\quad R_2=J_{z_1}J_{z_4}
$$
acting on $V^{3,2}$.
\begin{table}[h]
\center\caption{}
\begin{tabular}{|c||c|c|c|c|c|c|c|}
\hline
\ &$J_{z_1}$&$J_{z_2}$&$J_{z_3}$&$J_{z_4}$&$J_{z_5}$&$R_1$&$R_2$
\\
\hline\hline
$P$&-1&-1&1&-1&-1&-1&1
\\
\hline   
$T$&1&1&1&-1&-1& &-1
\\
\hline\hline
\end{tabular}\label{t:3,2}
\end{table}
In Table~\ref{t:3,2} we show commutation relations of the involutions, complementary operators, and the representation maps $J_{z_j}$. We conclude that the spaces $E^{1}_{P}$ and $E^{-1}_{P}$ are neutral. We pick up a vector $v\in E^{1}_{P}$, $\la v,v\ra_{V^{3,2}}=1$ and construct an orthonormal basis for $E^{1}_{P}$  
$$
x_1=v,\quad x_2=J_{z_1}J_{z_2}v,\quad x_3=J_{z_1}J_{z_4}v,\quad x_4=J_{z_2}J_{z_4}v
$$
with $\la x_i,x_i\ra_{V^{3,2}}=-\la x_{i+2},x_{i+2}\ra_{V^{3,2}}=1$, $i=1,2$.
Table~\ref{t:3,2} also shows that
$$
Tx_1=x_1,\quad Tx_2=x_2,\quad Tx_3=-x_3,\quad Tx_4=-x_4.
$$

Assuming now that there is an isomorphism $\Phi\colon\mathcal N_{3,2}\to \mathcal N_{2,3}$, $\Phi=A\oplus C$ such that $A\colon V^{3,2}\to V^{2,3}$ and $C(z_j)=w_j$, we obtain a contradiction by Lemma~\ref{rem:contradiction} with $V=E^1_P$, $S=A^{\tau}_1A_1$, and $T=J_{z_1}J_{z_2}J_{z_3}$.
\\

{\sc Case $(r,s)=(3,7)$}.
We define the mutually commuting involutions
$$
P_1=J_{z_1}J_{z_2}J_{z_5}J_{z_6},\quad P_2=J_{z_1}J_{z_2}J_{z_7}J_{z_8},\quad
P_3=J_{z_1}J_{z_2}J_{z_9}J_{z_{10}},\quad T=J_{z_1}J_{z_2}J_{z_3},
$$
and the complementary anti-isometric operators 
$$
R_1=J_{z_5},\quad R_2=J_{z_7},\quad R_3=J_{z_9},\quad R_4=J_{z_4},
$$
acting on $V^{3,7}$.
\begin{table}[h]
\center\caption{Commutation relations of operators on $V^{3,7}$}
\begin{tabular}{|c||c|c|c|c|c|c|c|c|c|c|c|c|c|c|c|}
\hline
\ &$J_{z_1}$&$J_{z_2}$&$J_{z_3}$&$J_{z_4}$&$J_{z_5}$&$J_{z_6}$&$J_{z_7}$&$J_{z_8}$&$J_{z_9}$&$J_{z_{10}}$&$R_1$&$R_2$&$R_3$&$R_4$&$Q$
\\
\hline\hline
$P_1$&-1&-1&1&1&-1&-1&1&1&1&1&-1&1&1&1&-1
\\
\hline
$P_2$&-1&-1&1&1&1&1&-1&-1&1 &1&\ &-1&1&1&-1
\\
\hline
$P_3$&-1&-1&1&1&1&1&1&1&-1&-1&\ &\ &-1&1&-1
\\
\hline
$T$&1&1&1&-1&-1&-1&-1&-1&-1&-1&\ &\ &\ & -1&-1
\\
\hline\hline
\end{tabular}\label{t:37}
\end{table}

Since the dimension of the minimal admissible module $V^{3,7}$, which is also irreducible, is 64, the common eigenspace $E^1=\cap_{j=1}^{3}E^{1}_{P_j}$ is 8-dimensional and neutral. We choose the following basis for $E^1$, starting from $v\in E^1$, $\la v,v\ra_{V^{3,7}}=1$ and making use the anti-isometric operator $Q=J_{z_5}J_{z_7}J_{z_9}$,
$$
\begin{array}{llllll}
&x_1=v,\quad &x_2= J_{z_1}J_{z_2}v,\quad &x_3=J_{z_4}J_{z_1}Qv,\quad &x_4=J_{z_4}J_{z_2}Qv,
\\
&x_5=J_{z_4}v, &x_6= J_{z_4}J_{z_1}J_{z_2}v,\quad &x_7=J_{z_1}Qv,\quad &x_8=J_{z_2}Qv.
\end{array}
$$
 The basis is orthonormal by Lemma~\ref{orthogonal}, satisfies $\la x_j,x_j\ra_{V^{3,7}}=-\la x_{4+j},x_{4+j}\ra_{V^{3,7}}=1$ and
\[
T(x_i)=x_i,\ \ i=1,2,3,4\quad\text{and}\quad T(x_i)=-x_i,\ \ i=5,6,7,8,
\]
due to the choice of the operators $R_4$ and $Q$. Thus we can apply Lemma~\ref{rem:contradiction} to the neutral space $V=E^{1}$ with operators $S=A^{\tau}_1A_1$ and $T$. It finishes the proof. 
\\

{\sc Case $(r,s)=(3,11)$}.
The minimal admissible modules $V^{11,3}$ and $V^{3,11}$ are isometric to $\mathbb R^{64,64}$. We choose a set of mutually commuting isometric involutions:
$$
P_j=J_{z_1}J_{z_2}J_{z_{3+2j}}J_{z_{4+2J}},\ j=1,\ldots,5,\quad T=J_{z_{1}}J_{z_{2}}J_{z_{3}},
$$
acting on $V^{3,11}$. The complementary operators are
$
R_k=J_{z_{3+2k}}$, $k=1,\ldots,5$, and $R_6=J_{z_{4}}$.
{\tiny
\begin{table}[h]
\center\caption{Commutation relations of operators on $V^{3,11}$}
\begin{tabular}{|c||c|c|c|c|c|c|}
\hline
\ &$R_1$&$R_2$&$R_3$&$R_4$&$R_5$&$R_6$
\\
\hline\hline
$P_1$&-1&1&1&1&1&1
\\
\hline
$P_2$&\ &-1 &1 &1&1&1 
\\
\hline
$P_3$&\ &\ &-1 &1&1&1
\\
\hline
$P_4$&\ &\ &\ &-1&1&1
\\
\hline
$P_5$&\ &\ &\ &\ &-1&1 
\\
\hline
$T$&\ &\ &\ &\ &\ &-1 
\\
\hline\hline
\end{tabular}\label{t:113}
\end{table}
}
The space $E^1=\bigcap\limits_{j=1}^5E^{1}_{P_j}$ is 4-dimensional neutral space. We choose the orthonormal basis
$\{
x_1=v,\ x_2=J_{z_1}J_{z_2}v,\ x_3=J_{z_4}v,\ x_4=J_{z_4}J_{z_1}J_{z_2}v\}
$
for $E^1$ with $v\in E^1$, $\la v,v\ra_{V^{3,11}}=1$.
It is easy to see that
$$
T(x_j)=x_j,\ \quad T(x_{2+j})=-x_{2+j},\quad j=1,2.
$$
Thus, if we assume that there is an isomorphism $\Phi=A\oplus C\colon\mathcal{N}_{3,11}\to\mathcal{N}_{11,3}$, then the operator $S=A^{\tau}_1A_1$ will act on $E^1$. Applying Lemma~\ref{rem:contradiction} to the neutral space $E^1$, operators $S$ and $T$, we obtain a contradiction. This finishes the proof.
\end{proof}

\begin{cor}\label{non-existence Cl33}
There are no automorphism 
$\Phi=A\oplus C$ 
of $\mathcal{N}_{3,3}$
with the condition $C^{\tau}C =-Id$.
\end{cor}

\begin{proof}
If we assume that such an automorphism $\Psi$ exists, 
then it must induce an isomorphism between 
$\mathcal{N}_{3,2}$ and $\mathcal{N}_{2,3}$, which contradicts to Theorem~\ref{32-23}. The constructive proof can be performed as follows. Let us assume 
the existence of an automorphism
$\Psi=A\oplus C$ with $C^{\tau}C=-\Id$.
We fix mutually commuting isometric involutions and the complementary operators
$$
P_1=J_{z_1}J_{z_2}J_{z_4}J_{z_5},\  P_2=J_{z_1}J_{z_3}J_{z_5}J_{z_6},\
T=J_{z_1}J_{z_2}J_{z_3},\ \  R_1=J_{z_4},\ R_2=J_{z_6},\ R_3=J_{z_3}J_{z_6},
$$
acting on $V^{3,3}$. Denote by $\{w_6,\ldots,w_1\}$ another orthonormal basis of $\mathbb R^{3,3}$, where $w_6,w_5,w_4$ are positive vectors and $w_3,w_2,w_1$ are negative. Put $C(z_i)=w_i$. The common eigenspace $E^1=\cap_{j=1}^2 E^{1}_{P_j}$
is spanned by $\{x_1=v,\ x_2=J_{z_1}J_{z_5}v\}$, where $v=P_1(v)=P_2(v)=T(v)$ and $\la v,v\ra_{V^{3,3}}=1$. Observe that $T(x_2)=-x_2$. Thus we obtain a contradiction as in Lemma~\ref{rem:contradiction} by setting $S=A^{\tau}A$ for the neutral space $E^1$.
\end{proof}

We can not apply directly the arguments of Theorem~\ref{periodicity} to non-isomorphic pairs. Nevertheless, by a direct construction we obtain that the non-isomorphic properties are also respect the same periodicity. 

\begin{theorem}\label{th:nonisom_high}
If $(r,s)\in\{
(3,1),\  (3,2),\  (3,7),\  (3,11)\}$,
then 
\begin{itemize}
\item[1.] {the Lie algebra $\mathcal{N}_{r+4k,s+4k}$ is not isomorphic to $\mathcal{N}_{s+4k,r+4k}$ for any $k=0,1,2,\ldots$,}
\item[2.]{the Lie algebra $\mathcal{N}_{r,s+8k}$ is not isomorphic to $\mathcal{N}_{s+8k,r}$ for any $k=0,1,2,\ldots$,}
\item[3.]{the Lie algebra $\mathcal{N}_{r+8k,s}$ is not isomorphic to $\mathcal{N}_{s,r+8k}$ for any $k=0,1,2,\ldots$.}
\end{itemize}
\end{theorem}
\begin{proof}
Observe that if the Lie algebra $\mathcal{N}_{r,s}$ has a system of $p$ mutually commuting isometric involutions, then the Lie algebra  $\mathcal{N}_{r+4k,s+4k}$ has $p+4k$ mutually commuting isometric involutions. The dimensions of minimal admissible modules are related   by $\dim(V^{r+4k,s+4k})=16\dim(V^{r,s})$. Therefore, the dimension of the common eigenspace $E^1$, corresponding to  eigenvalues 1 of all the involutions, does not change and equal for $\mathcal{N}_{r,s}$ and $\mathcal{N}_{r+4k,s+4k}$ for any $k$. The same argument valid for the Lie algebras $\mathcal{N}_{r,s+8k}$ and $\mathcal{N}_{s+8k,r}$. During the proof we show that for each value of $(r,s)$ in the statement of the theorem, we can apply Lemma~\ref{rem:contradiction} and deduce that 
$\mathcal{N}_{r+4k,s+4k}\not\cong\mathcal{N}_{s+4k,r+4k}$ and $\mathcal{N}_{r,s+8k}\not\cong\mathcal{N}_{s+8k,r}$ for any $k$. 

We slightly change the notations. Denote by $\{z_1,\ldots,z_r,\zeta_1,\ldots,\zeta_s\}$ the orthonormal basis of $\mathbb R^{r,s}$ with $\la z_k,z_k\ra_{r,s}=1$, $k=1,\ldots,r$, and $\la \zeta_j,\zeta_j\ra_{r,s}=-1$, $j=1,\ldots,s$.
\\

{\sc Case $(r,s)=(3,1)$}. The minimal admissible module $V^{3,1}$ has the isometric involution $T=J_{z_1}J_{z_2}J_{z_3}$. The minimal admissible module $V^{3+4k,1+4k}$ has the following mutually commuting isometric involutions 
$$
P_1=J_{z_1}J_{z_2}J_{z_{4}}J_{z_{5}},\ P_{2}=J_{z_1}J_{z_2}J_{z_{6}}J_{z_{7}},\ldots,\  P_{2k}=J_{z_1}J_{z_2}J_{z_{2+4k}}J_{z_{3+4k}},
$$
$$
P_{2k+1}=J_{z_1}J_{z_2}J_{\zeta_{2}}J_{\zeta_{3}},\ \ \ldots,\  \ P_{4k}=J_{z_1}J_{z_2}J_{\zeta_{4k}}J_{\zeta_{1+4k}},\ \ T=J_{z_1}J_{z_2}J_{z_3}.
$$
The complementary operators are 
$$
R_l=J_{z_{2l+3}}J_{\zeta_1},\ l=1,\ldots,2k,\quad R_l=J_{\zeta_{2(l-2k)+1}}, \ l=2k+1,\ldots, 4k,\quad R_{4k+1}=J_{\zeta_1}.
$$
We choose the basis of $E^1=\bigcap_{j}^{4k}E^1_{P_j}$, starting from $v\in E^1$, $\la v,v\ra_{V^{3+4k,1+4k}}=1$. We also need an isometric operator $Q=\prod_{j=1}^{4k}R_j$. Thus we have
\begin{equation}\label{eq:Q}
QP_j=-P_jQ,\quad QT=TQ,\quad J_{\zeta_1}P_j=P_jJ_{\zeta_1},\quad J_{\zeta_1}T=-TJ_{\zeta_1},\quad Q^2=\Id.
\end{equation}
If it is necessary, we can apply Lemma~\ref{orthogonal} and find the following orthonormal basis 
\begin{equation}\label{eq:b310}
\begin{array}{lllll}
x_1=v,\quad & x_2=J_{z_1}J_{z_2}v,\quad & x_3=J_{z_1}Qv,\quad &x_4=J_{z_2}Qv,
\\
x_5=J_{\zeta_1}v,\quad & x_6=J_{\zeta_1}J_{z_1}J_{z_2}v,\quad & x_7=J_{\zeta_1}J_{z_1}Qv,\quad &x_8=J_{\zeta_1}J_{z_2}Qv,
\end{array}
\end{equation}
with $\la x_k,x_k\ra_{V^{3+4k,1+4k}}=-\la x_{k+4},x_{k+4}\ra_{V^{3+4k,1+4k}}=1$, $k=1,\ldots,4$ and
$$
T(x_j)=x_j,\quad\text{and}\quad T(x_{4+j})=-x_{4+j},\ j=5,6,7,8.
$$
due to the choice of the corresponding operators. Thus assuming that there is a Lie algebra isomorphism $\Phi=A\oplus C\colon\mathcal N_{3+4k,1+4k}\to \mathcal N_{1+4k,3+4k}$, we define the map $S=A^{\tau}A\colon V^{3+4k,1+4k}\to V^{1+4k,3+4k}$ and obtain a contradiction by Lemma~\ref{rem:contradiction}.

The minimal admissible module $V^{3,1+8k}$ has the following mutually commuting isometric involutions 
$$
P_{1}=J_{z_1}J_{z_2}J_{\zeta_{2}}J_{\zeta_{3}},
\quad
P_{2}=J_{z_1}J_{z_2}J_{\zeta_{4}}J_{\zeta_{5}},\ \ \ldots,\  \ P_{4k}=J_{z_1}J_{z_2}J_{\zeta_{8k}}J_{\zeta_{1+8k}},\ \ T=J_{z_1}J_{z_2}J_{z_3}.
$$
The complementary operators are 
$
R_l=J_{\zeta_{2l+3}}$, $l=1,\ldots,4k$, and $R_{4k+1}=J_{\zeta_1}$. We also need the isometric operator $Q=\prod_{j=1}^{4k}R_j$.
Choose the basis~\eqref{eq:b310} and finish the proof by applying Lemma~\ref{rem:contradiction}.

For the minimal admissible module $V^{3+8k,1}$ we choose the mutually commuting isometric involutions 
$$
P_{1}=J_{z_1}J_{z_2}J_{z_4}J_{z_5},
\quad
P_{2}=J_{z_1}J_{z_2}J_{z_6}J_{z_7},\ \ \ldots,\  \ P_{4k}=J_{z_1}J_{z_2}J_{z_{2+8k}}J_{z_{3+8k}},\ \ T=J_{z_1}J_{z_2}J_{z_3}.
$$
The complementary operators are 
$
R_l=J_{z_{2l+3}}J_{\zeta_1}$, $l=1,\ldots,4k$, and $R_{4k+1}=J_{\zeta_1}$ and $Q=\prod_{j=1}^{4k}R_j$.
We choose the basis~\eqref{eq:b310} and finish the proof by applying Lemma~\ref{rem:contradiction}.
\\

{\sc Case $(r,s)=(3,7)$}. This case is similar to the previous. Recall that for $V^{3,7}$ the mutually commuting involutions are
$$
P_j=J_{z_1}J_{z_2}J_{\zeta_{2j}}J_{\zeta_{1+2j}},\ j=1,2,3,\quad T=J_{z_1}J_{z_2}J_{z_3}.
$$
The complementary anti-isometric operators are
$$
R_l=J_{\zeta_{1+2l}},\ l=1,2,3,\quad R_4=J_{\zeta_1},\quad Q=\prod_{l=1}^{3}J_{\zeta_{1+2l}}.
$$
For the minimal admissible module $V^{3+4k,7+4k}$ we choose the following involutions
$$
P_m=J_{z_1}J_{z_2}J_{z_{2+2j}}J_{z_{3+2j}},\ m=1,\ldots, 2k,\quad P_j=J_{z_1}J_{z_2}J_{\zeta_{2j}}J_{\zeta_{1+2j}},\ j=1,\ldots,3+2k,
$$
and $T=J_{z_1}J_{z_2}J_{z_3}$. The complementary operators are 
$$
R_p=J_{z_{3+2p}}J_{\zeta_1},\ p=1,\ldots 2k,\quad R_l=J_{\zeta_{1+2l}},\ l=1,\ldots,3+2k,\quad R_{4+4k}=J_{\zeta_1},
$$
and the isometry $Q=\prod_{l=1}^{4k+4}R_j$. Since all the chosen operators satisfy~\eqref{eq:Q}, then we can take the basis~\eqref{eq:b310} and finish the proof, applying Lemma~\ref{rem:contradiction}. 

For the minimal admissible module $V^{3,7+8k}$ we write
$$
P_j=J_{z_1}J_{z_2}J_{\zeta_{2j}}J_{\zeta_{1+2j}},\ j=1,\ldots,3+4k,\quad T=J_{z_1}J_{z_2}J_{z_3},
$$
$$
R_l=J_{\zeta_{1+2l}},\ l=1,\ldots,3+4k,\quad R_{4+4k}=J_{\zeta_1},\quad Q=\prod_{l=1}^{4k+4}R_j.
$$

For the minimal admissible module $V^{3+8k,7}$ we define
$$
P_m=J_{z_1}J_{z_2}J_{z_{2+2j}}J_{z_{3+2j}},\ m=1,\ldots, 4k,\quad P_j=J_{z_1}J_{z_2}J_{\zeta_{2j}}J_{\zeta_{1+2j}},\ j=1,2,3,
$$
$$
T=J_{z_1}J_{z_2}J_{z_3},\quad Q=\prod_{l=1}^{4k+4}R_j,\quad\text{where}
$$
$$
R_l=J_{z_{3+2l}}J_{\zeta_1},\ l=1,\ldots 4k,\quad R_{l+4k}=J_{\zeta_{1+2l}},\ l=1,2,3,\quad R_{4+4k}=J_{\zeta_1}.
$$
\\

{\sc Case $(r,s)=(3,11)$}. We recall that the set of mutually commuting isometric involutions acting on $V^{3,11}$ is:
$$
P_j=J_{z_1}J_{z_2}J_{\zeta_{2j}}J_{\zeta_{1+2j}},\ j=1,\ldots,5,\quad T=J_{z_{1}}J_{z_{2}}J_{z_{3}}.
$$
The complementary operators are
$R_j=J_{\zeta_{1+2j}}$, $j=1,\ldots,5$, and $R_6=J_{\zeta_{1}}$.
For the minimal admissible module $V^{3+4k,11+4k}$ we make the following modifications
$$
P_m=J_{z_1}J_{z_2}J_{z_{2+2j}}J_{z_{3+2j}},\ m=1,\ldots, 2k,\quad P_j=J_{z_1}J_{z_2}J_{\zeta_{2j}}J_{\zeta_{1+2j}},\ j=1,\ldots,5+2k,
$$
and $T=J_{z_1}J_{z_2}J_{z_3}$. The complementary operators are 
$$
R_p=J_{z_{3+2p}}J_{\zeta_1},\ p=1,\ldots 2k,\quad R_l=J_{\zeta_{1+2l}},\ l=1,\ldots,5+2k,\quad R_{4+4k}=J_{\zeta_1},
$$
The space $E^1=\bigcap\limits_{j=1}^{5+4k}E^{1}_{P_j}$ is 4-dimensional neutral space. We choose the orthonormal basis:
$$
x_1=v,\quad x_2=J_{z_1}J_{z_2}v,\quad x_3=J_{\zeta_1}v,\quad x_4=J_{\zeta_1}J_{z_1}J_{z_2}v
$$
with $v\in E^1$, $\la v,v\ra_{V^{3+4k,11+4k}}=1$.
Since $T_1(x_j)=x_j$, $T_1(x_{2+j})=-x_{2+j}$, $j=1,2$, we can finish the proof by applying Lemma~\ref{rem:contradiction}. 

It is clear what changing have to be done for the rest of the proof.
\\

{\sc Case $(r,s)=(3,2)$}. The minimal admissible module $V^{3,2}$ allows two mutually commuting isometric involutions $
P=J_{z_1}J_{z_2}J_{\zeta_1}J_{\zeta_2}$ and $T=J_{z_1}J_{z_2}J_{z_3}$.
For the minimal admissible module $V^{3+4k,2+4k}$ we choose the following involutions
$$
P_1=J_{z_1}J_{z_2}J_{z_4}J_{z_5},\quad P_2=J_{z_1}J_{z_2}J_{z_6}J_{z_7},\quad\ldots, \quad P_{2k}=J_{z_1}J_{z_2}J_{z_{2+4k}}J_{z_{3+4k}},
$$
$$
P_{2k+1}=J_{z_1}J_{z_2}J_{\zeta_1}J_{\zeta_2},\quad P_{2k+2}=J_{z_1}J_{z_2}J_{\zeta_3}J_{\zeta_4},\quad\ldots\quad P_{4k+1}=J_{z_1}J_{z_2}J_{\zeta_{1+4k}}J_{\zeta_{2+4k}},
$$
and $T=J_{z_1}J_{z_2}J_{z_3}$. The complementary anti-isometric operators are 
$$
R_j=J_{z_{3+2j}}J_{\zeta_{1}},\ j=1,\ldots,2k,\quad R_{2k+m}=J_{\zeta_{2m-1}},\ m=1,\ldots 2k+1, 
$$
and $R_{4k+2}=J_{z_1}\prod_{j=1}^{4k+1}R_j$. It is easy to see that $R_{4k+2}$ commutes with all $P_j$ and anti-commute with $T$.
The space $E^{1}=\bigcap_{j=1}^{4k+1}E^1_{P_j}$ is neutral 4-dimensional. We construct an orthonormal basis, starting from $v\in E^{1}$, $\la v,v\ra_{V^{3+4k,2+4k}}=1$: 
$$
x_1=v,\quad x_2=J_{z_1}J_{z_2}v,\quad x_3=R_{4k+2}(x_1),\quad x_4=R_{4k+2}(x_2)
$$
with $\la x_i,x_i\ra_{{3+4k,2+4k}}=-\la x_{i+2},x_{i+2}\ra_{{3+4k,2+4k}}=1$, $i=1,2$.
We see that
$$
Tx_1=x_1,\quad Tx_2=x_2,\quad Tx_3=-x_3,\quad Tx_4=-x_4.
$$
We can apply now Lemma~\ref{rem:contradiction}.

For the minimal admissible module $V^{3,2+8k}$ we choose the involutions
$$
P_{1}=J_{z_1}J_{z_2}J_{\zeta_1}J_{\zeta_2},\quad P_{2}=J_{z_1}J_{z_2}J_{\zeta_3}J_{\zeta_4},\quad\ldots\quad P_{4k+1}=J_{z_1}J_{z_2}J_{\zeta_{1+8k}}J_{\zeta_{2+8k}},
$$
and $T=J_{z_1}J_{z_2}J_{z_3}$. The complementary anti-isometric operators are 
$$
R_j=J_{\zeta_{2j-1}},\ j=1,\ldots 2k+1, \quad R_{4k+2}=J_{z_1}\prod_{j=1}^{4k+1}R_j.
$$
Then we finish the proof as in the previous case.

For the minimal admissible module $V^{3+8k,2}$ we choose the involutions
$$
P_1=J_{z_1}J_{z_2}J_{z_4}J_{z_5},\quad\ldots, \quad P_{4k}=J_{z_1}J_{z_2}J_{z_{2+8k}}J_{z_{3+8k}},\quad P_{4k+1}=J_{z_1}J_{z_2}J_{\zeta_1}J_{\zeta_2},
$$
and $T=J_{z_1}J_{z_2}J_{z_3}$. The complementary anti-isometric operators are 
$$
R_j=J_{z_{3+2j}}J_{\zeta_{1}},\ j=1,\ldots,4k,\quad R_{4k+1}=J_{\zeta_{1}},\quad\text{and}\quad R_{4k+2}=J_{z_1}\prod_{j=1}^{4k+1}R_j.
$$ 
It is easy to see that $R_{4k+2}$ commutes with all $P_j$ and anti-commute with $T$ and we can finish the proof by applying Lemma~\ref{rem:contradiction}.
\end{proof}

\begin{cor}
There are no automorphisms of the algebra
$\mathcal{N}_{3+4k,3+4k}$, $k=0,1,\ldots$, of the form
$\Psi=A\oplus C$
with $C^{\tau}C=-Id$. 
\end{cor}
\begin{proof}
If such an automorphism would exist, then it could induce
an isomorphism between
$\mathcal{N}_{3+4k,2+4k}$ and $\mathcal{N}_{2+4k,3+4k}$, which is a contradiction. The proof can be also obtained by method of Theorem~\ref{th:nonisom_high} as in Corollary~\ref{non-existence Cl33}.
\end{proof}


\section{Step 2: trivially non-isomorphic Lie algebras}\label{sec:4}


In this section we study the isomorphism between the Lie algebras
$\mathcal N_{r,s}$ and $\mathcal N_{s,r}$, where one of the Lie
algebras 
is constructed from minimal admissible Clifford module and another one
is constructed 
by using the direct sum of two minimal admissible Clifford modules. 
We formulate one theorem, where we state all the cases that could be
used for further applications of periodicity~\eqref{perCl}. 
We continue to use the notation $V^{r,s}$ for minimal admissible
modules 
and we write $U^{s,r}$ to denote a non-minimal admissible module. 
In the case when there are two minimal admissible modules, 
we write
$V^{r,s}_+\cong V_{+}$ and $V^{r,s}_-\cong V_{-}$. 
We use the notation $\mathcal
N_{r,s}^2$
for the Lie algebra constructed by using the direct sum of 
two minimal admissible modules.
Below in Theorem~\ref{th:double} we write 
$\mathcal{N}_{r,s}^2\cong\mathcal{N}_{r,s}(V^{r,s}\oplus V^{r,s})$ for the case of $r-s\neq 3(mod\,\,4)$. In the cases $r-s= 3(mod\,\,4)$ the Clifford algebra $\Cl_{r,s}$ has two minimal
admissible module and in this case we write $\mathcal{N}^{2}_{r,s}\cong
\mathcal{N}_{r,s}(V^{r,s}_{+}\oplus V^{r,s}_{-})$.

\begin{theorem}\label{th:double}
The following pairs of the Lie algebras 
are isomorphic
$$
\begin{array}{lllllll}
&\mathcal N_{3,0}^2\cong\mathcal N_{0,3},\quad &\mathcal N_{5,0}^2\cong\mathcal N_{0,5},\quad&\mathcal N_{6,0}^2\cong\mathcal N_{0,6},\quad&\mathcal N_{7,0}^2\cong\mathcal N_{0,7},
\\
&\mathcal N_{2,1}\cong\mathcal N_{1,2}^2,\quad&\mathcal N_{4,1}^2\cong\mathcal N_{1,4},\quad&\mathcal N_{7,1}^2\cong\mathcal N_{1,7},
\\
&\mathcal N_{4,2}\cong\mathcal N_{2,4}^2,\quad&\mathcal N_{7,2}^2\cong\mathcal N_{2,7},
\\
&\mathcal N_{4,3}\cong\mathcal N_{3,4}^2,\quad&\mathcal N_{5,3}\cong\mathcal N_{3,5}^2,&\mathcal N_{6,3}\cong\mathcal N_{3,6}^2.
\end{array}
$$
\end{theorem}
\begin{proof}
The scheme of the proof is the following: for each pair of Lie algebras $\mathcal N_{r,s}^2$ and  $\mathcal N_{s,r}$ we find another pair  $\mathcal N_{l,m}$, $\mathcal N_{m,l}$ of isomorphic algebras such that $V^{l,m}$ is isometric to $V^{s,r}$ and $U^{r,s}$. Then representations of the Clifford algebras $\Cl_{l,m}$ and $\Cl_{m,l}$ will induce actions on $V^{s,r}$ and $U^{r,s}$ that allow to induce the isomorphism $\overline \Phi\colon \mathcal N_{s,r}\to\mathcal N_{r,s}^2$ from the existing isomorphism between $\Phi\colon \mathcal N_{l,m}\to\mathcal N_{m,l}$, as it is shown on the diagram:
\begin{equation}\label{eq:scheme}
\begin{CD}
\{0\}@>>> K_{-} @>>> \mathcal{N}_{l,m} @>{I\oplus\pi_{-}}>> \mathcal{N}_{s,r}@>>>\{0\}\\
@.        @VV{C}V @V{\Phi}VV @VV{\overline{\Phi}}V\\                          
\{0\}@>>> K_{+}@>>> \mathcal{N}_{m,l}@>{I\oplus\pi_{+}}>> \mathcal{N}_{r,s}^2@>>>\{0\}.
\end{CD}
\end{equation}
We consider case by case finding suitable isomorphic pairs of Lie algebras that will induce the isomorphism for the pairs listed in the statement of the theorem.
\\

{\sc Case $\mathcal N_{3,0}^2\cong\mathcal N_{0,3}$.} Let $V^{4,0}$ be the minimal admissible module of
the Clifford algebra $\Cl_{4,0}$, then
the natural inclusion $\mathbb{R}^{3,0}\subset\mathbb{R}^{4,0}$
defines an admissible module $U^{3,0}$ of $\Cl_{3,0}$. Then it must be $U^{3,0}=V_{+}^{3,0}\oplus V_{-}^{3,0}$, since
the operator $J_{z_4}$ anti-commutes with the volume form
$\Omega^{3,0}=J_{z_1}J_{z_2}J_{z_3}$. Thus $U^{3,0}$ includes both eigenspaces of $\Omega^{3,0}$ and $U^{3,0}$ is isometric to $V^{4,0}$.    

The orthogonal projection $\pi_+\colon \mathbb{R}^{4,0} \to\mathbb{R}^{3,0}$ with kernel $K_+=\spn\{z_1\}$ and the isometry map $I\colon V^{4,0}\to U^{3,0}$ define a surjective Lie algebra homomorphism $
\rho=I\oplus\pi_+\colon\mathcal{N}_{4,0}\to \mathcal{N}_{3,0}^2$. Analogously, the orthogonal projection $\pi_-\colon \mathbb R^{0,4}\to\mathbb R^{0,3}$ with the kernel $K_-=\spn\{\zeta_1\}$ and the isometry map $I\colon V^{0,4}\to V^{0,3}$ induce a surjective Lie algebra homomorphism $
\rho=I\oplus\pi_-\colon\mathcal{N}_{0,4}\to \mathcal{N}_{0,3}$.
Then the isomorphism
$\Phi\colon\mathcal{N}_{0,4}\to \mathcal{N}_{4,0}$ induces an isomorphism
$\overline \Phi\colon \mathcal{N}_{0,3}\to \mathcal{N}_{3,0}^2$ by~\eqref{eq:scheme},
since $\Phi(\zeta_1)=C(\zeta_{1})=z_1$.
\\

{\sc Cases $\mathcal N_{5,0}^2\cong\mathcal N_{0,5}$, $\mathcal N_{6,0}^2\cong\mathcal N_{0,6}$, $\mathcal N_{7,0}^2\cong\mathcal N_{0,7}$.} From now on we will only indicate the structure of $U^{r,s}$ and the isomorphic Lie algebras that induce the necessary isomorphism.

We have $U^{5,0}=V^{5,0}\oplus V^{5,0}$, $U^{6,0}=V^{6,0}\oplus V^{6,0}$, and $U^{7,0}=V^{7,0}_+\oplus V^{7,0}_-$. The isomorphisms $\overline \Phi$ are induced from $\Phi\colon\mathcal{N}_{8,0}\to \mathcal{N}_{0,8}$.
\\

{\sc Cases $\mathcal N_{2,1}\cong\mathcal N_{1,2}^2$, $\mathcal N_{4,1}\cong\mathcal N_{1,4}^2$, $\mathcal N_{7,1}^2\cong\mathcal N_{1,7}$.}
Let
$\Psi=A\oplus C\colon\mathcal{N}_{2,2}\to\mathcal{N}_{2,2}$, be a Lie algebra automorphism, such that 
\[
C(z_1)=z_4,\quad C(z_2)=z_3,\quad C(z_3)=z_2,\quad C(z_4)=z_1,\quad\text{and}\quad CC^{\tau}=-\Id.
\]
We have $U^{1,2}=V^{1,2}\oplus V^{1,2}$ and the isomorphism $\overline\Phi\colon \mathcal N_{2,1}\to \mathcal N_{1,2}^2$ is induced by the automorphism~$\Psi$.

We have $U^{1,4}=V^{1,4}\oplus V^{1,4}$ and the isomorphism $\overline \Phi\colon \mathcal N_{4,1}\to\mathcal N_{1,4}^2$ is induced from $\Phi\colon\mathcal{N}_{5,1}\to \mathcal{N}_{1,5}$. Analogously, $U^{7,1}=V^{7,1}\oplus V^{7,1}$ and the isomorphism $\overline \Phi\colon \mathcal N_{7,1}^2\to\mathcal N_{1,7}$ is induced from $\Phi\colon\mathcal{N}_{8,1}\to \mathcal{N}_{1,8}$.
\\

{\sc Cases $\mathcal N_{4,2}\cong\mathcal N_{2,4}^2$, $\mathcal N_{7,2}^2\cong\mathcal N_{2,7}$}. 
We have $U^{2,4}=V^{2,4}\oplus V^{2,4}$ and the isomorphism $\overline \Phi\colon \mathcal N_{4,2}\to\mathcal N_{2,4}^2$ is induced from $\Phi\colon\mathcal{N}_{5,2}\to \mathcal{N}_{2,5}$. One has $U^{7,2}=V^{7,2}\oplus V^{7,2}$ and the isomorphism $\overline \Phi\colon \mathcal N_{7,2}^2\to\mathcal N_{2,7}$ is induced from $\Phi\colon\mathcal{N}_{8,2}\to \mathcal{N}_{2,8}$.
\\

{\sc Cases $\mathcal N_{4,3}\cong\mathcal N_{3,4}^2$, $\mathcal N_{5,3}\cong\mathcal N_{3,5}^2$, $\mathcal N_{6,3}\cong\mathcal N_{3,6}^2$.} We have $U^{3,4}=V^{3,4}_+\oplus V^{3,4}_-$, and the isomorphism $\overline \Phi\colon \mathcal N_{4,3}\to\mathcal N_{3,4}^2$ is induced from the automorphism $\Psi$ of $\mathcal{N}_{4,4}$. 

One has the following modules $U^{3,k}=V^{3,k}\oplus V^{3,k}$ and the isomorphism $\overline \Phi\colon \mathcal N_{k,3}\to\mathcal N_{3,k}^2$ is induced from $\Phi\colon\mathcal{N}_{k+1,4}\to \mathcal{N}_{4,k+1}$ for $k=5,6$.
\end{proof}

The last theorem is an application of the construction made in Theorem~\ref{periodicity} to show the isomorphism of Lie algebras of high dimention.

\begin{theorem}\label{periodicity3}
If the Lie algebra $\mathcal{N}^2_{r,s}$ is isomorphic to the Lie algebra $\mathcal{N}_{s,r}$, then 
\begin{itemize}
\item[1.]{the Lie algebras $\mathcal{N}^2_{r,s+8k}$ and $\mathcal{N}_{s+8k,r}$ 
are isomorphic;}
\item[2.]{the Lie algebras $\mathcal{N}^2_{r+8k,s}$ and $\mathcal{N}_{s,r+8k}$ 
are isomorphic;}
\item[3.]{the Lie algebras $\mathcal{N}^2_{r+4k,s+4k}$ and $\mathcal{N}_{s+4k,r+4k}$ 
are isomorphic.}
\end{itemize}
for any $k=1,2,\ldots$.
\end{theorem}


\section{Step 3: uniqueness of minimal dimensional Lie algebras}\label{sec:step3}


Recall that we are classifying the resulting Lie algebras, constructed as $H$-type Lie algebras by making use of (probably) different admissible scalar products on the same representation space. In the present section we discuss the uniqueness of Lie algebras first with the same representation space but different scalar products and then the Lie algebras constructed from non-equivalent representation spaces. 
\begin{prop}\label{genisom1}  
Let $(V,\la.\,,.\ra_V)$ be an admissible module (not necessarily
minimal) and denote by 
$V_1=(V,-\la.\,,.\ra_V)$ an admissible module with the scalar product of
the opposite sign, see Remark~\ref{sign change}. Then
the Lie algebras $\mathcal N_{r,s}(V)$ and $\mathcal N_{r,s}(V_1)$ are isomorphic under the map 
$$
\begin{array}{ccc}
\mathcal N_{r,s}(V)=V\oplus\mathbb{R}^{r,s}&\longrightarrow &\mathcal N_{r,s}(V_1)=V\oplus\mathbb{R}^{r,s}
\\
(x,z)&\mapsto& (x,-z).
\end{array}
$$
\end{prop}

\begin{prop}\label{genisom2} Let $V$ be a representation space of a
  Clifford algebra $\Cl_{r,s}$. We denote by 
$V_1=(V,\la.\,,.\ra_V^{(1)})$ and $V_2=(V,\la.\,,.\ra_V^{(2)})$ two {\bf minimal}
admissible modules with different scalar products. 
Then the Lie algebras $\mathcal N_{r,s}(V_1)$ and $\mathcal N_{r,s}(V_2)$ are isomorphic. 
\end{prop}

\begin{proof}
Any minimal admissible module is cyclic in the following sense. There is a
vector, generating an orthonormal basis of the module by successive
actions 
of the maps $J_{z_j}$ on this generating vector, see Theorem~\ref{integral basis I}, item 3. 
We can find the generating vector in the space $E^1=\cap E^1_{P_i}$,
where $\{P_{i}\}$ is a set of the 
maximal number of mutually
commuting isometric involutions. We choose two vectors $u,v\in E^1$ such that 
$\la v,v\ra^{(1)}_V=1$,
and $\la u,u\ra^{(2)}_V=1$, where, if it is necessary, we can change the
sign 
of the scalar products to be opposite according 
to Remark~\ref{sign change} and Proposition~\ref{genisom1}.
Then these vectors 
generate the Clifford module, in the sense that $V$ is the span 
of all $\{J_{z_{i_1}}J_{z_{i_2}}\cdots J_{z_{i_{k}}}v\}$
and it is also the span of all $\{J_{z_{i_1}}J_{z_{i_2}}\cdots
J_{z_{i_{k}}}u\}$. Moreover 
$$
\la J_{z_{i_1}}J_{z_{i_2}}\cdots J_{z_{i_{k}}}v,J_{z_{j_1}}J_{z_{j_2}}\cdots J_{z_{j_{k'}}}v\ra^{(1)}_V
=
\la J_{z_{i_1}}J_{z_{i_2}}\cdots J_{z_{i_{k}}}u,J_{z_{j_1}}J_{z_{j_2}}\cdots J_{z_{j_{k'}}}u\ra^{(2)}_V,
$$
for any choice of the basis vectors. Let us denote by $[.\,,.]^{(k)}$, $k=1,2$ the brackets defined by scalar products $\la.\,,.\ra_V^{(k)}$.
Then
\begin{eqnarray*}
&&\la z_{\ell},[J_{z_{i_1}}\cdots J_{z_{i_{k}}}v,J_{z_{j_1}}\cdots J_{z_{j_{k'}}}v]^{(1)}\ra_{r,s}
=\la J_{z_{\ell}}J_{z_{i_1}}\cdots J_{z_{i_{k}}}v,J_{z_{j_1}}\cdots J_{z_{j_{k'}}}v\ra^{(1)}_V
\\
&=&\la J_{z_{\ell}}J_{z_{i_1}}\cdots J_{z_{i_{k}}}u,J_{z_{j_1}}\cdots J_{z_{j_{k'}}}u\ra^{(2)}_V
=\la z_{\ell},[J_{z_{i_1}}\cdots J_{z_{i_{k}}}u,J_{z_{j_1}}\cdots J_{z_{j_{k'}}}u]^{(2)}\ra_{r,s}
\end{eqnarray*}
for any $l=1,\ldots,r+s$. The map
$
(J_{z_{i_1}}J_{z_{i_2}}\cdots J_{z_{i_{k}}}v,z)\mapsto 
(J_{z_{i_1}}J_{z_{i_2}}\cdots J_{z_{i_{k}}}u,z)
$
is well defined and gives an isomorphism between the Lie algebras
$\mathcal N_{r,s}(V_1)$ and $\mathcal N_{r,s}(V_2)$.
\end{proof}

As it was mentioned in Section~\ref{sec:adm_mod}, some of the Clifford
algebras $\Cl_{r,s}$ have two minimal admissible modules, that correspond to two
non-equivalent irreducible modules supplied 
with a neutral or a positive definite scalar product,
making the representation maps skew-symmetric. 
We denote these modules by $V^{r,s}_+$ and $V^{r,s}_-$ 
and the corresponding Lie algebras by $\mathcal N_{r,s}(V^{r,s}_+)$ and $\mathcal N_{r,s}(V^{r,s}_-)$.

\begin{theorem}\label{th:step3}
If there are two minimal admissible $\Cl_{r,s}$-modules $V^{r,s}_+$ and $V^{r,s}_-$, then the Lie algebras $\mathcal N_{r,s}(V^{r,s}_+)$ and $\mathcal N_{r,s}(V^{r,s}_-)$ are isomorphic.
\end{theorem}
The proof of Theorem~\ref{th:step3} is contained in four lemmas. Lemma~\ref{lem:2new} is a reformulation of Lemma~\ref{relation between volume form}, Lemma~\ref{cor:constr} states general properties of Lie algebra isomorphism, Lemma~\ref{lem:basic3-7} shows the isomorphism for lower dimensional cases and the last Lemma~\ref{lem:highdimstep3} is an application of the periodicity~\eqref{perCl}. 

\begin{lemma}\label{lem:2new}
Let $\{z_i\}_{i=1}^{r+s}$ be an orthonormal basis of $\mathbb{R}^{r,s}$ and $\Phi=A\oplus C\colon \mathcal N_{r,s}(V^{r,s}_+)\to \mathcal N_{r,s}(V^{r,s}_-)$ a Lie algebra isomorphism. Then the following relations hold
$$
A\prod_{j=1}^{p}J_{z_j}=
\begin{cases}
\prod\limits_{j=1}^{p}\wJ_{C(z_j)}A,\quad \text{if}\quad p=2m, 
\\
\prod\limits_{j=1}^{p}\wJ_{C(z_j)}(A^{\tau})^{-1},\quad \text{if}\quad p=2m+1, 
\end{cases}
$$
$$
A^{\tau}\prod_{j=1}^{p} \wJ_{w_j}=
\begin{cases}
\prod\limits_{j=1}^{p}{J}_{C^{\tau}(w_j)}A^{\tau}\quad \text{if}\quad p=2m, 
\\
\prod\limits_{j=1}^{p}{J}_{C^{\tau}(w_j)}A^{-1}\quad \text{if}\quad p=2m+1.
\end{cases}
$$
for any $m \in \mathbb{N}$.
\end{lemma}

\begin{proof} 
Recall, that a Lie algebra isomorphism $\Phi=A\oplus C\colon\mathcal N_{r,s}(V^{r,s}_+)\to\mathcal N_{r,s}(V^{r,s}_-)$ satisfies $CC^{\tau}=\Id$ by Corollary~\ref{Identity relation}. Then the proof follows literally the proof of Lemma~\ref{relation between volume form}, where one has to change the condition $CC^{\tau}=-\Id$ to $CC^{\tau}=\Id$.
\end{proof}

\begin{lemma}\label{cor:constr}
Let $r-s=3(mod\,\,4)$ and the Clifford algebra $\Cl_{r,s}$ admits two minimal admissible modules $V_{\pm}^{r,s}$. Assume that $\Phi=A\oplus C\colon \mathcal{N}_{r,s}(V_{+}^{r,s})\to
\mathcal{N}_{r,s}(V_{-}^{r,s})$ is a Lie algebra isomorphism. Then $AA^{\tau}=-\det C\,\Id$ and $\det C=-1$ if $s=0$.
\end{lemma}

\begin{proof}
Let $z_1,\ldots z_{r+s}$ be orthonormal generators of the
algebra $\Cl_{r,s}$ and 
$V_{\pm}^{r,s}$ two non-equivalent minimal admissible
modules of the algebra $\Cl_{r,s}$ with the module actions, which are 
denoted by $J$ and $\wJ$, respectively. Both of admissible modules are irreducible and are distinguished by the actions of the volume forms
$\Omega^{r,s}=\prod_{j=1}^{r+s}J_{z_j}$ and
$\widetilde{\Omega}^{r,s}=\prod_{j=1}^{r+s}\wJ_{z_j},$
that is
\[
\Omega^{r,s}\equiv \Id~ \text{on}~V_{+}^{r,s},
\quad~\text{and}~\tilde{\Omega}^{r,s}\equiv -\Id~\text{on} ~V_{-}^{r,s}.
\]
Let us assume that there is an isomorphism 
$\Phi=A\oplus C\colon \mathcal{N}_{r,s}(V_{+}^{r,s})\to
\mathcal{N}_{r,s}(V_{-}^{r,s})$,
where $A\colon V_{+}^{r,s}\to V_{-}^{r,s}$ and $C\colon\mathbb{R}^{r,s}\to \mathbb{R}^{r,s}$ with $CC^{\tau}=\Id$ by Corollary~\ref{Identity relation}. We set $C(z_j)=\sum_{i=1}^{r+s} c_{ji}z_i=w_j$, $j=1,\ldots,r+s$. 
Then we have to satisfy the following
\begin{equation}\label{eq:v1}
AA^{\tau}=A\Omega^{r,s}A^{\tau}=\prod_{j=1}^{r+s}\wJ_{C(z_j)}=\det C\,\prod_{j=1}^{r+s}\wJ_{z_j}=\det C\,\widetilde\Omega^{r,s}=-\det C\,\Id
\end{equation}
by Lemma~\ref{lem:2new} and the definition of the map $C$. 

In the case $s=0$ the scalar product on $V_{\pm}^{r,0}$ is positive definite and the matrix $AA^{\tau}$ is positive. Therefore,  
$$
AA^{\tau}=-\det C\,\Id\quad\Longrightarrow\quad \det C=-1,\quad\Longrightarrow\quad AA^{\tau}=\Id.
$$
\end{proof}

We present general ideas for the construction of a possible map $A\colon V^{r,s}_+\to V^{r,s}_-$ in this case.  Observe that 
\begin{equation}\label{eq:Jpn}
J_{z_i}x=
\begin{cases}
(-1)^i\prod\limits_{j\neq i}J_{z_j}x,\quad&\text{if}\quad i=1,\ldots,r,
\\
(-1)^{i-1}\prod\limits_{j\neq i}J_{z_j}x,\quad&\text{if}\quad i=r+1,\ldots,r+s,
\end{cases}\quad\text{for any}\quad x\in V^{r,s}_+.
\end{equation}
Since the map $A$ has to commute with the product of any even number of representations, we obtain for $i\in\{1,\ldots,r\}$
$$
AJ_{z_i}A^{\tau}=(-1)^{i}\prod_{j\neq i}\wJ_{C(z_j)}AA^{\tau}=\wJ_{C(z_i)}\det C\,\widetilde\Omega^{r,s}(-\det C)\Id=(\det C)^2\wJ_{C(z_j)}=\wJ_{C(z_j)},
$$
and analogously for $i\in\{r+1,\ldots,r+s\}$.
Since $A^{-1}=-\det C A^{\tau}$, in order to satisfy Lemma~\ref{lem:2new} we must define the map $A\colon V^{r,s}_+\to V^{r,s}_-$ by
\begin{equation}\label{eq:gen_rule}
AJ_{z_i}=-\det C\,\wJ_{C(z_i)}A,\quad AJ_{z_i}J_{z_j}=\wJ_{C(z_i)}\wJ_{C(z_j)}A,\ldots.
\end{equation}
We also obtain that 
$$
\la A(x),A(x)\ra_{V^{r,s}_-}=\la x,A^{\tau}A(x)\ra_{V^{r,s}_+}=-\det C\la x,x\ra_{V^{r,s}_+}
$$
for any $x\in V^{r,s}_+$. Thus we see that if $\det C=1$, then the map $A$ became anti-isometry and in the case $\det C=-1$ the map $A$ is an isometry.
In the following lemma we give the precise construction of the isomorphism for the basic cases.

\begin{lemma}\label{lem:basic3-7}
The Lie algebras $\mathcal N_{r,s}(V^{r,s}_+)$ and $\mathcal N_{r,s}(V^{r,s}_-)$ are isomorphic for the set of indices $(r,s)\in \{(3,0),\ (7,0),\ (1,2),\ (3,4),\ (5,2)\ (1,6)\}$.
\end{lemma}
\begin{proof}
{\sc Case $(r,s)=(3,0)$}. Making use of the notations of Lemma~\ref{cor:constr}, we are 
interested in the construction of an isomorphism 
$\Phi=A\oplus C\colon \mathcal{N}_{3,0}(V_{+}^{3,0})\to
\mathcal{N}_{3,0}(V_{-}^{3,0})$,
where $A\colon V_{+}^{3,0}\to V_{-}^{3,0}$ and $C\colon\mathbb{R}^{3,0}\to \mathbb{R}^{3,0}$ with $CC^{\tau}=\Id$, $\det C=-1$. The spaces $V_{+}^{3,0}$ and $V_{-}^{3,0}$ are spanned by orthonormal bases
\begin{equation}\label{eq:basis3012}
\begin{array}{llllll}
&x_1=v,\ \ &x_2=J_{z_1}v,\ \ &x_3=J_{z_2}v,\ \ &x_4=J_{z_1}J_{z_2}v,\qquad\text{and}
\\
&y_1=u,\ \ &y_2=\wJ_{w_1}u,\ \ &y_3=\wJ_{w_2}u,\ \ &y_4=\wJ_{w_1}\wJ_{w_2}u,
\end{array}
\end{equation}
respectively, for some $\la v,v\ra_{V_{+}^{3,0}}=\la u,u\ra_{V_{-}^{3,0}}=1$. 
In order to construct the map $A$ we assume 
$$
A(v)=\sum_{i=1}^{4}a_iy_i, \quad\ \sum_{i=1}^{4} {a_i}^2=\la A(v),A(v)\ra_{V^{3,0}_-}=\la v,A^{\tau}A(v)\ra_{V^{3,0}_+}=1.
$$ 
Then
\begin{align}\label{al:matrixA30}
&A(x_2)=AJ_{z_1}v= \wJ_{w_1}A(v)=
-a_{2}y_{1}+a_{1}y_{2}-a_{4}y_{3}+a_3y_{4},\nonumber
\\
&A(x_3)=AJ_{z_2}v= \wJ_{w_2}A(v)=
-a_{3}y_{1}+a_{4}y_{2}+a_{1}y_{3}-a_{2}y_{4},
\\
&A(x_4)=AJ_{z_1}J_{z_2}=\wJ_{w_1}\wJ_{w_2}A(v)=
-a_{4}y_{1}-a_{3}y_{2}+a_{2}y_{3}+a_{1}y_{4}.\nonumber
\end{align}
Thus, for any orthogonal transformation $C\in
O(3)$, with $\det C=-1$ we can find a map $A\colon V_{+}^{3,0}\to
V_{-}^{3,0}$ such that $\Phi=A\oplus C\colon \mathcal
N_{3,0}(V_{+}^{3,0})\to\mathcal N_{3,0}(V_{-}^{3,0})$ 
is an isomorphism.  
\\

{\sc Case $(r,s)=(1,2)$}. Notice that in this case $V^{1,2}_{+}=V_+\oplus V_+$ and $V^{1,2}_{-}=V_-\oplus V_-$, where $V_{\pm}$ are non-equivalent irreducible modules. We choose basis~\eqref{eq:basis3012} and look for the map $A\colon  V^{1,2}_+\to V^{1,2}_-$ satisfying
$$
A(v)=\sum_{i=1}^{4}a_iy_i, \quad\ \sum_{i=1}^{2} {a_i}^2-\sum_{i=3}^{4} {a_i}^2=\la A(v),A(v)\ra_{V^{1,2}_-}=\la v,A^{\tau}A(v)\ra_{V^{1,2}_+}=-\det C.
$$
We denote $\det C=\epsilon$. Making calculations similar to~\eqref{al:matrixA30}, we find
\begin{equation}\label{eq:matrixA12}
A=\begin{pmatrix}
a_1&\epsilon a_2&-\epsilon a_3&-\epsilon a_4
\\
a_2&-\epsilon a_1&  \epsilon a_4&-\epsilon a_3
\\
a_3&\epsilon a_4&-\epsilon a_1&-\epsilon a_2
\\
a_4&-\epsilon a_3&\epsilon a_2&-\epsilon a_1
\end{pmatrix},\qquad
A^{\tau}=\begin{pmatrix}
a_1&-\epsilon a_2&\epsilon a_3&\epsilon a_4
\\
-a_2&-\epsilon a_1&-\epsilon a_4&\epsilon a_3
\\
-a_3&-\epsilon a_4&-\epsilon a_1&\epsilon a_2
\\
-a_4&\epsilon a_3&-\epsilon a_2&-\epsilon a_1
\end{pmatrix}
\end{equation}
By direct calculations we obtain $A^{\tau}A=-\det C\,\Id$. We conclude that for any transformation $C\in O(1,2)$ we can find a map $A\colon V_{+}^{1,2}\to V_{-}^{1,2}$ such that $\Phi=A\oplus C\colon \mathcal N_{1,2}(V_{+}^{1,2})\to\mathcal N_{1,2}(V_{-}^{1,2})$ is an isomorphism.
\\

{\sc Case $(r,s)=(7,0)$}.
We choose
the involutions
$$
P_{1}=J_{z_1}J_{z_2}J_{z_3}J_{z_4},\quad P_{2}=J_{z_1}J_{z_2}J_{z_5}J_{z_6},\quad
P_{3}=J_{z_1}J_{z_3}J_{z_5}J_{z_7},
$$
acting on the module $V_{+}^{7,0}$. For the module $V_{-}^{7,0}$ we fix the involutions $\wP_{i}$, $i=1,2,3$
changing the basis vectors $z_j$ by $w_j=C(z_j)$, $j=1,\ldots,7$. We take vectors $v\in V_{+}^{7,0}$ and $u\in V_{-}^{7,0}$ 
such that
$$
P_{j}(v)=v,\quad \wP_{j}(u)=u\quad\text{for}\quad j=1,2,3,\quad\text{and}
\quad
\la v,v\ra_{V_{+}^{7,0}}=\la u,u\ra_{V_{-}^{7,0}}=1.
$$
The $8$ common eigenspaces of $P_{j}$ are one dimensional
and are spanned by the vectors $v$ and $J_{z_i}v$. Analogously, one dimensional eigenspaces of $\wP_{j}$ are spanned by $u$ and $\wJ_{w_i}u$.
We set
$$
A(v)=\lambda u,\qquad
AJ_{z_i}v=\lambda \wJ_{w_i}u,\ 
\quad j=1,\ldots,7, \ \ \lambda =\pm 1.
$$
Thus, for each $C\in O(7)$ with $\det C=-1$ and $\lambda=\pm 1$ there is a Lie algebra isomorphism between $\mathcal{N}_{7,0}(V_{+}^{7,0})$ and $\mathcal{N}_{7,0}(V_{-}^{7,0})$.
\\

{\sc Case $(r,s)=(3,4)$}.
By using the notations as in the previous case, 
we choose the mutually commuting isometric involutions
$$
P_1=J_{z_1}J_{z_2}J_{z_4}J_{z_5},\quad
P_2=J_{z_1}J_{z_2}J_{z_6}J_{z_7},\quad
P_3=J_{z_1}J_{z_3}J_{z_5}J_{z_7},
$$
acting on $V_{+}^{3,4}$ and a vector $v\in V_{+}^{3,4}$ such that
$P_i(v)=v$, $i=1,2,3$, $\la v,v\ra_{V_{+}^{3,4}}=1$.
The 8 one dimensional common eigenspaces of involutions
$P_j$, $j=1,2,3$ are spanned by the vectors
$v$ and $J_{z_i}v$ ($i=1,\ldots,7$), respectively. 
For the module $V_{-}^{3,4}$ we also choose the involutions $\wP_{j}$,
$j=1,2,3$ with the same combinations of the generators and
take a positive unit vector $u\in V_{-}^{3,4}$ (if necessary, by
changing the sign of the scalar product, see Remark \ref{sign change})
such that $\wP_{i}(u)=u$ for $i=1,2,3$.


The one dimensional common eigenspaces of the involutions $\wP_i$ are spanned by
$u$ and $\wJ_{w_i}$, $i=1,\ldots,7$. We may set 
$A(v)=\lambda u$. Then 
we have $A^{\tau}(u)= \lambda v$, according to the choice $\la
u,u\ra_{V_{-}^{3,4}}=1$. So that
\[
-\det C\cdot A^{\tau}\,A(v)=-\det C\cdot \lambda^2 \cdot v=
v\quad\Longrightarrow\quad \det C=- 1,
\]
since $\lambda^2=1$.

Thus, there is a Lie algebra isomorphism $A\oplus C$ between
$\mathcal{N}_{3,4}(V_{+}^{3,4})$ and $\mathcal{N}_{3,4}(V_{-}^{3,4})$
with $C\in O(3,4)$ and we can always assume that $\det C=-1$.
\\

{\sc Case $(r,s)\in \{(5,2),\ (1,6)\}$}. First we present arguments of existence of an isomorphism and then we give the constructive proof. In Theorem~\ref{th:51-62} it was shown that 
$\mathcal{N}_{5,2}$ is isomorphic to $\mathcal{N}_{2,5}$, where we implicitly assumed that $\Omega^{5,2}(v)=-v$, which was used in the condition $J_{z_5}v=-\bi(v)$. Thus we actually showed the isomorphism $\mathcal N_{5,2}(V^{5,2}_-)\cong\mathcal N_{2,5}$. The assumption $\Omega^{5,2}(v)=v$ leads to $J_{z_5}v=\bi(v)$, and thus $\mathcal N_{5,2}(V^{5,2}_+)\cong\mathcal N_{2,5}$. We conclude $\mathcal N_{5,2}(V^{5,2}_-)\cong\mathcal N_{5,2}(V^{5,2}_+)$. The same arguments shows $\mathcal N_{1,6}(V^{1,6}_-)\cong\mathcal N_{1,6}(V^{1,6}_+)$.

We propose the constructive proof now. Let $(r,s)=(5,2)$. The mutually commuting isometric involutions and the complementary operators acting on $V^{5,2}_+$
$$
P_1=J_{z_1}J_{z_2}J_{z_3}J_{z_4},\quad P_2=J_{z_1}J_{z_2}J_{z_6}J_{z_7},
\qquad R_1=J_{z_4}J_{z_6}, \quad R_2=J_{z_5}J_{z_6}
$$ 
show that common eigenspaces are 4-dimensional and neutral. It is enough to construct the isomorphism by defining map $A_1\colon E^1\to \wE^1$, since it can be extended to the map $A\colon V^{5,2}_+\to V^{5,2}_-$  in a similar way as in Theorem~\ref{th:general}. By making use of the quaternion structure  
\begin{equation}\label{eq:quat52-1}
\bi=J_{z_1}J_{z_2},\quad \bj=J_{z_1}J_{z_3}J_{z_5}J_{z_7},\quad \bk=J_{z_2}J_{z_3}J_{z_5}J_{z_7},
\end{equation}
we fix the orthonormal basis for $E^1$ as follows
$$
x_1=v,\quad x_2=\bi(v),\quad x_3=\bj(v),\quad x_4=\bk(v),
$$
where $v\in E_1=\cap_{i=1}^{2}E^1_{P_1}$ and $\la v,v\ra_{V^{5,2}_+}=1$. The existence of such a vector $v$ and the orthogonality of the basis is justified by Lemma~\ref{orthogonal}. 
Analogously, we choose the involutions $\wP_i$, $i=1,2$, and quaternion structure, acting on $V^{5,2}_-$ by changing $J_{z_j}$ to $\wJ_{C(z_j)}$. It allows to get the orthonormal basis
$$
y_1=u,\quad y_2=\wbi(u),\quad y_3=\wbj(u),\quad y_4=\wbk(u),
$$
for some $u\in \wE_1=\cap_{i=1}^{2}\wE^1_{\wP_1}$ with $\la u,u\ra_{V^{5,2}_-}=1$. We are looking for the map $A_1\colon  V^{5,2}_+\to V^{5,2}_-$ satisfying $A_1(v)=\sum_{i=1}^{4}a_iy_i$ and 
$$
a_1^2+a_2^2-a_3^2-a_4^2=\la A_1(v),A_1(v)\ra_{V^{5,2}_-}=\la v,A_1^{\tau}A_1(v)\ra_{V^{5,2}_+}=-\det C.
$$
We use the property $A_1\bi=\wbi A_1$ (and the same for $\bj,\bk$) according to~\eqref{eq:gen_rule}, and calculate the matrix for $A_1$ and $A_1^{\tau}A_1$:
\begin{equation}\label{eq:matrix52-16}
A_1=
\begin{pmatrix}
a_1&-a_2&-a_3&-a_4
\\
a_2&a_1&a_4&-a_3
\\
a_3&-a_4&a_1&a_2
\\
a_4&a_3&-a_2&a_1
\end{pmatrix},\qquad
A_1^{\tau}=
\begin{pmatrix}
a_1&a_2&-a_3&-a_4
\\
-a_2&a_1&a_4&-a_3
\\
a_3&-a_4&a_1&-a_2
\\
a_4&a_3&a_2&a_1
\end{pmatrix}
\end{equation}
and 
$$
A_1^{\tau}A_1=
\begin{pmatrix}
a_1^2+a_2^2-a_3^2-a_4^2&0&-2a_1a_3+2a_2a_4&-2a_2a_3-2a_1a_4
\\
0&a_1^2+a_2^2-a_3^2-a_4^2&2a_2a_3+2a_1a_4&-2a_1a_3+2a_2a_4
\\
2a_1a_3-2a_2a_4&-2a_2a_3-2a_1a_4&a_1^2+a_2^2-a_3^2-a_4^2&0
\\
2a_2a_3+2a_1a_4&2a_1a_3-2a_2a_4&0&a_1^2+a_2^2-a_3^2-a_4^2
\end{pmatrix}
$$
In order to satisfy the condition $A_1^{\tau}A_1=-\det C\,\Id$ we need to solve the equations
$$
-2a_1a_3+2a_2a_4=0,\quad -2a_2a_3-2a_1a_4=0.
$$
Thus,  if $\det C=-1$, then we have to choose $a_3=a_4=0$ and if $\det C=1$, then $a_1=a_2=0$. We conclude that
for each $C\in O(5,2)$ there is a Lie algebra isomorphism between $\mathcal{N}_{5,2}(V_{+}^{5,2})$ and $\mathcal{N}_{5,2}(V_{-}^{5,2})$.

Let now $(r,s)=(1,6)$. The arguments are essentially the same as in the previous case. The mutually commuting isometric involutions acting on $V^{1,6}_+$ are $P_1=J_{z_2}J_{z_3}J_{z_4}J_{z_5}$, $P_2=J_{z_2}J_{z_3}J_{z_6}J_{z_7}$. The quaternion structure is $
\bi=J_{z_2}J_{z_3}$, $\bj=J_{z_1}J_{z_3}J_{z_5}J_{z_7}$, $\bk=J_{z_1}J_{z_3}J_{z_5}J_{z_7}$.
We fix orthonormal bases for $E^1$ and $\wE^1$ as in the previous case and construct the map $A_1$ as in~\eqref{eq:matrix52-16}. We come to the same conclusion  that
for each $C\in O(1,6)$ there is a Lie algebra isomorphism between $\mathcal{N}_{1,6}(V_{+}^{1,6})$ and $\mathcal{N}_{1,6}(V_{-}^{1,6})$.
\end{proof}

\begin{cor}\label{cor:compos1}
If $\Phi=A\oplus C\colon \mathcal N_{r,s}(V^{r,s}_+)\to\mathcal N_{r,s}(V^{r,s}_-)$ is a Lie algebra isomorphism, then we always can assume that $\det C=-1$. 
\end{cor}
\begin{proof}
The cases of the indices $(r,s)=\{(3,0),(7,0),(1,6)\}$ contains in the proof of Lemma~\ref{lem:basic3-7}. For the rest of the cases in Lemma~\ref{lem:basic3-7} if $\det C=1$, then we compose the isomorphism $\Phi_1=A\oplus C$ with the isomorphism $\Phi_-=\Id\oplus -\Id$ that gives the isomorphism between $\mathcal N_{r,s}(V_1)$ with $V_1=(V^{r,s}_-,\la.\,,.\ra_{V^{r,s}_-})$ and the Lie algebra $\mathcal N_{r,s}(V_2)$ with  $V_2=(V^{r,s}_-,-\la.\,,.\ra_{V^{r,s}_-})$ by Proposition~\ref{genisom1}. The composition map $\Phi=\Phi_1\circ\Phi_-=A\oplus(-C)$ will have the properties $\det C=-1$ since $r+s$ is odd and $AA^{\tau}=\Id$ due to the change of the sign of the scalar product.
\end{proof}

\begin{lemma}\label{lem:highdimstep3}
If the Lie algebra $\mathcal{N}_{r,s}(V_{+}^{r,s})$ is isomorphic to the Lie algebra $\mathcal{N}_{r,s}(V_{-}^{r,s})$, then 
\begin{itemize}
\item[1.]{the Lie algebras $\mathcal{N}_{r,s+8k}(V_{+}^{r,s+8k})$ and $\mathcal{N}_{r,s+8k}(V_{-}^{r,s+8k})$ 
are isomorphic;}
\item[2.]{the Lie algebras $\mathcal{N}_{r+8k,s}(V_{+}^{r+8k,s})$ and $\mathcal{N}_{r+8k,s}(V_{-}^{r+8k,s})$ 
are isomorphic;}
\item[3.]{the Lie algebras $\mathcal{N}_{r+4k,s+4k}(V_{+}^{r+4k,s+4k})$ and $\mathcal{N}_{r+4k,s+4k}(V_{-}^{r+4k,s+4k})$ 
are isomorphic.}
\end{itemize}
for any $k=1,2,\ldots$.
\end{lemma}
\begin{proof}
The proof is similar to the proof of Theorem~\ref{periodicity} and we only show the first statement, since the others can be obtained analogously. If $\bar \Psi=\bar A\oplus\bar C\colon \mathcal N_{0,8}\to\mathcal N_{0,8}$, with $\bar C\bar C^{\tau}=\Id$ and $\Phi= A\oplus C\colon \mathcal N_{r,s}(V_{+}^{r,s})\to\mathcal N_{r,s}(V_{-}^{r,s})$ with $CC^{\tau}=\Id$, then the map $\hat \Phi=\hat A\oplus\hat C$ with $\hat C=C\oplus\bar C$ and $\hat A=A\hat\otimes\bar A$ given by~\eqref{eq:hatAstep3} is a Lie algebra isomorphism between $\mathcal{N}_{r,s+8k}(V_{+}^{r,s+8k})$ and $\mathcal{N}_{r,s+8k}(V_{-}^{r,s+8k})$.

First of all we observe that an automorphism $\bar \Psi=\bar A\oplus\bar C\colon \mathcal N_{0,8}\to\mathcal N_{0,8}$, with $\bar C\bar C^{\tau}=\Id$ always exists, where we can simply set $\bar C=\Id$ and $\bar A\colon V^{0,8}\to V^{0,8}$ can be any map $A\in\GL(8)$ satisfying $
AJ_{z_j}A^{\tau}=J_{z_j}$,
where $\{z_1,\ldots,z_8\}$ is an orthonormal basis for $\mathbb R^{0,8}$.

We only need to define the map $\hat A\colon V_{+}^{r,s+8k}\to V_{-}^{r,s+8k}$. By making use the notations of Theorem~\ref{periodicity}, we set 
\begin{equation}\label{eq:hatAstep3}
\begin{array}{lll}
&\hat A\prod_{j=1}^{p} \hat J_{z_j}\prod_{\alpha=1}^q\hat J_{\zeta_{\alpha}}=
\\
\\
&=
\begin{cases}
A\prod\limits_{j=1}^{p}J_{z_j}\otimes\bar A(\Omega^{0,8})^{p}\prod\limits_{\alpha=1}^qJ_{\zeta_{\alpha}}
=\prod\limits_{j=1}^{p}\wJ_{\hat C(z_j)}(A^{\tau})^{-1}\otimes\Omega^{8,0}\prod\limits_{\alpha=1}^q\wJ_{\hat C(\zeta_{\alpha})}(\bar A^{\tau})^{-1},
\\  
\qquad\qquad \text{if}\ p=2m+1, \ \  q=2k+1,
\\
\\
A\prod\limits_{j=1}^{p}J_{z_j}  \otimes\bar A(\Omega^{0,8})^{p}\prod\limits_{\alpha=1}^qJ_{\zeta_{\alpha}}
=\prod\limits_{j=1}^{p}\wJ_{\hat C(z_j)}(A^{\tau})^{-1}\otimes\Omega^{8,0}\prod\limits_{\alpha=1}^q\wJ_{\hat C(\zeta_{\alpha})}\bar A,
\\  
\qquad\qquad \text{if}\ p=2m+1, \ \  q=2k,
\\
\\
A\prod\limits_{j=1}^{p}J_{z_j} \otimes\bar A\prod\limits_{\alpha=1}^qJ_{\zeta_{\alpha}}
=\prod\limits_{j=1}^{p}\wJ_{\hat C(z_j)}A\otimes\prod\limits_{\alpha=1}^q\wJ_{\hat C(\zeta_{\alpha})}(\bar A^{\tau})^{-1},
\\  
\qquad\qquad  \text{if}\ p=2m, \ \  q=2k+1,
\\
\\
A\prod\limits_{j=1}^{p}J_{z_j} \otimes\bar A\prod\limits_{\alpha=1}^qJ_{\zeta_{\alpha}}
=\prod\limits_{j=1}^{p}\wJ_{\hat C(z_j)}A\otimes\prod\limits_{\alpha=1}^q\wJ_{\hat C(\zeta_{\alpha})}\bar A,
\\  
\qquad\qquad \text{if}\ p=2m, \ \  q=2k.
\end{cases}
\end{array}
\end{equation}
\end{proof}


\end{document}